\documentclass[a4paper,11pt]{amsart}

\pdfoutput=1
\usepackage[text={400pt,660pt},centering]{geometry}

\usepackage{amsthm, amssymb, amsmath, amsfonts, mathrsfs,relsize}
\usepackage{mathtools}
\usepackage{scalerel} 
\usepackage[colorlinks=true, pdfstartview=FitV, linkcolor=blue, citecolor=blue, urlcolor=blue,pagebackref=false]{hyperref}

\usepackage{esint} 
\usepackage{MnSymbol} 
\usepackage{bbm}

\usepackage{setspace}
\usepackage{tikz}
\usepackage[compat=1.1.0]{tikz-feynman}
\usepackage{caption}
\usepackage{enumitem}
\setlist[itemize]{leftmargin=*}

\usepackage{microtype}

\definecolor{labelkey}{gray}{.8}
\definecolor{refkey}{gray}{.8}

\definecolor{darkgreen}{rgb}{0,0.5,0}
\definecolor{darkblue}{rgb}{0,0,0.7}
\definecolor{darkred}{rgb}{0.9,0.1,0.1}

\newcommand{\la}{\langle}
\newcommand{\ra}{\rangle}

\usepackage[titletoc,title]{appendix}
\usepackage{tikz} 
\usetikzlibrary{shapes,arrows,positioning,calc}

\tikzstyle{result} = [rectangle, 
minimum width=4cm, 
minimum height=1cm,  
text width=4cm, 
align=flush center,
font = \footnotesize, 
draw=black]

\tikzstyle{input} = [rectangle, 
minimum width=3cm, 
minimum height=1cm, 
text width=3cm, 
align=flush center,
font = \footnotesize,
draw=black]

\tikzstyle{arrow} = [thick, ->,>=stealth]

\newtheorem{proposition}{Proposition}
\newtheorem{theorem}[proposition]{Theorem}
\newtheorem{lemma}[proposition]{Lemma}

\theoremstyle{remark}
\newtheorem{remark}[proposition]{Remark}

\theoremstyle{definition}

\newtheorem{hypothesis}[proposition]{Hypothesis}

\numberwithin{equation}{section}
\numberwithin{proposition}{section}

\newcommand{\M}{\mathcal M}
\renewcommand{\le}{\leqslant}

\renewcommand{\leq}{\leqslant}
\renewcommand{\geq}{\geqslant}

\renewcommand{\subset}{\subseteq}

\newcommand{\F}{\mathcal{F}}

\renewcommand{\S}{\mathcal{S}}
\newcommand{\D}{{\mathbf{D}}}
\newcommand{\cc}{\mathbf{c}}

\newcommand{\E}{\mathbb{E}}
\newcommand{\Er}{\mathbb{E}_{\rho}}
\renewcommand{\Pr}{\mathbb{P}_{\rho}}

\renewcommand{\L}{\mathcal{L}}

\newcommand{\N}{\mathbb{N}}

\newcommand{\Ll}{\left}
\newcommand{\Rr}{\right}

\newcommand{\1}{\mathbf{1}}
\newcommand{\R}{\mathbb{R}}

\newcommand{\Z}{\mathcal{Z}}
\newcommand{\Zd}{{\mathbb{Z}^d}}
\newcommand{\Td}{\mathbb{T}^d}
\renewcommand{\P}{\mathbb{P}}
\newcommand{\ov}{\overline}
\renewcommand{\bar}{\overline}

\renewcommand{\tilde}{\widetilde}

\renewcommand{\d}{{\mathrm{d}}}

\renewcommand{\epsilon}{\varepsilon}
\newcommand{\T}{\mathbb{T}}

\newcommand{\X}{\mathcal{X}} 

\newcommand{\Ito}{It\^o\ }

\renewcommand{\r}{\mathbf{r}}

\newcommand{\x}{\boldsymbol x}

\newcommand{\Y}{\mathcal{Y}}

\newcommand{\norm}[1]{\left\Vert{#1}\right\Vert}

\title{Quadratic fluctuations of speed-change Kawasaki dynamics}
\author{Chenlin Gu, Baige Zhou}

\address[Chenlin Gu]{Yau Mathematical Sciences Center, Tsinghua University, Beijing, China}
\email{gclmath@tsinghua.edu.cn}  

\address[Baige Zhou]{Department of Mathematical Sciences, Tsinghua University, Beijing, China}
\email{zbg22@mails.tsinghua.edu.cn, baigezhou1@outlook.com}

\begin{document}

	\begin{abstract}
		
		For the speed-change Kawasaki dynamics, we study the weak convergence of its quadratic field, and derive the equilibrium fluctuation. This extends the result of Gon{\c{c}}alves and Jara [ALEA, Lat. Am. J. Probab. Math. Stat. 16, 605–632 (2019)] to the non-gradient case.

		\bigskip
		
		\medskip
		
		\noindent \textsc{Keywords:} interacting particle system, non-gradient process, equilibrium fluctuation, density field correlation.
		
	\end{abstract}
	\maketitle

    \begin{center}
        \textit{Dedicated to Claudio Landim on the occasion of his 60th birthday}
    \end{center}

	
	%
	%
	%
	%
	%
	%
	%
	%
	
	\section{Introduction}
	
	Our object is to understand the behavior of the equilibrium fluctuations of the non-gradient exclusion process. The linear fluctuations for the non-gradient exclusion process have been proved by Funaki in \cite{fun96}. The higher-order fluctuations have also attracted a lot of attention. For example, the quadratic fluctuations of SSEP have been shown by Gon{\c{c}}alves and Jara in \cite{GJ19}. This work aims to develop a parallel result in the non-gradient exclusion process.

	We briefly recall the necessary notation of the exclusion process and the results of the previous work. Let $\Zd$ be the Euclidean lattice, and we use $\X := \{0,1\}^{\Zd}$ to represent the space of the configuration of particles under the exclusion rule. The element of $\X$ will be denoted by $\eta = \left\{\eta(x): x \in \Zd \right\}$. Here $\eta(x) = 0$ means that the site $x$ is vacant and $\eta(x) = 1$ means that the  site is occupied by one particle. We denote by $x \sim y$ for $x,y \in \Zd$ if $\vert x - y\vert = 1$. Then $\{x,y\}$ is called an unoriented bond. For every $\Lambda \subset \Zd$, we denote by $\Lambda^*$ the bonds in $\Lambda$ that 
	\begin{align}\label{eq.defBond}
		\Lambda^* := \left\{ \{x,y\}: x,y \in \Lambda, x \sim y\right\}.
	\end{align} 
	
	For $x,y \in \Zd$, the exchange operator $\eta^{x,y}$ is defined as 
	\begin{align*}
		\eta^{x,y}(z) := \Ll\{\begin{array}{ll}
			\eta(z), & \qquad z \neq x,y; \\
			\eta(y), & \qquad z = x; \\
			\eta(x), & \qquad z = y.
		\end{array}\Rr.
	\end{align*}
	Especially, when $b = \{x,y\}$ is a bond, we also write $\eta^b$ instead of $\eta^{x,y}$, and define the Kawasaki operator $\pi_b \equiv \pi_{x,y}$ 
	\begin{align*}
		\pi_b F(\eta) := F(\eta^b) - F(\eta).
	\end{align*} 
	For every $x \in \Zd$, the translation operator $\tau_x$ is defined as 
	\begin{align*}
		(\tau_x \eta)(y) := \eta({x+y}),
	\end{align*}
	and given a function $F$ on $\X$, we also define $\tau_x F$ as 
	\begin{align*}
		(\tau_x F)(\eta) := F(\tau_x \eta).
	\end{align*}

	The \emph{speed-change exclusion process} on $\Zd$ is defined through the generator
	\begin{align}\label{eq.Generator}
		\L := \sum_{b \in (\Zd)^*} c_b(\eta) \pi_b =\sum_{i=1}^d\sum_{x \in \mathbb Z^d} c_
		{x,x+e_i}(\eta) \pi_{x,x+e_i},
	\end{align}
	where the family of functions
	\begin{align*}
		\left\{c_b(\eta) \equiv c_{x,y}(\eta) = c_{y,x}(\eta); \ b=\{x,y\} \in \left(\Zd\right)^*\right\},
	\end{align*}
	determines the jump rate of particles on the nearest bonds. This model is also called \emph{the speed-change Kawasaki dynamics} or \emph{the lattice gas} in the literature.
	
	The following usual conditions for the jump rate are assumed in the literature; see \cite{FGW24,FUY96}. They are also the setting throughout the paper without specific explanation.
	\begin{hypothesis} The following conditions are assumed for  $\{c_b\}_{b \in (\Zd)^*}$.
		\begin{enumerate}
			\item Non-degenerate and local: $c_{x,y}(\eta)$ depends only on $\{\eta_z: \vert z - x\vert \leq \r\}$ for some integer $\r > 0$, and is uniformly bounded from above and below $1 \leq c_{x,y}(\eta) \leq \lambda$.
			\item Spatially homogeneous: for all $\{x,y\} \in \left(\Zd\right)^*$, $c_{x,y} = \tau_x c_{0,y-x}$.
			\item Detailed balance under Bernoulli product measures: $c_{x,y}(\eta)$ is independent of $\eta(x)$ and $\eta(y)$.
		\end{enumerate}
	\end{hypothesis}
	This model is known to be of \emph{non-gradient} type, i.e. we cannot find functions $\{h_{i,j}\}_{1\leq i,j\leq d}$ such that
	\[
	c_{0,e_i} (\eta)(\eta({e_i}) - \eta(0) ) = \sum_{j=1}^d \Ll((\tau_{e_j} h_{i,j})(\eta) - h_{i,j}(\eta)\Rr),
	\]
	for general $\{c_{b}\}_{b \in (\Zd)^*}$, with $\{e_i\}_{1 \leq i \leq d}$ the canonical basis of $\Zd$. 
	
	\smallskip
	
    Our non-gradient process is defined on the torus. Let $\Td_N := (\mathbb{Z} / N \mathbb{Z})^d$ be the lattice torus of scale $N$, and we can define all the notation by replacing $\Zd$ with $\Td_N$. We denote by $\X^N := \{0,1\}^{\Td_N}$ the configuration space on $\Td_N$, and define 
	\[
	\eta^N_t := \left\{\eta^N_t(x), \ x \in \Td_N\right\},
	\]
	as the $\X^N$-valued Markov jump process on the torus governed by the generator 
    \begin{equation}\label{eq.L_N}
       	\L_{N}:=N^2\L, 
    \end{equation}
	the counterpart of \eqref{eq.Generator} on $\Td_N$. For equilibrium fluctuations, we fix $\rho\in(0,1)$ and consider $\eta^N$ to have an initial distribution of the Bernoulli product measure with density $\rho$. Since this distribution is reversible with respect to $\L_N$, the distribution of $\eta^N_t$ remains the same for all $t> 0$. We thus define the centered configuration similarly by 
	\begin{align*}
		\bar\eta_t^N:=\eta_t^N-\rho.
	\end{align*}
	The fluctuation $\Y^N_t(\d u)$  around the density $\rho$ is defined as follows: 
	\begin{align*}
		\Y^N_t(\d u): =N^{-\frac{d}{2}}\sum_{x \in \Td_N} \bar\eta^N_t(x)\delta_{x/N}(\d u).
	\end{align*}
	$\Y^N_t$ takes value in the Schwartz distribution on the torus $\S'\left(\Td\right)$. 
	As $N \to \infty$, it converges in the c\`adl\`ag topology $D\left([0,T], \S' \left(\Td\right)\right)$. This was proved by Funaki in \cite[Theorem~1]{fun96}, and we restate it here.
	\begin{proposition}\cite[Theorem~1]{fun96}\label{prop.linear fluctuation}
		The process $(\Y_t^N)_{t \in [0,T]}$ converges weakly as $N\to\infty$ in the space $D\left([0,T], \S' \left(\Td\right)\right)$. The limit is the Ornstein--Uhlenbeck process:
		\begin{equation}\label{first-order fluctuation}
			\d \Y_t=\mathrm{Tr}\left(\mathbf D(\rho)\partial^2 \Y_t\right)\,\d t+\sqrt{\cc(\rho)}\nabla\cdot\d\omega(t),
		\end{equation}
		where $\partial^2\Y=\{\partial_{i}\partial_{j}\Y\}_{1\leq i, j\leq d}$ and $\omega=\{\omega_i(t)\}_{1\leq i\leq d}$ is the $d$-dimensional space-time white noise. Moreover, the initial data $\Y_0$ is a white noise with intensity $\chi(\rho)$.
	\end{proposition}
	
	We explain the definition of $\cc, \D, \xi$ above. We denote by $\F_0$ the space of local functions on $\X$. We first define a quadratic form with respect to the function $F \in (\F_0)^d$ 
	\begin{align}\label{eq.c.form}
		\xi \cdot \cc(\rho; F) \xi = \frac{1}{2} \sum_{\vert x\vert = 1} \E_{\rho}\Ll[{c_{0,x}\Ll(\xi \cdot \Ll\{ x(\eta_x - \eta_0) - \pi_{0,x}\Bigg(\sum_{y \in \Zd} \tau_y F\Bigg)\Rr\}\Rr)^2}\Rr].
	\end{align}
	Then \emph{the effective conductivity} $\cc(\rho)$ is the minimization of $\cc(\rho; F)$
	\begin{align*}
		\xi \cdot \cc(\rho) \xi := \inf_{F \in \F_0^d}\xi \cdot \cc(\rho; F) \xi.
	\end{align*}
	\emph{The diffusion matrix} $\mathbf D : (0,1) \to \R^{d \times d}$ is then given by \emph{the Einstein relation} 
	\begin{align}\label{eq.def.D}
		\D(\rho) := \frac{\cc(\rho)}{2 \chi(\rho)},
	\end{align}
	where $\chi(\rho)$ is \emph{the compressibility}
	\begin{align*}
		\chi(\rho) := \rho (1-\rho).
	\end{align*}

	The object of this paper is the limit of the quadratic field $\left\{Q_t^N(\d u,\d v);\ t\in [0,T]\right\}$ defined as follows 
	\begin{equation}\label{eq.def.Q}
		Q_t^N(\d u,\d v) :=N^{-d}\sum_{\substack{ x,y \in \T_N^d\\ x\neq y}}\bar{\eta}^N_t(x)\bar{\eta}^N_t(y)\delta_{x/N}(\d u)\delta_{y/N}(\d v).
	\end{equation}
	Our main result is the counterpart of \cite[Theorem~2.4]{GJ19} in the non-gradient process. In the statement,  $\partial^2_{1,2}$ denotes differentiation with respect to the first and second variables respectively
		\begin{equation*}
			\partial_{1,2}^2 f :=\left\{\partial_{x_i}\partial_{x_j} f+\partial_{y_i}\partial_{y_j} f \right\}_{1\leq i,j\leq d}.
		\end{equation*}

	\begin{theorem}\label{thm:quadratic fluctuation}
		Let $(\mathcal{M}_t)_{t\in [0,T]}$ be the martingale process defined by
		\begin{equation}\label{eq.W}
			\mathcal{M}_t(f) :=\int_0^t\int_{\T^d}\left\{\Y_s\left(\nabla_{1}f(x,\cdot)\right)+\Y_s\left(\nabla_{2}f(\cdot,x)\right)\right\}\cdot\sqrt{\cc(\rho)}\,\d{\omega}(s,x).
		\end{equation}
		for every $f\in C^{\infty}\left(\T^{2d}\right)$, with $\Y$ and $\omega$ given in Proposition~\ref{prop.linear fluctuation}.  The sequence $(Q_t^N)_{t\in [0,T]}$ converges weakly as $N\to\infty$ in the space $D\left([0,T],\S'\left(\T^{2d}\right)\right)$, with a limit $(Q_t)_{t\in [0,T]}$ satisfying
		\begin{equation*}
			\d Q_t=\mathrm{Tr}\left(\D(\rho)\partial^2_{1,2} Q_t\right) \,\d t+\d \mathcal{M}_t,
		\end{equation*}
		and $Q_0$ as a white noise with intensity $\chi(\rho)^2$. 
	\end{theorem}
	\begin{remark}\label{remark.probability space}
		Let us make more comments on the weak convergence mentioned above. The weak convergence usually does not specify the probability space. However, as the definition \eqref{eq.W} indicates, the process $(\mathcal{M}_t)_{t \in [0,T]}$ and $(\Y_t)_{t \in [0,T]}$ live in the same probability space. Therefore, Theorem~\ref{thm:quadratic fluctuation} implies the following weak convergence in $D\left([0,T], \S' \left(\Td\right)\times\S'\left(\T^{2d}\right)\right)$
		\begin{align*}
			(\Y^N_t, Q^N_t)_{t \in [0,T]} \xRightarrow{N \to \infty} (\Y_t, Q_t)_{t \in [0,T]}.
		\end{align*} 
		This convention is kept throughout the paper, and will also be recalled from time to time for some other related processes.
	\end{remark}

	We mention the organization of the paper and highlight the novelty. In Section~\ref{section.pre}, we introduce the notation and the basic tools. By Dynkin's formula, we get a decomposition for our target process $(Q_t^N)_{t\in [0,T]}$, and we treat the martingale term and the drift term separately in Section~\ref{section.martingale} and Section~\ref{section.drift}. The main argument is the tightness and the characterization of the limit. Concerning the non-gradient model, an important step is to correct the process $Q^N$ as 
	\[
	\mathcal{Q}^N :=Q^N+Z^N,
	\]
	with $Z^N$ defined in \eqref{eq.corrector}. Afterwards, we develop the \emph{replacement argument} to analyze the modified field $\mathcal Q^N$. The diffusion matrix $\D$ is closely related to the homogenization theory, and the proof in this paper makes use of some recent quantitative results from \cite{FGW24, gu2025relaxation}; see Section~\ref{subsec.corrector} for details. In Section~\ref{section.characterization}, we get a characterization for the quadratic field following \cite[Theorem~3.9]{GJ19}.

	\section{Preliminary}\label{section.pre}
	\subsection{Notation}
	\subsubsection{Probability space}
	Recall that $\Td_N := (\mathbb{Z} / N \mathbb{Z})^d$ stands for the lattice torus of scale $N$.	For every $\Lambda \subset \Td_N$, we denote by $\F_{\Lambda}$ the $\sigma$-algebra generated by $\left\{\eta^N(x): x \in \Lambda\right\}$ and write $\F$ as a shorthand notation of $\F_{\Td_N}$. 
	
	Given $\rho \in (0,1)$ as the density of particles, let $\Pr = \operatorname{Bernoulli}(\rho)^{\otimes \Td_N}$ stand for the Bernoulli product measure on $\X^N$. The triplet $\left(\X^N, \F, \Pr\right)$ is the probability space in this paper. For the expectation under $\Pr$, we use the notation $\Er[\ \cdot\ ]$. 
	
	Since the law of $(\eta^N_t)_{t \geq 0}$ is invariant under $\Pr$, we sometimes omit the subscript when calculating the expectation under $\Pr$. Meanwhile, the constant $T > 0$ is fixed throughout the paper to indicate the interval of time, and we abuse $Q \equiv (Q_t)_{t \in [0,T]}$ in some statement. The meaning will become clear in the context, and these conventions apply to all the processes.

	\subsubsection{Geometry}
	We denote the hypercube of side length $L$ by 
	\begin{equation*}
		\Lambda_L := \left(-\frac{L}{2}, \frac{L}{2}\right)^d \cap \T^d_N.
	\end{equation*}
	For simplicity, we assume that $L$ is an odd integer and is a factor of $N$, which allows us to divide  $\Td_N$ into a disjoint union  
	\begin{equation*}
		\Td_N=\bigsqcup_{z\in \mathcal Z_L}\Lambda_L^z,
	\end{equation*}
	with the box of length $L$ centered at $z$
	\begin{equation*}
		\Lambda_L^z := z+\Lambda_L,
	\end{equation*}
	and the set of $z$
	\begin{equation*}
		\mathcal Z_L:=L\mathbb Z^d\cap \Td_N.
	\end{equation*}
	We use $z$ to denote the block index and $x$ to denote microscopic lattice sites. For every $x\in \Td_N$, there exists a unique  $z$ such that $x\in \Lambda_L^z$, and we denote it by $z(x)$. We define $\partial \Lambda$ as the boundary of $\Lambda$: 
	\[
	\partial \Lambda:=\{x\in \Lambda: \exists y\notin \Lambda, x\sim y\},
	\]
	and $\Lambda^{-}$ the interior of $\Lambda$:
	\[
	\Lambda^-:=\Lambda\backslash\partial\Lambda.
	\]
	Recall the set of bonds in \eqref{eq.defBond} and we define its enlarged version by 
	\begin{equation*}
		\bar{\Lambda^*}:=\left\{\{x,y\}:x\in \Lambda, y=x+e_i, 1\leq i\leq d\right\}.
	\end{equation*}
	This notation provides a better structure for bonds:
	\begin{equation*}
		\left(\Td_N\right)^*=\bigsqcup_{z\in \mathcal Z_L}\bar{\left(\Lambda_L^z\right)^*}.
	\end{equation*}
	We also use the following convention for the summation 
	\begin{align*}
		\sum_x f(x) \equiv \sum_{x\in \Td_N} f(x).
	\end{align*}
	For every $\xi\in\R^d$, we define the affine function
	\begin{equation*}
		\ell_{\xi}=\sum_x\left(\xi\cdot x\right)\eta^N(x).
	\end{equation*}
    Please note that for every $\xi\in \R^d$ and $b\in \left(\Td_N\right)^*$, $\pi_b \ell_{\xi}$ has no ambiguity, although the summation of $x$ is on the torus.  
	\subsubsection{Test function and discrete derivative.}
	Throughout the paper, $f$ is called a \textit{test function} if and only if $f\in C^\infty\left(\T^{2d}\right)$ and is symmetric in the sense 
	\begin{equation*}
		\forall u,v\in\Td, \qquad f(u,v)=f(v,u).
	\end{equation*}
	We define the discrete derivative notation
	\begin{equation*}
		\nabla_{1,i}^{N}f\left(\frac{x}{N},\frac{y}{N}\right):=\frac{f\left(\frac{x+e_i}{N},\frac{y}{N}\right)-f\left(\frac{x}{N},\frac{y}{N}\right)}{\frac{1}{N}},
	\end{equation*}
	and
	\begin{equation*}
		\nabla_{2,i}^{N}f\left(\frac{x}{N},\frac{y}{N}\right):=\frac{f\left(\frac{x}{N},\frac{y+e_i}{N}\right)-f\left(\frac{x}{N},\frac{y}{N}\right)}{\frac{1}{N}}.
	\end{equation*}
	We use the subscript $1$ and $2$ to indicate differentiation with respect to the first and second variables respectively. The following identity is then valid for the test function and all $1\leq i\leq d$ thanks to the symmetry 
	\begin{equation*}
		\nabla_{1,i}^{N}f\left(\frac{x}{N},\frac{y}{N}\right)=\nabla_{2,i}^{N}f\left(\frac{y}{N},\frac{x}{N}\right).
	\end{equation*}
	Then the discrete gradients for the first and second variables are defined respectively as 
	\begin{align*}
		\nabla_1^{N}f(x/N,y/N) &:=\left\{\nabla_{1,i}^{N}f(x/N,y/N)\right\}_{1\leq i\leq d}, \\
		\nabla_2^{N}f(x/N,y/N) &:=\left\{\nabla_{2,i}^{N}f(x/N,y/N)\right\}_{1\leq i\leq d}.
	\end{align*}
	\subsection{Criteria of tightness}
	We recall Mitoma's criterion for the tightness of distribution-valued processes.  
	\begin{proposition}[\cite{Mitoma1983}, Mitoma's criterion]
		The sequence of processes \(\left\{ (X_t^N)_{t\in [0,T]} \right\}_{N \in \mathbb{N}}\) is tight in  
		\( D\left([0, T],\S'\left(\mathbb{T}^{2d}\right)\right) \) if and only if \(\left\{ (X_t^N(f))_{t\in [0,T]} \right\}_{N \in \mathbb{N}}\) is tight in  	\( D([0, T], \mathbb{R}) \) for every \( f \in C^\infty\left(\mathbb{T}^{2d}\right) \). Moreover, if every limit point of \(\left\{ (X_t^N(f))_{t\in [0,T]} \right\}_{N \in \mathbb{N}}\) is supported on continuous, real-valued trajectories for every \( f \in C^\infty\left(\mathbb{T}^{2d}\right) \),  
		then every limit point of \(\left\{ (X_t^N)_{t\in [0,T]} \right\}_{N \in \mathbb{N}}\) is supported on \( C\left([0, T],\S'\left(\mathbb{T}^{2d}\right)\right) \) .
	\end{proposition}
    To make the limit for the sequence of real-valued processes lie in the Skorokhod space $D([0, T], \mathbb{R})$, we usually use the moduli 
    \[
    w'(\varphi, r) := \inf_{\{t_i\}_{i=0, \dots, k}} \max_{0 \le i \le k-1} \sup_{s, t \in [t_i, t_{i+1})}\vert \varphi(t)-\varphi(s)\vert, \qquad r > 0,
    \]
    where \(\{t_i\}_{i=0, \dots, k}\) runs over all partitions \(0 = t_0 < \dots < t_k = T\) for \(k \in \mathbb{N}_0\), such that \(\min_i |t_{i+1} - t_i| > r\).  
    One can show that \(\varphi \in D([0, T], \mathbb R)\) if and only if \(\lim_{r \to 0} w'(\varphi, r) = 0\); see \cite[Theorem~3.21, Chapter~VI]{Limit_theorems_for_stochastic_processes} and \cite[Theorem~12.3]{Convergence_of_probability_measures}.

    In order to obtain the limit process with continuous trajectory, one approach is to use a stronger topology \emph{the modulus of continuity} $\omega(\varphi,r)$: 
    \begin{equation*}
		\omega(\varphi,r):=\sup_{s,t\in[0,T], \vert t-s\vert\leq r}\vert\varphi(t)-\varphi(s)\vert. 
	\end{equation*}
    We will use the following criteria, which is called \emph{the C-tightness}. One can find it in the references \cite[Theorem~1.3, Remarks~1.4, 1.5]{kipnis1998scaling} and  \cite[Theorem~3.21, Chapter~VI]{Limit_theorems_for_stochastic_processes}. 
	\begin{proposition}\label{Prop.Tightness.real-valued}
		A process \(\left\{ (X_t^N)_{t\in [0,T]} \right\}_{N \in \mathbb{N}}\) is tight on the space $D\left([0,T],\R\right)$ and admits the limit in $C\left([0,T],\R\right)$ if

		\begin{itemize}[leftmargin=2em]
			\item [(i)] For every $t\in [0,T]$, the sequence $\left\{X^N_t\right\}_{N\in \mathbb{N}}$ is tight in $\R$;
            
			\item [(ii)] For all $\varepsilon>0$, we have 
			\[
            \inf_{r>0}\limsup_{N\to\infty}\P\left[\omega\left(X^N,r\right)\geq \varepsilon\right]=0.
            \]
		\end{itemize}
	\end{proposition}
    In the next sections, we combine these two criteria to conclude the tightness for sequences of martingale and drift separately. \subsection{Corrector}\label{subsec.corrector}
	Since we treat the non-gradient process, the corrector method is needed to close the equation. This corrector is introduced to eliminate the non-gradient part of the drift in the Kawasaki dynamics. The corrector in this paper is $\phi_L=\left\{\phi_{L,e_i}\right\}_{1\leq i\leq d}$ introduced in \cite{FGW24} (see also \cite[(4.10)]{gu2025relaxation}), which is the unique minimizer defined below
	\begin{equation*}
		\phi_{L,e_i} := \arg \min_{\substack{\phi \in \F_0(\Lambda_L^-)\\ \Er[\phi] = 0}} \Ll\{  \sum_{b\in\ov{\Lambda_L^*}} \Er\Ll[c_b(\pi_b (\ell_{e_i} + \phi))^2\Rr] \Rr\}.
	\end{equation*}
	Roughly, we have $\phi_{L,e_i} \simeq e_i \cdot F_L$, where $F_L$ is the minimizer of \eqref{eq.c.form} in $\F_0(\Lambda_L)$. Then the centered flux $\mathbf{g}_{L,e_i,b}$ is defined as 
	\begin{equation}\label{eq.flux}
		\mathbf{g}_{L,e_i,b}:= c_b\pi_b\left(\ell_{e_i}+\phi_{L,e_i}\right)-\pi_b\ell_{\mathbf D(\rho)e_i}.
	\end{equation}
	It measures the error of the replacement argument, as we expect that 
	\begin{align}\label{eq.replace}
		c_b\pi_b\left(\ell_{e_i}+\phi_{L,e_i}\right) \simeq \pi_b\ell_{\mathbf D(\rho)e_i}.
	\end{align}
	We then define their version after shift as
	\begin{align*}
		\phi^z_{L,e_i}  := \tau_z \phi_{L,e_i}, \qquad \mathbf{g}^z_{L,e_i,b} := \tau_z \mathbf{g}_{L,e_i,b}.
	\end{align*}

	The following estimates about the corrector $\phi_L$ and $\mathbf{g}_{L}$ are developed in \cite{FGW24} and \cite{gu2025relaxation}. We highlight that, (4) is a quantitative version of the replacement argument \eqref{eq.replace}.
	\begin{lemma}\label{lem.Homogenization}
		\begin{enumerate}
			\item\label{lem.Correctorlocal} \cite[Proposition 4.3 (1)]{gu2025relaxation}:
			The local corrector $\phi_L^z$ is a local function and  $\F_{(\Lambda_L^z)^-}$-measurable.
			\item\label{lem.CorrectorLp} \cite[Lemma 4.4]{FGW24}: 
			There exists a finite positive constant $C(d,\lambda)$ such that the $L^\infty$ and $L^{2}$ norms for the corrector $\phi_L^z$ satisfy the following estimate:
			\begin{align*}
				&\norm{\phi_{L,e_i}^z}_{L^\infty}\leq CL^{d+2}\log L,\qquad 1\leq i\leq d,\ z\in \Z_L,\\
				&\norm{\phi_{L,e_i}^z}_{L^2}^2\leq CL^{d+2},\ \ \ \ \ \ \ \qquad 1\leq i\leq d,\ z\in \Z_L.
			\end{align*}
			\item\label{lem.CorrectorH1} \cite[Proposition 4.3 (3)]{gu2025relaxation}:
			There exists a finite positive constant $C(\lambda,\rho)$ such that the corrector satisfies:
			\begin{equation*}
				\sum_{b\in{(\Lambda_L^z)^*}}\E_{\rho}\left[c_b(\eta)\left(\pi_b\phi_{L,e_i}^z\right)^2\right]\leq CL^d,\quad 1\leq i\leq d,\ z\in \mathcal Z_L.
			\end{equation*}
			\item\label{lem.FluxReplacement}\cite[Proposition 4.3 (4)]{gu2025relaxation}
			There exists an exponent $\alpha(\lambda,\r) > 0$ and a positive constant $C(\lambda,\r) < \infty$, such that for every $G : \X \to \R$ and $1\leq i\leq d$, we have 
			\begin{align*}
				\Ll\vert L^{-d}\sum_{b \in \bar{\left(\Lambda_L^z\right)^*} }\E_\rho\left[(\pi_b G) \mathbf g_{L,e_i, b}^z\right] \Rr\vert \leq C L^{-\alpha}\left(L^{-d}\sum_{b \in \bar{\left(\Lambda_L^z\right)^*} }\E_\rho\left[(\pi_b G)^2\right]\right)^\frac{1}{2}.
			\end{align*}
		\end{enumerate}
	\end{lemma}
	
	\subsection{Dynkin's formula}
	Inspired by \cite{fun96}, we define the corrected process $Z^N(f)$ for every test function $f$,  
    \begin{equation}\label{eq.corrector}
        Z^N(f) := 2N^{-1-d}\sum_{\substack{ z\in \mathcal Z_L\\ }}\sum_{y:y\notin \Lambda_L^z} \bar\eta^N(y)\phi_L^z\left(\eta^N\right)\cdot{\nabla_{1}^{N}f\left(\frac{z}{N},\frac{y}{N}\right)}.
    \end{equation}
	The length $L \equiv L(N)$ depends on $N$ and satisfies $1 \ll L(N) \ll N$. We assume the usual condition in the following paragraphs
	\begin{hypothesis}\label{hyp}
		$L(N) \xrightarrow{N \to \infty} +\infty, \frac{L^{100}(N)}{N} \xrightarrow{N \to \infty} 0$. 
	\end{hypothesis}
    The first observation is that, the term $Z^N(f)$ is indeed small.
    \begin{lemma}\label{lemma.Z}
        Under the Hypothesis~\ref{hyp}, for every $t >0$, we have $Z^N_t(f) \xrightarrow[N \to \infty]{L^2} 0$. 
    \end{lemma}
    \begin{proof}
        With this Hypothesis, the $L^2$-norm of $Z^N(f)$ is given by 
    \begin{align*}
        \E_\rho\left[\left(Z^N(f)\right)^2\right]&= 4N^{-2-2d}\E_\rho\left[\left(\sum_{\substack{ z\in \mathcal Z_L}}\sum_{y:y\notin \Lambda_L^z} \bar\eta^N(y)\phi_L^z\left(\eta^N\right)\cdot{\nabla_{1}^{N}f\left(\frac{z}{N},\frac{y}{N}\right)}\right)^2\right]\\
        &\leq C_fN^{-2-2d}\sum_{\substack{ z\in \mathcal Z_L}}\sum_{y:y\notin \Lambda_L^z} \E_\rho\left[\bar\eta^N(y)^2\right]\E_\rho\left[\left|\phi_L^z\left(\eta^N\right)\right|^2\right]\notag\\
        &\quad +C_fN^{-2-2d}\sum_{\substack{ z_1,z_2\in \mathcal Z_L\\ z_1\neq z_2}}\sum_{y_1\in \Lambda_L^{z_2}}\sum_{y_2\in \Lambda_L^{z_1}} \E_\rho\left[\left|\bar\eta^N(y_1)\phi_L^{z_2}\right|\right]\E_\rho\left[\left|\bar\eta^N(y_2)\phi_L^{z_1}\right|\right]\notag\\
        &\leq C_fN^{-2}\norm{\phi_{L}}_{\infty}^2\notag,
    \end{align*}
    which shows that the corrector is indeed a small term in $L^2$.
    \end{proof}
    
      We then denote by $\mathcal Q^N$ the modified quadratic field
    \begin{equation}\label{eq.correction}
        \mathcal{Q}^N(f) :=Q^N(f)+Z^N(f). 
    \end{equation}
	By Dynkin's formula, for every test function $f$, we have the decomposition
	\begin{equation}\label{eq.Dynkin}
		\mathcal{Q}_t^N(f)=\mathcal{Q}_0^N(f)+\mathcal{A}_t^N(f)+\M_t^N(f),
	\end{equation}
	where $\mathcal A^N \equiv (\mathcal A^N_t)_{t \geq 0}$ is the drift term
	\begin{equation*}
		\mathcal A_t^N(f) :=\int_0^t \L_N \mathcal Q_s^N(f)\,\d s.
	\end{equation*}
	Here $\L_N$ defined in \eqref{eq.L_N} is the generator of the speed-change Kawasaki dynamics. The martingale term $\M^N \equiv (\M^N_t)_{t \geq 0}$ is defined as
	\begin{equation*}
		\M_t^N(f) :=\mathcal{Q}_t^N(f)-\mathcal{Q}_0^N(f)-\mathcal{A}_t^N(f).
	\end{equation*}
	The rest of the paper is devoted to the study of $\left(\M^N, \mathcal A^N, Z^N\right)$, which allows us to understand the limit behavior of the quadratic field.

	\section{Martingale}\label{section.martingale}
	We study the martingale term $\M^N$ in this section, which states the following result.
	\begin{proposition}[Limit of martingale]\label{prop.Characterization for the convergence of the martingale term}
		For $f\in C^\infty\left(\mathbb{T}^{2d}\right)$, the sequence of martingales $\{\mathcal M^N(f)\}_{N \in \N}$ admits a subsequential limit $\mathcal M(f)$ as $N\to\infty$ in $D([0, T], \mathbb{R})$. Every such limit point $\mathcal M(f)$ is a martingale in $C\left([0,T],\R\right)$ and satisfies
		\begin{multline*}
			\left\langle{\mathcal M}(f)\right\rangle_t= \int_0^t\int_{\T^d}\left\{\Y_s\left(\nabla_{1}f(x,\cdot)\right)+\Y_s\left(\nabla_{2}f(\cdot,x)\right)\right\}\\\cdot \,\cc(\rho)\left\{\Y_s\left(\nabla_{1}f(x,\cdot)\right)+\Y_s\left(\nabla_{2}f(\cdot,x)\right)\right\}\,\d x\,\d s.
		\end{multline*}
	\end{proposition}
	
	Recall that for every test function $f$, the martingale term ${\mathcal M}^N(f)$ has the form
	\begin{equation*}
		\mathcal M_t^N(f)=\mathcal{Q}_t^N(f)- \mathcal{Q}_0^N(f)-\int_0^t \mathcal{L}_{N}\mathcal{Q}_s^N(f)\,\d s.
	\end{equation*}
	Its associated quadratic variation can be expressed as
	\begin{equation*}
		\left\langle{\mathcal M}^N(f)\right\rangle_t=\int_0^t \mathcal{B}_s^N(f)\,\d s,
	\end{equation*}
	where $\mathcal{B}_s^N(f)$ is \emph{the carr\'e du champ operator}
	\begin{align}\label{eq.quadratic variation}
		\mathcal{B}_s^N(f)&=\L_{N}\left(\mathcal{Q}_s^N(f)^2\right)-2 \mathcal{Q}_s^N(f)\L_{N}\mathcal{Q}_s^N(f)\\
		&=N^{2}\sum_{i=1}^d\sum_{x}c_{x,x+e_i}\left(\eta^N_s\right)\left(\pi_{x,x+e_i}\mathcal{Q}_s^N(f)\right)^2\notag.
	\end{align}
	The tightness will be established in Section~\ref{sec.M.tight}, and the characterization of the limit is given in Section~\ref{sec.M.limit}. Moment estimates about $\mathcal{B}^N(f)$ are presented in Section~\ref{sec.M.moment} and Appendix.

	\subsection{Tightness}\label{sec.M.tight}
	\begin{proposition}\label{prop.tightness.quadratic.M}
		For every test function $f$, the sequence of martingales and the corresponding quadratic variation  
        \[
        \left\{\mathcal M^N_t(f), t\in [0,T]\right\}_{N \in \N}, \quad \left\{\la\mathcal M^N(f)\ra_t, t\in [0,T]\right\}_{N \in \N}
        \]
         are tight in $D([0, T], \mathbb{R})$ and and all limit points are concentrated on $C\left([0,T],\R\right)$. If we have the following weak convergence in $D\left([0,T], \S' \left(\T^{2d}\right)\right)$ along a subsequence $N_k\to\infty$
         \begin{align}\label{eq.M.Nk.convergence}
			(\M^{N_k}_t)_{t \in [0,T]} \xRightarrow{N_k \to \infty} (\M_t)_{t \in [0,T]},
		\end{align}
         then the limit process $\mathcal M_t$ is a continuous martingale and the following weak joint convergence holds 
         \begin{equation}\label{eq.M.Nk.joint}
         \Ll(\mathcal M^{N_k}_t,\la\mathcal M^{N_k}\ra_t\Rr)_{t\in [0,T]}  \xRightarrow{N_k \to \infty} \Ll(\M_t,\la\mathcal M\ra_t\Rr)_{t \in [0,T]}.
         \end{equation}
	\end{proposition}
	The tightness is reduced to the moment estimates below, whose proof is postponed to the next section.
	\begin{proposition}\label{prop.B.bound.L1.L2}
		For every test function $f$ and every $t > 0$, the carr\'e du champ $\mathcal B^N(f)$ is bounded in $L^1$ and $L^2$ uniformly with respect to $N$.
	\end{proposition}
	We use Proposition~\ref{prop.B.bound.L1.L2} to give the tightness for martingale $\M^N(f)$.
	
	\begin{proof}[Proof of Proposition~\ref{prop.tightness.quadratic.M}] We need to verify the C-tightness criterion stated in Proposition~\ref{Prop.Tightness.real-valued}.
		As to condition (i), for every $t\in[0,T]$,
		\begin{align*}
			\E_{\rho}\left[\left\vert \mathcal M_t^N(f)\right\vert^2\right]= \E_{\rho}\left[\left\la \mathcal M^N(f)\right\ra_t\right]=\int_0^t \E_{\rho}\left[ \mathcal B_s^N(f)\right]\d s\leq C_ft,
		\end{align*}
		where the last step follows from the stationarity and the $L^1$-boundedness of $\mathcal B^N$ in Proposition~\ref{prop.B.bound.L1.L2}. This gives the tightness for $\M _t^N(f)$ and $\la\M^N(f) \ra_t$ for fixed time $t\in [0,T]$.

		For the condition (ii), by Chebyshev's inequality, we have
		\begin{equation}\label{ineq.Chebyshev}
			\P_{\rho}\left[\omega\left(\M^N(f),r\right)\geq \varepsilon\right]\leq\frac{1}{\varepsilon^4}\E_{\rho}\left[\omega\left(\M^N(f),r\right)^4\right].
		\end{equation}
		We calculate the fourth moment
		\begin{equation*}
			\omega\left(\M^N(f),r\right)^4=\sup_{\substack{ s,t\in[0,T]\\\vert t-s\vert\leq r}}\left\vert \M_t^N(f)-\M_s^N(f)\right\vert^4\\
			\leq C \sup_{\substack{t\in[0,T], t-s\leq r \\s\in\{0,r,\cdots,\lfloor \frac{T}{r} \rfloor r\}}}\left\vert \M_t^N(f)-\M_s^N(f)\right\vert^4.
		\end{equation*}
		Therefore, we have
		\begin{align}\label{eq.martingale.modulus.Quadratic}
			\E_{\rho}\left[\omega\left(\M^N(f),r\right)^4\right]&\leq C\E_{\rho}\Bigg[ \sup_{\substack{t\in[0,T], t-s\leq r \\s\in\{0,r,\cdots,\lfloor \frac{T}{r} \rfloor r\}}}\left\vert \M_t^N(f)-\M_s^N(f)\right\vert^4\Bigg]\\ 
			&\leq C\sum_{s\in\{0,r,\cdots,\lfloor \frac{T}{r} \rfloor r\}}\E_{\rho}\left[\sup_{\substack{t\in[0,T], t-s\leq r }}\left\vert \M_t^N(f)-\M_s^N(f)\right\vert^4\right]\notag.
		\end{align}
		By Doob's inequality and Burkholder--Davis--Gundy's inequality, we obtain
		\begin{align}\label{ineq.Quadratic.Doob.BDG}
			\E_{\rho}\left[\sup_{\substack{t\in[0,T], t-s\leq r }}\left\vert \M_t^N(f)-\M_s^N(f)\right\vert^4\right]&\leq C\E_{\rho}\left[\left\vert \M_{s+r}^N(f)-\M_s^N(f)\right\vert^4\right]
			\\&\leq C \E_{\rho}\left[\left(\left\la \M^N(f)\right\ra_{s+r}-\left\la \M^N(f)\right\ra_s\right)^2\right]\notag\\
			&\leq C_{f} r^2\notag,
		\end{align}
		where the last step follows from $L^2$-moment of $\mathcal B^N(f)$
		\begin{align}\label{eq.second moment.quadratic variation}
			\E_{\rho}\left[\left(\left\la \M^N(f)\right\ra_{s+r}-\left\la \M^N(f)\right\ra_s\right)^2\right]&=\E_{\rho}\left[\left(\int_s^{s+r} \mathcal B_\kappa^N(f)\,\d\kappa\right)^2\right]\\&\leq r\E_{\rho}\left[\int_s^{s+r} \mathcal B_\kappa^N(f)^2\d\kappa\right]\notag\\
			&=r\int_s^{s+r}\E_{\rho}\left[ \mathcal B_\kappa^N(f)^2\right]\,\d\kappa\notag\\
			&\leq C_{f}r^2\notag.
		\end{align}
		Combining  \eqref{ineq.Chebyshev}-\eqref{ineq.Quadratic.Doob.BDG}, we have
		\begin{align*}
			\P_{\rho}\left[\omega\left(\M^N(f),r\right)\geq \varepsilon\right]
			&\leq \frac{C}{\varepsilon^4}\sum_{s\in\{0,r,\cdots,\lfloor \frac{T}{r} \rfloor r\}}\E_{\rho}\left[\sup_{\substack{t\in[0,T], t-s\leq r }}\left\vert \M_t^N(f)-\M_s^N(f)\right\vert^4\right]\\
			&\leq \frac{C_{f}}{\varepsilon^4}\sum_{s\in\{0,r,\cdots,\lfloor \frac{T}{r} \rfloor  r\}}r^2\\
			&\leq \frac{C_{f}}{\varepsilon^4}Tr.
		\end{align*}
		This concludes the tightness for the sequence of martingale $\{\mathcal M_t^N(f), t\in [0,T]\}_{N\in \mathbb N}$ in $D([0, T], \mathbb{R})$. Along any convergent subsequence, the limit points are supported on $C\left([0,T],\R\right)$.

        As to the modulus of $\la \mathcal M^N(f)\ra_t$, we have
        \begin{align}\label{eq.quadratic variation.modulus}
            \P_{\rho}\left[\omega\left(\la\M^N(f)\ra,r\right)\geq \varepsilon\right]&\leq\frac{1}{\varepsilon^2}\E_{\rho}\left[\omega\left(\la\M^N(f)\ra,r\right)^2\right]\\
            &\leq \frac{1}{\varepsilon^2}\E_{\rho}\left[\sup_{t\in [0,T]}\left(\la\M^N(f)\ra_{(t+r)\wedge T}-\la\M^N(f)\ra_t\right)^2\right]\notag\\
            &\leq \frac{C}{\epsilon^2}\sum_{t\in\{0,r,\cdots,\lfloor \frac{T}{r} \rfloor r\}}\E_{\rho}\left[\left(\la\M^N(f)\ra_{(t+r)\wedge T}-\la\M^N(f)\ra_t\right)^2\right]\notag.
        \end{align}
        Here from the first step to the second step, we use that the quadratic variation is an increasing process. Combining \eqref{eq.second moment.quadratic variation} and \eqref{eq.quadratic variation.modulus}, we have
        \begin{align*}
			\P_{\rho}\left[\omega\left(\la\M^N(f)\ra,r\right)\geq \varepsilon\right]
			&\leq \frac{C}{\epsilon^2}\sum_{t\in\{0,r,\cdots,\lfloor \frac{T}{r} \rfloor r\}}\E_{\rho}\left[\left(\la\M^N(f)\ra_{(t+r)\wedge T}-\la\M^N(f)\ra_t\right)^2\right]\\
			&\leq \frac{C_{f}}{\varepsilon^2}\sum_{s\in\{0,r,\cdots,\lfloor \frac{T}{r} \rfloor  r\}}r^2\\
			&\leq \frac{C_{f}}{\varepsilon^2}Tr.
		\end{align*}
        This concludes the tightness for the sequence of quadratic variation $\{\la\mathcal M^N(f)\ra_t, t\in [0,T]\}_{N\in \mathbb N}$ in $D([0, T], \mathbb{R})$ and every limit point is supported on $C\left([0,T],\R\right)$.
 
        Since for every test function $f$, $\M_t^N(f)$ is uniformly bounded in $L^2$, we have uniform integrability for $t\in [0,T]$ and $N\in \mathbb N$. By \cite[Proposition~1.12, Chapter IX]{Limit_theorems_for_stochastic_processes}, this implies that along any subsequence $N_k$ such that the convergence in \eqref{eq.M.Nk.convergence} holds, the limit $\M_t(f)$ is a continuous martingale with respect to the natural filtration. Since we have the tightness of $\la \mathcal M^N(f)\ra_t$ for every fixed $t \in [0,T]$, \cite[Proposition~6.13, Chapter VI]{Limit_theorems_for_stochastic_processes} implies that $\mathcal M^N_t(f)$ is predictably uniformly tight. Then by \cite[Theorem~6.26, Chapter VI]{Limit_theorems_for_stochastic_processes}, we have the joint weak convergence along any convergent subsequence
        \begin{equation*}
         (\mathcal M^{N_k}_t,\la\mathcal M^{N_k}\ra_t)_{t\in [0,T]}  \xRightarrow{N_k \to \infty} (\M_t,\la\mathcal M\ra_t)_{t \in [0,T]}.
         \end{equation*}
        This completes the proof.
	\end{proof}
	The characterization of the limit quadratic variation will be the key ingredient in identifying the law of the limit martingale which will be left to Section~\ref{sec.M.limit}.
	\subsection{Moment estimate of $\mathcal{B}^N$}\label{sec.M.moment}
	In this section, we give the moment estimates of $\mathcal{B}^N$ and some related terms. Recall \eqref{eq.quadratic variation} that for every test function $f$, 
	\begin{align*}\label{eq.quadratic variation}
		\mathcal{B}^N(f)=N^{2}\sum_{i=1}^d\sum_{x}c_{x,x+e_i}\left(\eta^N\right)\left(\pi_{x,x+e_i}\mathcal{Q}^N(f)\right)^2\notag.
	\end{align*}
	We make a decomposition
	\begin{equation}\label{eq.pi.Q}
		\pi_{x,x+e_i}\mathcal{Q}^N(f)=\mathcal I_1^N(f,x,i)+\mathcal R_1^N(f,x,i).
	\end{equation}
	The main term is ${\mathcal  I}_1^N$:
	\begin{equation*}
		{\mathcal  I}_1^N(f,x,i)=2N^{-1-d}\sum_{\substack{y:y\notin \Lambda_L^{z(x)}\\y\neq x,x+e_i}}\bar\eta^N(y)
		\mathbf{v}_{x,i}\cdot{\nabla_{1}^{N}f\left(\frac{z(x)}{N},\frac{y}{N}\right)},
	\end{equation*}
	where the vector field $\mathbf{v}_{x,i}=\{\mathbf{v}_{x,i,j}\}_{1\leq j\leq d}$ is defined by
	\begin{align}\label{eq.v}
		\mathbf{v}_{x,i,j} &:=\pi_{x,x+e_i}\left(\ell_{e_j}+\phi_L^{z(x)}\cdot e_j\right) \\
		&=\left(\bar\eta^N(x)-\bar\eta^N(x+e_i)\right)e_i\cdot e_j+\pi_{x,x+e_i}\phi_L^{z(x)}\cdot e_j. \notag
	\end{align}
	The remainder $\mathcal R_1^N$ consists of three parts:
	\begin{equation*}
		\mathcal R_1^N(f,x,i)=\mathcal R_{1,1}^N(f,x,i)+\mathcal R_{1,2}^N(f,x,i)+\mathcal R_{1,3}^N(f,x,i). 
	\end{equation*}
	The three terms are
	\begin{align*}
		\mathcal R_{1,1}^N(f,x,i)&:=2N^{-1-d}\sum_{\substack{ z\in \mathcal Z_L }}
		\phi_L^{z}\cdot\pi_{x,x+e_i}\left(\sum_{y:y\notin \Lambda_L^z}\bar\eta^N(y){\nabla_{1}^{N}f\left(\frac{z}{N},\frac{y}{N}\right)}\right),\\
		\mathcal R_{1,2}^N(f,x,i)&:=2N^{-1-d}\sum_{\substack{y\in \Lambda_L^{z(x)}\\y\neq x,x+e_i}}\bar\eta^N(y)
		\left(\bar\eta^N(x)-\bar\eta^N(x+e_i)\right)\nabla_{1,i}^{N}f\left(\frac{x}{N},\frac{y}{N}\right),\\
		\mathcal R_{1,3}^N(f,x,i)&:=2N^{-1-d}\sum_{\substack{y:y\notin \Lambda_L^{z(x)}\\y\neq x+e_i}}\bar\eta^N(y)\left(\bar\eta^N(x)-\bar\eta^N(x+e_i)\right)\\
		&\hspace{5cm}\left\{\nabla_{1,i}^{N}f\left(\frac{x}{N},\frac{y}{N}\right)-{\nabla_{1,i}^{N}f\left(\frac{z(x)}{N},\frac{y}{N}\right)}\right\}\notag.
	\end{align*}
	The three terms have the following interpretations. The first remainder $\mathcal R_{1,1}^N$ is  the Kawasaki operator $\pi$ acting on the linear statistic part of the corrector. The second remainder $\mathcal R_{1,2}^N$ describes the terms that are near the diagonal, and the third remainder $\mathcal R_{1,3}^N$ is to fix the slope. 
	
	For the proof of Proposition~\ref{prop.B.bound.L1.L2}, we need some preparations. We calculate the second and fourth moments for $\mathcal  I_1^N$ and the three remainders $\mathcal R_{1,1}^N$, $\mathcal R_{1,2}^N$ and $\mathcal R_{1,3}^N$. Indeed, $\mathcal  I_1^N$ is the main contribution for the moments of $\mathcal B^N$, which will be stated as follows. The moments of the remainders are all small, which will be postponed to the appendix. 
	
	\begin{lemma}\label{lemma.I_1.expectation}
		For every test function $f$, we have the following moment estimates for $\mathcal  I_1^N$:
		\begin{equation*}
			N^2\sum_{i=1}^d\sum_x\E_\rho\left[{\mathcal  I}_1^N(f,x,i)^2 \right]\leq C_{f},
		\end{equation*}
		and
		\begin{equation*}
			N^4\sum_{i=1}^d\sum_x\E_\rho\left[{\mathcal  I}_1^N(f,x,i)^4\right]\leq C_{f}N^{-d}\Vert \phi_L\Vert_\infty^4.
		\end{equation*}
	\end{lemma}
	\begin{proof}
		We make a direct calculation:
		\begin{align}\label{eq.I_1.expectation.square.1}
			&\quad N^2\sum_{i=1}^d\sum_x\E_\rho\left[{\mathcal  I}_1^N(f,x,i)^2 \right]\\
			&= 4N^{-2d}\sum_{i,j=1}^d\sum_{\substack{ z\in \mathcal Z_L }}\sum_{x\in \Lambda_L^z}\E_\rho\left[\left\{\sum_{\substack{y:y\notin \Lambda_L^{z}\\y\neq x,x+e_i}}\bar\eta^N(y)
			\mathbf{v}_{x,i,j}{\nabla_{1,j}^{N}f\left(\frac{z}{N},\frac{y}{N}\right)}\right\}^2\right]\notag\\
			&\leq 4d N^{-2d}\sum_{i,j=1}^d\sum_{\substack{ z\in \mathcal Z_L }}\sum_{x\in \Lambda_L^z}\E_{\rho}\left[\left(\mathbf{v}_{x,i,j}\right)^2\right]\left\{\sum_{\substack{y:y\notin \Lambda_L^z\\y\neq x+e_i}}\E_{\rho}\left[\bar\eta^N(y)^2\right]{\nabla_{1,j}^{N}f\left(\frac{z}{N},\frac{y}{N}\right)}^2\right\}\notag\\
			&\leq C_{f}N^{-d} \sum_{\substack{ z\in \mathcal Z_L }}\sum_{i=1}^d\sum_{x\in\Lambda_L^z}\E_{\rho}\left[\left|\mathbf{v}_{x,i}\right|^2\right]\notag,
		\end{align}
		where from the second line to the third line, we use Cauchy--Schwarz inequality and independence. Note that
		\begin{equation}\label{eq.I_1.expectation.square.2}
			\sum_{i=1}^d\sum_{x\in\Lambda_L^z}\E_{\rho}\left[\left|\mathbf{v}_{x,i}\right|^2\right]
			\leq \E_{\rho}\left[\sum_{i=1}^d\sum_{x\in\Lambda_L^z}c_{x,x+e_i}\left|\mathbf{v}_{x,i}\right|^2\right]
			\leq L^d\left(\sum_{j=1}^d\cc_{jj}(\rho)+CL^{-\gamma_1}\right)\leq CL^d,
		\end{equation}
		where the second step is shown in \cite[Proposition~6.1]{FGW24}. Combining \eqref{eq.I_1.expectation.square.1} and \eqref{eq.I_1.expectation.square.2}, we have
		\begin{align*}
			\quad N^2\sum_{i=1}^d\sum_x\E_\rho\left[{\mathcal  I}_1^N(f,x,i)^2 \right]
			\leq C_{f}N^{-d} \sum_{\substack{ z\in \mathcal Z_L }}\sum_{i=1}^d\sum_{x\in\Lambda_L^z}\E_{\rho}\left[\left|\mathbf{v}_{x,i}\right|^2\right]\leq C_{f}.
		\end{align*}
		For the fourth moment, we have
		\begin{align*}
			&\quad N^4\sum_{i=1}^d\sum_x\E_\rho\left[{\mathcal I}_1^N(f,x,i)^4\right]\\
			&\leq 16d^3 N^{-4d}\sum_{\substack{ z\in \mathcal Z_L }}\sum_{i,j=1}^d\sum_{x\in \Lambda_L^z}\E_{\rho}\left[\left(\mathbf{v}_{x,i,j}\right)^4\right]\E_{\rho}\left[\left\{\sum_{\substack{y:y\notin \Lambda_L^z\\y\neq x+e_i}}\bar\eta^N(y){\nabla_{1,j}^{N}f\left(\frac{z}{N},\frac{y}{N}\right)}\right\}^4
			\right]\notag\\
			&\leq C_{f}N^{-4d} \left(\frac{N}{L}\right)^d L^d \left(N^{2d}+N^d\right)\Vert \phi_L\Vert_\infty^4\notag\\
			&\leq C_{f}N^{-d}\Vert \phi_L\Vert_\infty^4,
		\end{align*}
		where from the first line to the second line, we use Cauchy--Schwarz inequality and from the second line to the third line, we use the observation that the expectation can be non-zero only when four $y$'s pair each other.
	\end{proof}   
	We put the moment estimates of the remainder $\mathcal R_1^N$ here, with proofs deferred to the Appendix.
	\begin{lemma}\label{lemma.R_1.expectation}
		For every test function $f$, we have the following moment estimates for the remainder term $\mathcal R_1^N$:
		\begin{align*}
			&N^2\sum_{i=1}^d\sum_{x}\E_{\rho}\left[\mathcal R_{1}^N(f,x,i)^2\right]
			\leq C_{f}\left(N^{-2}L^2+N^{-d}L^{d+2}\right),
		\end{align*}
		and
		\begin{multline*}
			\quad N^4\sum_{i=1}^d\sum_{x}\E_{\rho}\left[\mathcal R_{1}^N(f,x,i)^4\right]\\
			\leq C_{f}N^{-d}\left(N^{-4}L^4+N^{-2-d}L^{d+3}+N^{-2d}L^{2d+3}+\left(N^{-4-d}L^{-2d}+N^{-2d}L^{-1}\right)\Vert \phi_L\Vert_\infty^4\right)\notag.
		\end{multline*}
	\end{lemma}
	We end this section with the proof of Proposition~\ref{prop.B.bound.L1.L2}:
	\begin{proof}[Proof of Proposition~\ref{prop.B.bound.L1.L2}]
		Recall that by \eqref{eq.pi.Q}, we have
		\begin{align*}
			\E_\rho\left[\mathcal{B}^N(f)\right]&=N^{2}\E_\rho\left[\sum_{i=1}^d\sum_{x}c_{x,x+e_i}\left(\eta^N\right)\left(\pi_{x,x+e_i}\mathcal{Q}^N(f)\right)^2\right]\\
			&\leq \lambda N^{2}\sum_{i=1}^d\sum_{x} \E_\rho\left[\left(\pi_{x,x+e_i}\mathcal{Q}^N(f)\right)^2\right]\\
            &= \lambda N^{2}\sum_{i=1}^d\sum_{x} \E_\rho\left[\left(\mathcal I_1^N(f,x,i)+\mathcal R_1^N(f,x,i)\right)^2\right].
		\end{align*}
		Combining Cauchy--Schwarz inequality, Lemma~\ref{lemma.I_1.expectation} and Lemma~\ref{lemma.R_1.expectation}, we conclude that $\mathcal{B}^N(f)$ is bounded in $L^1$.
		
		As for $L^2$-norm of $\mathcal B^N(f)$, we calculate
		\begin{align*}
			\E_\rho\left[\left|\mathcal{B}^N(f)\right|^2\right]&=N^{4}\E_\rho\left[\left(\sum_{i=1}^d\sum_{x}c_{x,x+e_i}\left(\eta^N\right)\left(\pi_{x,x+e_i}\mathcal{Q}^N(f)\right)^2\right)^2\right]\\
			&\leq d N^{4} \sum_{i=1}^d \E_\rho\left[\left(\sum_{x}c_{x,x+e_i}\left(\eta^N\right)\left(\pi_{x,x+e_i}\mathcal{Q}^N(f)\right)^2\right)^2\right]\notag\\
			&\leq \lambda^2d N^{4} \sum_{i=1}^d \sum_{x_1,x_2}\E_\rho\left[\left(\pi_{x_1,x_1+e_i}\mathcal{Q}^N(f)\right)^2\left(\pi_{x_2,x_2+e_i}\mathcal{Q}^N(f)\right)^2\right]\notag.
		\end{align*}
		Here the first line to the second line follows from Cauchy--Schwarz inequality and the second line to the third line follows from the uniform upper bound of $c$. For the simplicity of notations, we omit the parameter $f$ in the following calculation. We calculate
		\begin{align*}
			&\quad\E_\rho\left[\left(\pi_{x_1,x_1+e_i}\mathcal{Q}^N\right)^2\left(\pi_{x_2,x_2+e_i}\mathcal{Q}^N\right)^2\right]\\
			&=\E_\rho\bigg[\left({\mathcal I}_1^N(x_1,i)+\mathcal R_{1}^N(x_1,i)\right)^2\left({\mathcal I}_1^N(x_2,i)+\mathcal R_{1}^N(x_2,i)\right)^2\bigg]\notag\\
			&\leq 4\E_\rho\bigg[\left({\mathcal I}_1^N(x_1,i)^2+\mathcal R_1^N(x_1,i)^2\right)\left({\mathcal I}_1^N(x_2,i)^2+\mathcal R_1^N(x_2,i)^2\right)\bigg]\notag.
		\end{align*}
		There are $4$ terms to be estimated. Except for the term
		\begin{equation*}
			\E_\rho\left[{\mathcal I}_1^N(x_1,i)^2{\mathcal I}_1^N(x_2,i)^2\right],
		\end{equation*}
		we estimate the other $3$ terms in the same way. Take 
		\begin{equation*}
			\E_\rho\left[{\mathcal I}_1^N(x_1,i)^2{\mathcal R}_1^N(x_2,i)^2\right],
		\end{equation*}
		for an example. By Cauchy--Schwarz inequality, we have
		\begin{align}
			&\quad N^{4} \sum_{i=1}^d\sum_{x_1,x_2}\E_\rho\left[{\mathcal I}_1^N(x_1,i)^2{\mathcal R}_1^N(x_2,i)^2\right]\notag\\
			&\leq N^{4} \sum_{i=1}^d\sum_{x_1,x_2}\E_\rho\left[{\mathcal I}_1^N(x_1,i)^4\right]^{\frac{1}{2}}\E_\rho\left[{\mathcal R}_1^N(x_2,i)^4\right]^{\frac{1}{2}}\notag\\
			&\leq \left(N^4\sum_{i=1}^d\sum_{x_1,x_2}\E_\rho\left[{\mathcal I}_1^N(x_1,i)^4\right]\right)^{\frac{1}{2}}\left(N^4\sum_{i=1}^d\sum_{x_1,x_2}\E_\rho\left[{\mathcal R}_1^N(x_2,i)^4\right]\right)^{\frac{1}{2}}\notag\\
			&\leq \left(N^4N^d\sum_{i=1}^d\sum_{x}\E_\rho\left[{\mathcal I}_1^N(x,i)^4\right]\right)^{\frac{1}{2}}\left(N^4N^d\sum_{i=1}^d\sum_{x}\E_\rho\left[{\mathcal R}_1^N(x,i)^4\right]\right)^{\frac{1}{2}}\notag.
		\end{align}
		By Lemma~\ref{lemma.I_1.expectation} and Lemma~\ref{lemma.R_1.expectation}, we  have bounds for the two multiplicative terms:
        \begin{multline*}
            N^{4+d}\sum_{i=1}^d\sum_{x}\E_\rho\left[{\mathcal I}_1^N(x,i)^4\right]\\
            \leq C_{f}\left(N^{-4}L^4+N^{-2-d}L^{d+3}+N^{-2d}L^{2d+3}+\left(N^{-4-d}L^{-2d}+N^{-2d}L^{-1}\right)\Vert \phi_L\Vert_\infty^4\right),
        \end{multline*}
        and 
        \begin{equation*}
           N^{4+d}\sum_{i=1}^d\sum_{x}\E_\rho\left[{\mathcal R}_1^N(x,i)^4\right]\leq C_f\Vert \phi_L\Vert_\infty^4.
        \end{equation*}
        Therefore, with Hypothesis~\ref{hyp} and the corrector estimate in \eqref{lem.CorrectorLp} of Lemma~\ref{lem.Homogenization}, we know that 
        \[
        N^{4} \sum_{i=1}^d\sum_{x_1,x_2}\E_\rho\left[{\mathcal I}_1^N(x_1,i)^2{\mathcal R}_1^N(x_2,i)^2\right]=o(N).
        \]
        And we have the same estimate for the other two terms 
        \[
        N^{4} \sum_{i=1}^d\sum_{x_1,x_2}\E_\rho\left[{\mathcal R}_1^N(x_1,i)^2{\mathcal I}_1^N(x_2,i)^2\right],\quad N^{4} \sum_{i=1}^d\sum_{x_1,x_2}\E_\rho\left[{\mathcal R}_1^N(x_1,i)^2{\mathcal R}_1^N(x_2,i)^2\right].
        \]
		Now we need to estimate the main contribution: 
		\begin{equation*}
			\quad N^{4} \sum_{i=1}^d \sum_{x_1,x_2}\E_\rho\left[{\mathcal I}_1^N(x_1,i)^2{\mathcal I}_1^N(x_2,i)^2\right].
		\end{equation*}
		By Cauchy-Schwarz inequality, this term is bounded by 
        \begin{equation}\label{eq.I_1.I_1.expectation.square}
			C_{f}N^{-4d}\sum_{\mathcal E}
			\E_\rho\left[\bar\eta^N(y_1)\bar\eta^N(y_1')\bar\eta^N(y_2)\bar\eta^N(y_2')\left|\mathbf{v}_{x_1,i}\right|^2\left|\mathbf{v}_{x_2,i}\right|^2\right],
		\end{equation}
		where $\mathcal E$ denotes collection of admissible tuples, consisting of all variables satisfying
        \begin{align}\label{eq.defE}
            \mathcal{E} := \left\{
		      \begin{aligned}
			         &z_1, z_2 \in \mathcal Z_L,\ 1 \leq i \leq d,\\
			         &x_1 \in \Lambda_L^{z_1},\ x_2 \in \Lambda_L^{z_2},\\
			         &y_1, y_1' \in \Td_N \setminus \left( \Lambda_L^{z_1} \cup \{x_1 + e_i\} \right),\\
			         &y_2, y_2' \in \Td_N \setminus \left( \Lambda_L^{z_2} \cup \{x_2 + e_i\} \right)
		      \end{aligned}
		      \right\}.
        \end{align}
        
		To calculate the expectation in \eqref{eq.I_1.I_1.expectation.square}, we make use of a Wick diagram. For each pair of the terms $\bar\eta^N(y)$'s and $\left|\mathbf{v}_{x,i}\right|^2$'s, if their supports have non-empty intersections, we draw a bond between them. Otherwise, there is no bond between the two terms. We notice that the expectation in \eqref{eq.I_1.I_1.expectation.square} vanishes once some $\bar\eta^N(y)$ is not paired. Therefore, we only focus on the non-vanishing tuples, denoted by $\mathcal{E}_*$. 
        \begin{align*}
            \mathcal{E}_* := \{(i, z_1, z_2, x_1, x_2, y_1, y_1', y_2, y_2') \in \mathcal{E}: &\\
            &\text{ all the } \bar\eta^N(y) \text{ are paired in the Wick diagram}\}.
        \end{align*}
        We then study the contribution case by case.		
		
		\textbf{Case 1:} $\left(\Lambda_L^{z_{1}}\cup \{x_1+e_i\}\right)\cap\left(\Lambda_L^{z_{2}}\cup \{x_2+e_i\}\right)\neq \emptyset$.

		In this case, there will be bonds between $\left|\mathbf{v}_{x_1,i}\right|^2$ and $\left|\mathbf{v}_{x_2,i}\right|^2$. There are two possible situations. 
		
		\textbf{Case 1.1:} $z_1=z_2$.
		
		In this case, the expectation is not zero if and only if two of $y_1,y_1',y_2,y_2'$ share the same value and the other two share the same value at the same time. The contribution in $\mathcal E_*$ is 
		\begin{equation*}
			\mathcal E_{1,1}=\mathcal E_*\cap\{z_1=z_2\}
		\end{equation*}
		\begin{figure}[htbp]
			\centering
			\includegraphics[width=0.6 \textwidth]{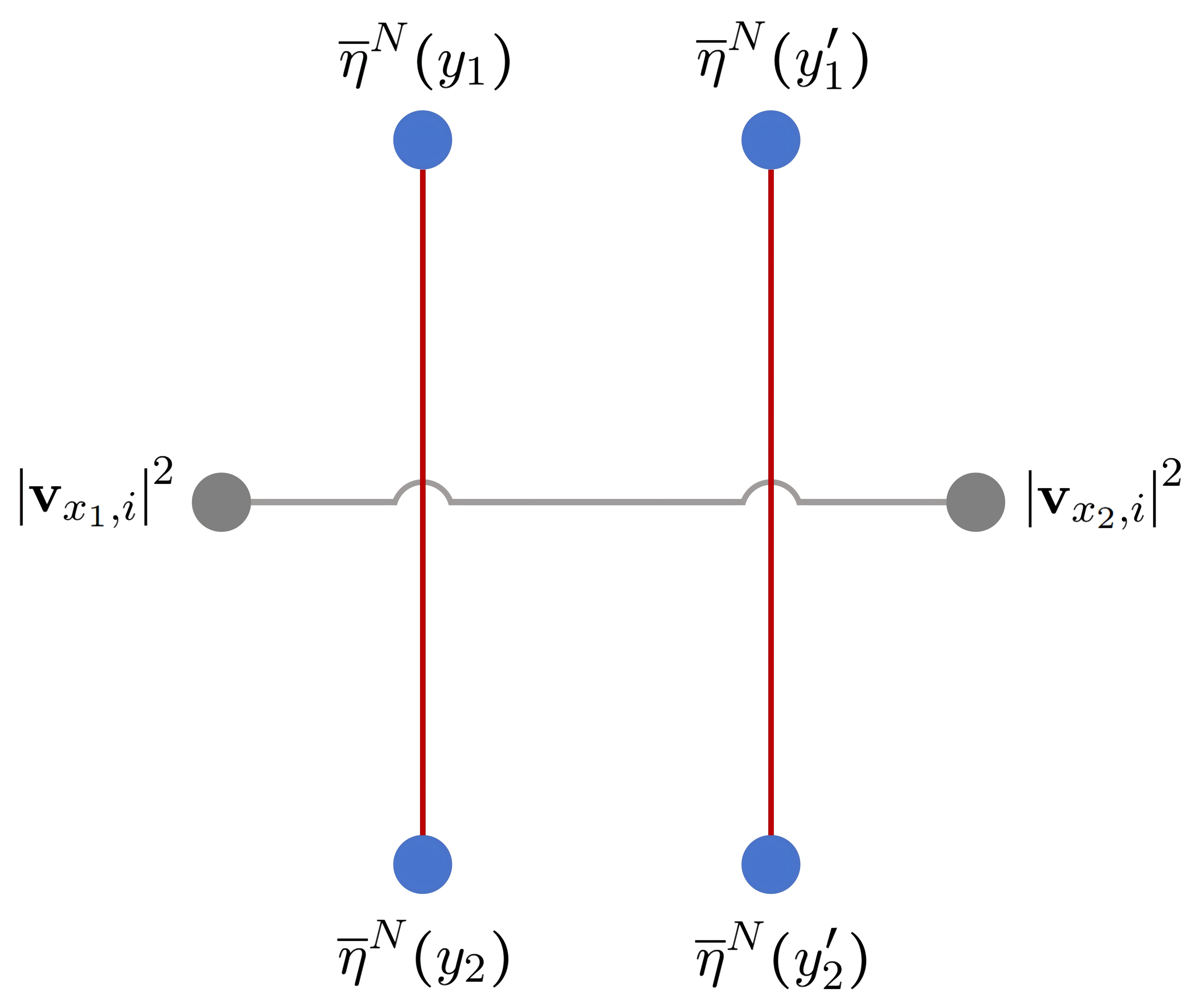}
			\caption*{Case 1.1}
		\end{figure}

		Therefore, the contribution in \eqref{eq.I_1.I_1.expectation.square} is bounded by
		\begin{align}\label{eq.I_1.I_1.expectation.square.Case1.1}
			&\quad N^{-4d}\sum_{\mathcal E_{1.1}}
			\E_\rho\left[\bar\eta^N(y_1)\bar\eta^N(y_1')\bar\eta^N(y_2)\bar\eta^N(y_2')\left|\mathbf{v}_{x_1,i}\right|^2\left|\mathbf{v}_{x_2,i}\right|^2\right]\\
			&\leq N^{-4d}\sum_{\substack{z\in \mathcal Z_L }}\sum_{i=1}^d\sum_{x_1,x_2\in \Lambda_L^{z}} \E_\rho\left[\left|\mathbf{v}_{x_1,i}\right|^2\left|\mathbf{v}_{x_2,i}\right|^2 \1_{\mathcal E_{1,1}}\right] \notag \\
            & \qquad \cdot \sum_{\substack{y_1,y_1',y_2,y_2'}}\E_\rho\left[\bar\eta^N(y_1)\bar\eta^N(y_1')\bar\eta^N(y_2)\bar\eta^N(y_2')\1_{\mathcal E_{1,1}}\right]\notag\\
			&\leq CN^{-4d} \left(\frac{N}{L}\right)^d \left(L^d\right)^2 \Vert \phi_L\Vert_\infty^4\left(N^{2d}+N^{d}\right)\notag\\
			&\leq CN^{-d}L^d\Vert \phi_L\Vert_\infty^4\notag.
		\end{align}
		\textbf{Case 1.2:} $z_1\neq z_2$ and the intersection occurs through a boundary bond.
		
		In this case, $\Lambda_L^{z_1}\cap\Lambda_L^{z_2}=\emptyset$. Since there is an intersection between $\mathbf{v}_{x_1,i}$ and $\mathbf{v}_{x_2,i}$, $x_1+e_i\in \Lambda_L^{z_2}$ or $x_2+e_i\in \Lambda_L^{z_1}$. We take the first case as an example. This implies $z_2=z_1+Le_i$. We have the observation that when $x_1+e_i\in \Lambda_L^{z_2}$ with $z_1\neq z_2$, we have
		\begin{equation*}
			\left|\mathbf{v}_{x
				_1,i}\right|^2=\left(\bar\eta^N(x_1)-\bar\eta^N(x_1+e_i)\right)^2.
		\end{equation*}
		The contribution in $\mathcal E_*$ is 
		\begin{equation*}
			\mathcal E_{1,2}=\mathcal E_*\cap\left\{z_1\neq z_2, \ x_1+e_i\in \Lambda_L^{z_2} \ \mathrm{or}\  x_2+e_i\in \Lambda_L^{z_1}\right\}.
		\end{equation*}
		To make the expectation non-zero, there are two possibilities:
		\begin{itemize}
			\item $y_1,y_1',y_2,y_2'$ pair with each other,
			\item $y_2=y_2'$ and $y_1,y_1'\in \Lambda_L^{z_2}\cup \{x_2+e_i\}$.
		\end{itemize}
        \begin{figure}[htbp]
			\centering
			\includegraphics[width=0.6\textwidth]{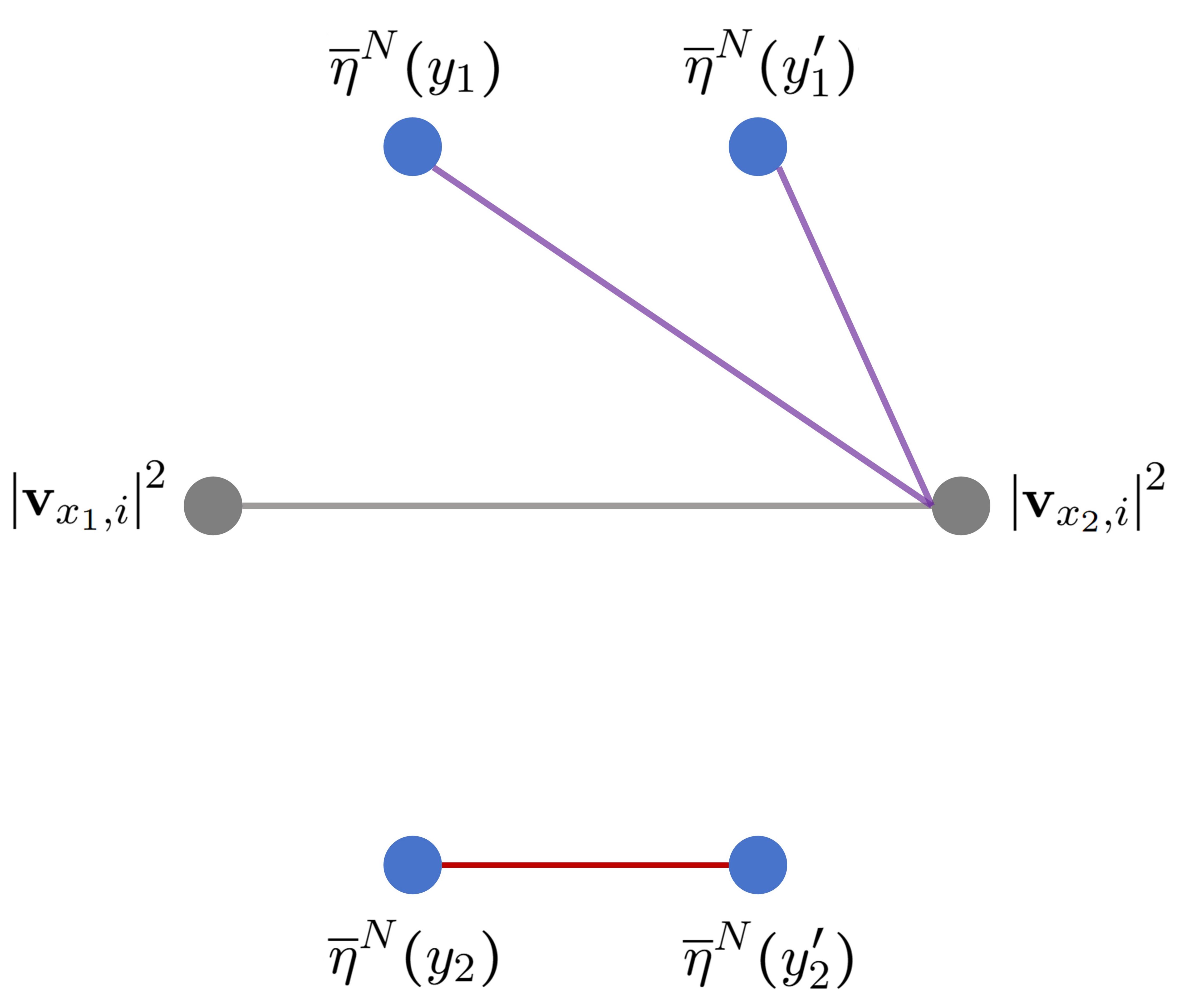}
			\caption*{Case 1.2}
		\end{figure}
		Therefore, the contribution in \eqref{eq.I_1.I_1.expectation.square} is
		\begin{align}\label{eq.I_1.I_1.expectation.square.Case1.2}
			&\quad 2N^{-4d}\sum_{\mathcal E_{1.2}}
			\E_\rho\left[\bar\eta^N(y_1)\bar\eta^N(y_1')\bar\eta^N(y_2)\bar\eta^N(y_2')\left|\mathbf{v}_{x_1,i}\right|^2\left|\mathbf{v}_{x_2,i}\right|^2\right]\\
			&\leq 2N^{-4d}\sum_{\substack{z_1\in \mathcal Z_L}}\sum_{\substack{x_1\in\Lambda_L^z}}\1_{\left\{x_1+e_i\in \Lambda_L^{z+Le_i}\right\}}\sum_{x_2\in \Lambda_L^{z+Le_i}}\sum_{\substack{y_1,y_1':y_1,y_1'\notin \Lambda_L^{z_1}\\y_1,y_1'\neq x_1+e_i}}\sum_{\substack{y_2,y_2':y_2,y_2'\notin \Lambda_L^{z_1+Le_i}\\y_2,y_2'\notin x_2+e_i}}\notag\\
			&\hspace{2cm}\E_\rho\left[\bar\eta^N(y_1)\bar\eta^N(y_1')\bar\eta^N(y_2)\bar\eta^N(y_2')\left|\mathbf{v}_{x_1,i}\right|^2\left|\mathbf{v}_{x_2,i}\right|^2\right]\notag\\
			&\leq CN^{-4d}\sum_{\substack{z_1\in \mathcal Z_L}}\sum_{\substack{x_1\in\Lambda_L^z\\x_1+e_i\in \Lambda_L^{z+Le_i}}}\sum_{x_2\in \Lambda_L^{z+Le_i}}\Vert \phi_L\Vert_\infty^2\left(N^{2d}+N^dL^{2d}\right)\notag\\
			&\leq CN^{-d}L^{d-1}\Vert \phi_L\Vert_\infty^2\left(1+N^{-d}L^{2d}\right)\notag.
		\end{align}
		\textbf{Case 2:} $\left(\Lambda_L^{z_1}\cup \{x_1+e_i\}\right)\cap\left(\Lambda_L^{z_{2}}\cup \{x_2+e_i\}\right)= \emptyset$.
		
		In this case, there will be no bond between $\left|\mathbf{v}_{x_1,i}\right|^2$ and $\left|\mathbf{v}_{x_2,i}\right|^2$. We further classify the cases according to the number of bonds between $\bar\eta^N(y)$'s  and $\left|\mathbf{v}_{x,i}\right|^2$'s, which is denoted by $k$. 
        \begin{itemize}
            \item $k = 0$: In this case, the expectation is not zero if two among $y_1,y_1',y_2,y_2'$ are paired, and the other two are also paired at the same time. The contribution in $\mathcal E_*$ is 
		\begin{multline*}
			\mathcal E_{2.0}=\mathcal E_*\cap\left\{z_1\neq z_2, \ y_1,y_1' \notin \Lambda_L^{z_2} \cup \{x_2 + e_i\}, \ y_2,y_2' \notin \Lambda_L^{z_1} \cup \{x_1 + e_i\} ,\right.\\\left.y_1=y_1'=y_2=y_2' \ \mathrm{or} \ y_1, y_1',y_2,y_2' \ \mathrm{pair \ each \ other \ in }\  2 \ \mathrm{groups}  \right\}.
		\end{multline*}
		\begin{figure}[htbp]
			\centering
			\includegraphics[width=0.6\textwidth]{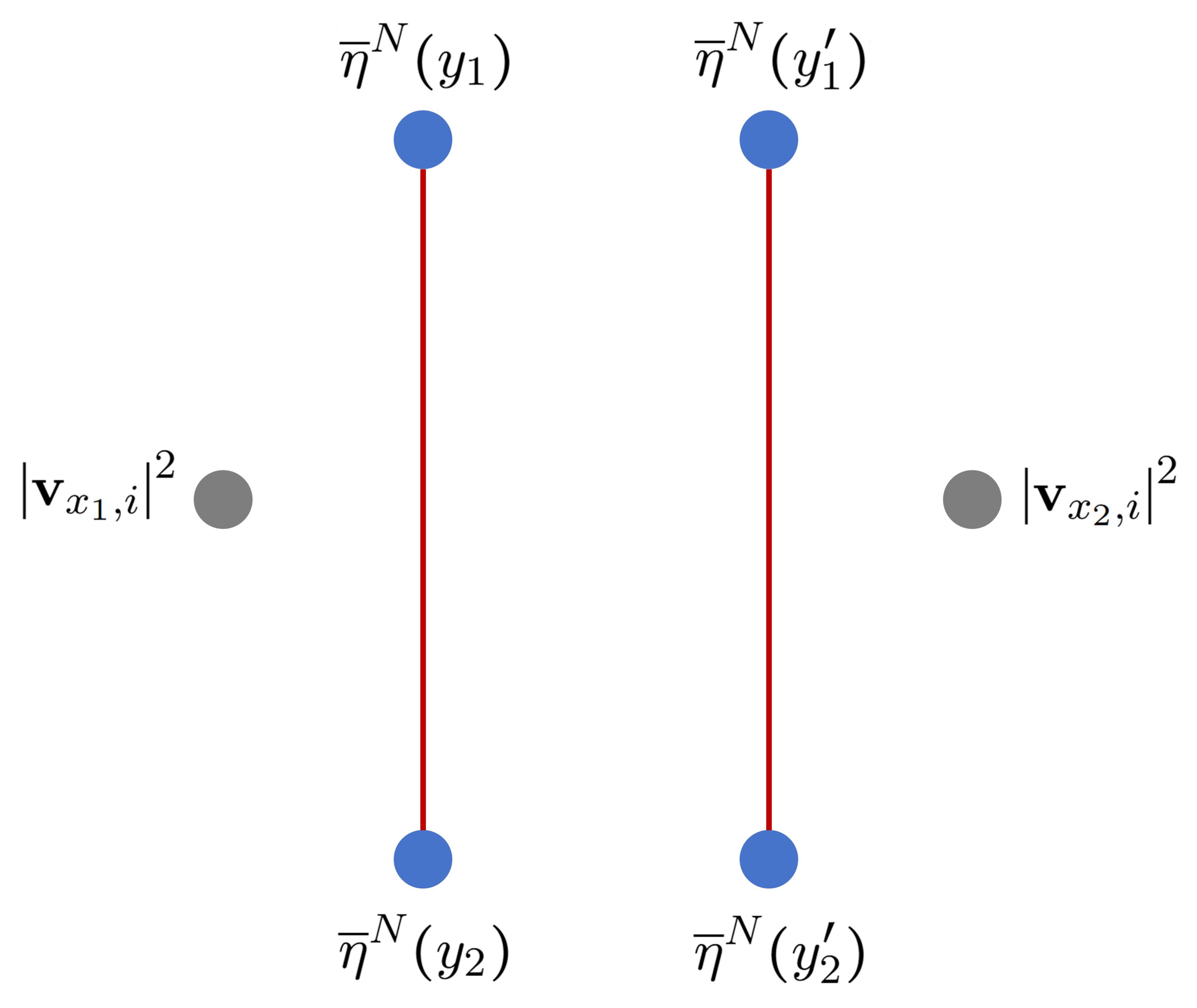}
			\caption*{Case 2: $k=0$}
		\end{figure}

		Therefore, the contribution in \eqref{eq.I_1.I_1.expectation.square} is
		\begin{align}\label{eq.I_1.I_1.expectation.square.Case2.0}
			&\quad N^{-4d} \sum_{\mathcal E_{2.0}} \E_\rho\left[\bar\eta^N(y_1)\bar\eta^N(y_1')\bar\eta^N(y_2)\bar\eta^N(y_2')\left|\mathbf{v}_{x_1,i}\right|^2\left|\mathbf{v}_{x_2,i}\right|^2\right]\\
			&\leq N^{-4d}\sum_{\substack{z_1,z_2\in \mathcal Z_L }}\sum_{i=1}^d\sum_{x_1\in \Lambda_L^{z_1}}\sum_{x_2\in \Lambda_L^{z_2}}\E_\rho\left[\left|\mathbf{v}_{x_1,i}\right|^2\right]\E_\rho\left[\left|\mathbf{v}_{x_2,i}\right|^2\right]\notag\\
			&\hspace{3cm}\left\{\sum_{y_1,y_1',y_2,y_2'}\E_\rho\left[\bar\eta^N(y_1)\bar\eta^N(y_1')\bar\eta^N(y_2)\bar\eta^N(y_2')\right]\right\}\notag\\
			&\leq CN^{-4d}\sum_{\substack{z_1,z_2\in \mathcal Z_L }}\sum_{i=1}^d\left\{\sum_{x_1\in \Lambda_L^{z_1}}\E_\rho\left[\left|\mathbf{v}_{x_1,i}\right|^2\right]\right\}\left\{\sum_{x_2\in \Lambda_L^{z_2}}\E_\rho\left[\left|\mathbf{v}_{x_2,i}\right|^2\right]\right\}\left(N^{2d}+N^{d}\right)\notag\\
			&\leq CN^{-2d} \left(\frac{N}{L}\right)^{2d}L^{2d}\notag,
		\end{align}
		where the last step follows from \eqref{eq.I_1.expectation.square.2}.
            \item $k = 1$: In this case, if the expectation is not zero, one $\bar\eta^N(y)$ connects to $\left|\mathbf{v}_{x,i}\right|^2$, and the other $3$ terms of type $\bar\eta^N(y)$ have to share the same value. The contribution in $\mathcal E_*$ is 
        \begin{multline*}
            \mathcal{E}_{2.1} = \mathcal{E}_* \cap \left\{ z_1 \neq z_2,\; \mathbf{1}_{\{y_1 \in \Lambda_L^{z_2} \cup \{x_2+e_i\}\}} + \mathbf{1}_{\{y_1' \in \Lambda_L^{z_2} \cup \{x_2+e_i\}\}} \right.\\\left.+ \mathbf{1}_{\{y_2 \in \Lambda_L^{z_1} \cup \{x_1+e_i\}\}} + 
\mathbf{1}_{\{y_2' \in \Lambda_L^{z_1} \cup \{x_1+e_i\}\}} = 1 \right\}.
        \end{multline*}
		\begin{figure}[htbp]
			\centering
			\includegraphics[width=0.6\textwidth]{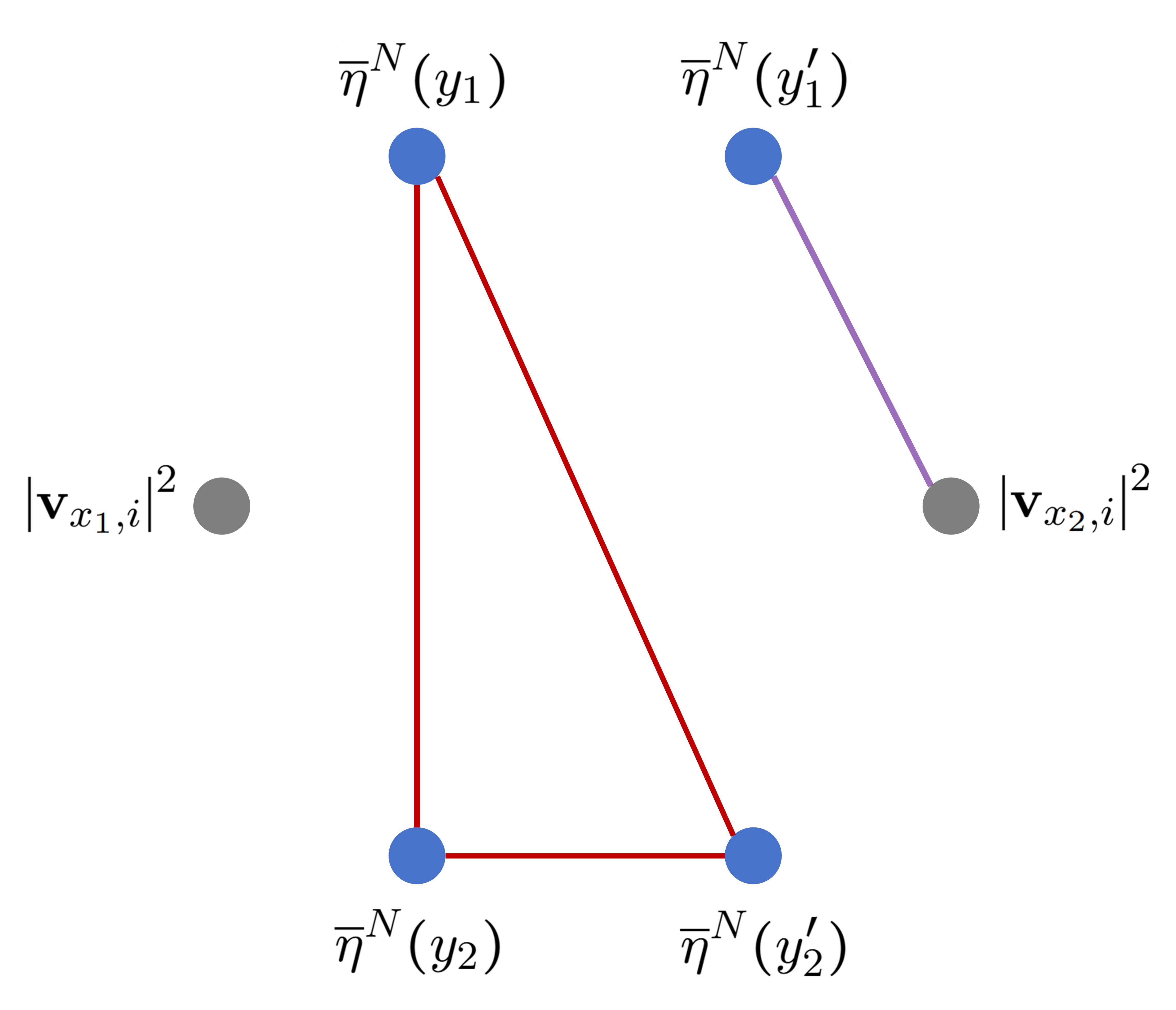}
			\caption*{Case 2: $k=1$.}
		\end{figure}
		
		Therefore, the contribution in \eqref{eq.I_1.I_1.expectation.square} is \begin{align}\label{eq.I_1.I_1.expectation.square.Case2.1}
			&\quad N^{-4d}\sum_{\mathcal E_{2.1}}\E_\rho\left[\bar\eta^N(y_1)\bar\eta^N(y_1')\bar\eta^N(y_2)\bar\eta^N(y_2')\left|\mathbf{v}_{x_1,i}\right|^2\left|\mathbf{v}_{x_2,i}\right|^2\right]\\
			&\leq 4N^{-4d}\sum_{\substack{z_1,z_2\in \mathcal Z_L }}\sum_{i=1}^d\sum_{x_1\in \Lambda_L^{z_1}}\sum_{x_2\in\Lambda_L^{z_2}}\sum_{y_1'\in\Lambda_L^{z_2}\cup \{x_2+e_i\}}
			\sum_{\substack{y_1,y_2,y_2'}}\notag\\
			&\hspace{2cm}\E_\rho\left[\bar\eta^N(y_1')\left|\mathbf{v}_{x_2,i}\right|^2\right]\E_\rho\left[\bar\eta^N(y_1)\bar\eta^N(y_2)\bar\eta^N(y_2')\right]\E_\rho\left[\left|\mathbf{v}_{x_1,i}\right|^2\right]\notag\\
			&\leq CN^{-4d}\left(\frac{N}{L}\right)^{2d}L^{2d}L^dN^d\Vert \phi_L\Vert_\infty^4\notag\\
			&\leq CN^{-d}L^{d}\Vert \phi_L\Vert_\infty^4\notag.
		\end{align}

            \item $k = 2$: For this case, if the expectation is not zero, the remaining two $\bar\eta^N(y)$'s disconnected to $\left|\mathbf{v}_{x,i}\right|^2$'s must pair with each other. The contribution in $\mathcal E_*$ is 
        \begin{multline*}
            \mathcal{E}_{2.2} = \mathcal{E}_* \cap \Bigl\{ z_1 \neq z_2,\; \mathbf{1}_{\{y_1 \in \Lambda_L^{z_2} \cup \{x_2+e_i\}\}} + \mathbf{1}_{\{y_1' \in \Lambda_L^{z_2} \cup \{x_2+e_i\}\}} \\+ \mathbf{1}_{\{y_2 \in \Lambda_L^{z_1} \cup \{x_1+e_i\}\}} + 
\mathbf{1}_{\{y_2' \in \Lambda_L^{z_1} \cup \{x_1+e_i\}\}} = 2 \Bigr\}.
        \end{multline*}
		\begin{figure}[htbp]
			\centering
			\includegraphics[width=0.6\textwidth]{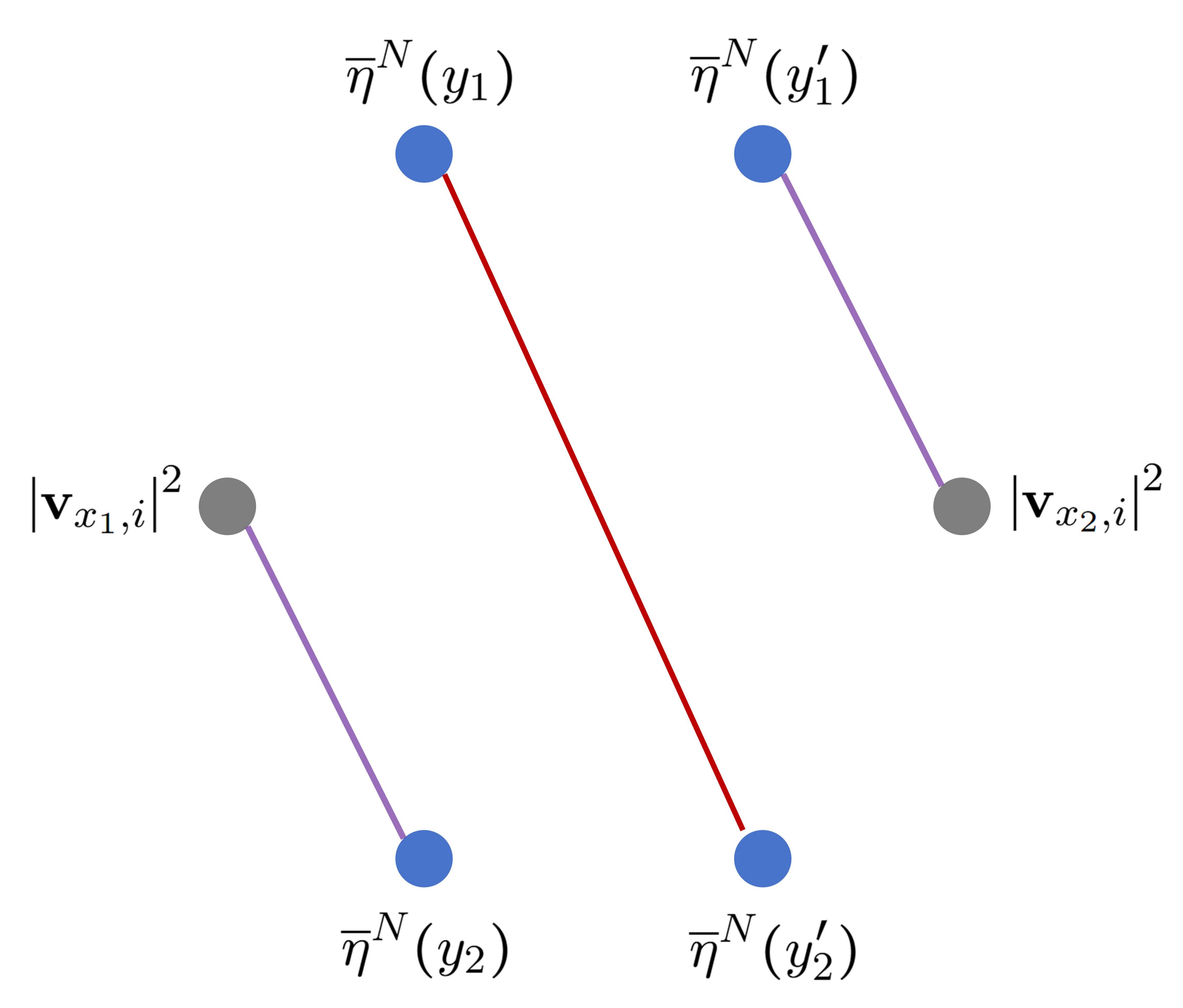}
			\caption*{Case 2: $k=2$.}
		\end{figure}
		
		Therefore, the contribution in \eqref{eq.I_1.I_1.expectation.square} is \begin{align}\label{eq.I_1.I_1.expectation.square.Case2.2}
			&\quad N^{-4d}\sum_{\mathcal E_{2.2}}\E_\rho\left[\bar\eta^N(y_1)\bar\eta^N(y_1')\bar\eta^N(y_2)\bar\eta^N(y_2')\left|\mathbf{v}_{x_1,i}\right|^2\left|\mathbf{v}_{x_2,i}\right|^2\right]\\
			&\leq 2N^{-4d}\sum_{\substack{z_1,z_2\in \mathcal Z_L }}\sum_{i=1}^d\sum_{x_1\in \Lambda_L^{z_1}}\sum_{x_2\in\Lambda_L^{z_2}}\sum_{y_1,y_1'\in\Lambda_L^{z_2}\cup \{x_2+e_i\}}
			\sum_{\substack{y_2,y_2'}}\notag\\
			&\hspace{2cm}\E_\rho\left[\bar\eta^N(y_1)\bar\eta^N(y_1')\left|\mathbf{v}_{x_2,i}\right|^2\right]\E_\rho\left[\bar\eta^N(y_2)\bar\eta^N(y_2')\right]\E_\rho\left[\left|\mathbf{v}_{x_1,i}\right|^2\right]\notag\\
			&\quad+4N^{-4d}\sum_{\substack{z_1,z_2\in \mathcal Z_L }}\sum_{i=1}^d\sum_{x_1\in \Lambda_L^{z_1}}\sum_{x_2\in\Lambda_L^{z_2}}\sum_{y_1\in\Lambda_L^{z_2}\cup \{x_2+e_i\}}\sum_{y_2\in\Lambda_L^{z_1}\cup\{x_1+e_i\}}\sum_{y_1',y_2'}\notag\\
			&\hspace{2cm}\E_\rho\left[\bar\eta^N(y_1)\left|\mathbf{v}_{x_2,i}\right|^2\right]\E_\rho\left[\bar\eta^N(y_2)\left|\mathbf{v}_{x_1,i}\right|^2\right]\E_\rho\left[\bar\eta^N(y_1')\bar\eta^N(y_2')\right]\notag\\
			&\leq CN^{-4d}\left(\frac{N}{L}\right)^{2d}L^{2d}L^dL^dN^d\Vert \phi_L\Vert_\infty^4\notag\\
			&\leq CN^{-d}L^{2d}\Vert \phi_L\Vert_\infty^4\notag.
		\end{align}
            \item $k = 3$: We have a quick observation that if $k=3$, then the only $\bar\eta^N(y)$ without bond is isolated, which makes the expectation zero. Therefore, this case does not contribute.
            
            \item $k = 4$: Viewing \eqref{eq.defE}, $\bar\eta^N(y_1)$ and $\bar\eta^N(y_1')$ cannot connect $\left|\mathbf{v}_{x_1,i}\right|^2$, so they connect $\left|\mathbf{v}_{x_2,i}\right|^2$. By the similar argument, $\bar\eta^N(y_2)$ and $\bar\eta^N(y_2')$ connect  $\left|\mathbf{v}_{x_1,i}\right|^2$. The contribution in $\mathcal E_*$ is
		\begin{equation*}
			\mathcal E_{2.4}=\mathcal E_*\cap\left\{z_1\neq z_2, \ y_1,y_1' \in \Lambda_L^{z_2}\cup\{x_2+e_i\},\ y_2,y_2' \in \Lambda_L^{z_1}\cup\{x_1+e_i\}\right\}.
		\end{equation*}
		
		\begin{figure}[htbp]
			\centering
			\includegraphics[width=0.6\textwidth]{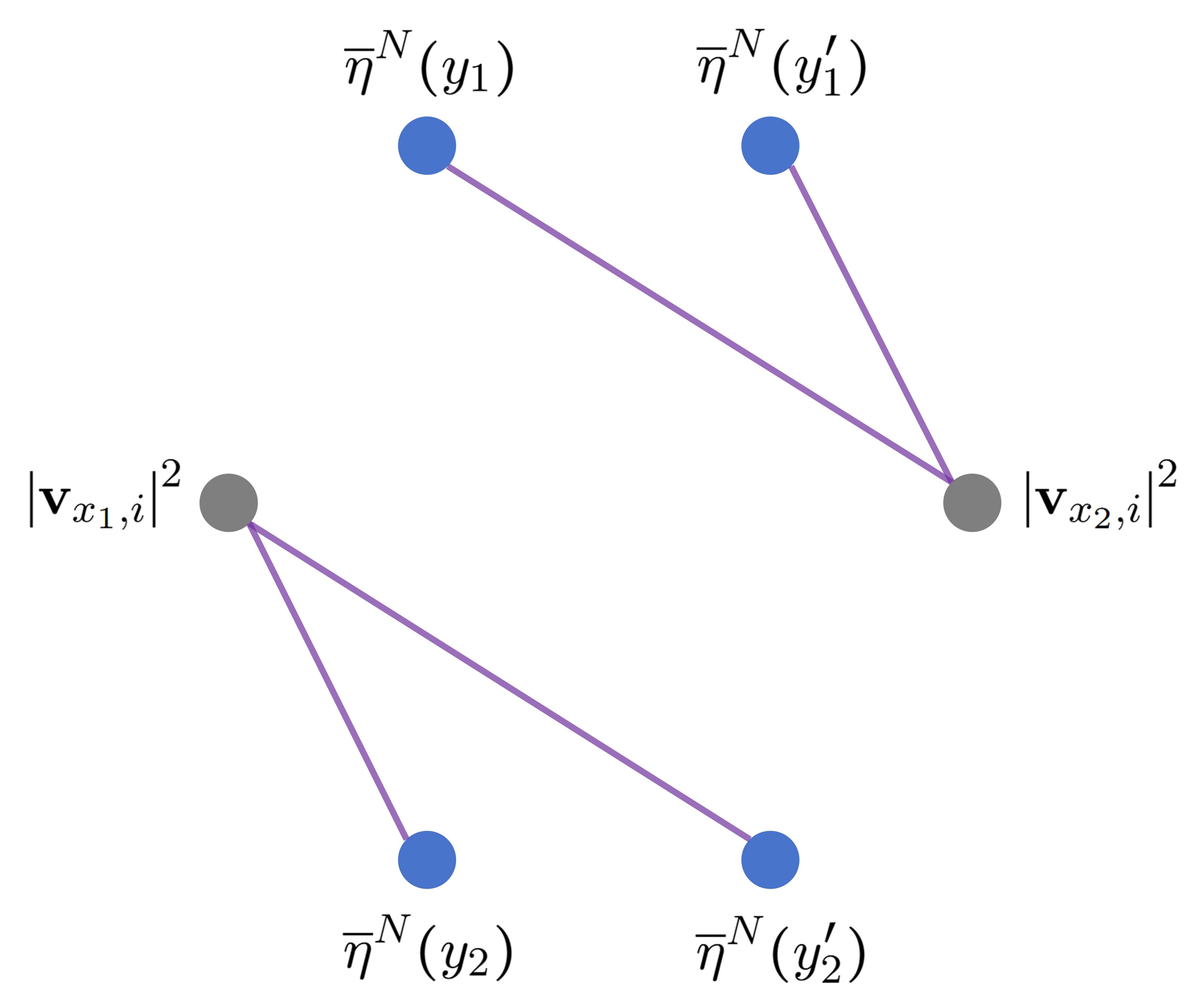}
			\caption*{Case 2: $k=4$.}
		\end{figure}
		
		The contribution in \eqref{eq.I_1.I_1.expectation.square} is
		\begin{align}\label{eq.I_1.I_1.expectation.square.Case2.4}
			&\quad N^{-4d}\sum_{\mathcal E_{2.4}}
			\E_\rho\left[\bar\eta^N(y_1)\bar\eta^N(y_1')\bar\eta^N(y_2)\bar\eta^N(y_2')\left|\mathbf{v}_{x_1,i}\right|^2\left|\mathbf{v}_{x_2,i}\right|^2\right]\\
			&=N^{-4d}\sum_{\substack{z_1,z_2\in \mathcal Z_L }}\sum_{i=1}^d\sum_{x_1\in \Lambda_L^{z_1}}\sum_{x_2\in \Lambda_L^{z_2}}\sum_{\substack{y_1,y_1'\in \Lambda_L^{z_2}\cup\{x_2+e_i\}}}\sum_{\substack{y_2,y_2'\in \Lambda_L^{z_1}\cup\{x_1+e_i\}}}\notag\\
			&\qquad\E_\rho\left[\bar\eta^N(y_1)\bar\eta^N(y_1')\left|\mathbf{v}_{x_2,i}\right|^2\right]\E_\rho\left[\bar\eta^N(y_2)\bar\eta^N(y_2')\left|\mathbf{v}_{x_1,i}\right|^2\right]\notag\\
			&\leq CN^{-4d} \left(\frac{N}{L}\right)^{2d} L^{2d} L^{2d}L^{2d}\Vert \phi_L\Vert_\infty^4\notag\\
			&\leq CN^{-2d}  L^{4d}\Vert \phi_L\Vert_\infty^4\notag.
		\end{align}
        \end{itemize}

		Combining \eqref{eq.I_1.I_1.expectation.square}-\eqref{eq.I_1.I_1.expectation.square.Case2.4}, we conclude that for every test function $f$, 
		\begin{equation*}
			N^{4} \sum_{i=1}^d \sum_{x_1,x_2}\E_\rho\left[{\mathcal I}_1^N(f,x_1,i)^2{\mathcal I}_1^N(f,x_2,i)^2\right]\leq C_{f},
		\end{equation*}
		which finishes the proof of the $L^2$-boundedness for $\mathcal B^N(f)$.
	\end{proof}

	\subsection{Characterization of limit}\label{sec.M.limit}
	In this section, we characterize the limit of $\left\langle{\mathcal M}^N(f)\right\rangle_t$. We first give our result.
	\begin{proposition}[Limit of quadratic variation]\label{prop.limit.quadratic.variation}
		For every $f\in C^{\infty}\left(\T^{2d}\right)$, the limit quadratic variation $\left\langle{\mathcal M}(f)\right\rangle_t$ is characterized by
		\begin{multline*}
			\left\langle{\mathcal M}(f)\right\rangle_t=\int_0^t\int_{\T^d}\left\{\Y_s\left(\nabla_{1}f(x,\cdot)\right)+\Y_s\left(\nabla_{2}f(\cdot,x)\right)\right\}\\\cdot \cc(\rho)\left\{\Y_s\left(\nabla_{1}f(x,\cdot)\right)+\Y_s\left(\nabla_{2}f(\cdot,x)\right)\right\}\,\d x\,\d s.
		\end{multline*}
	\end{proposition}
	\begin{proof}
		We first calculate for test function $f$, and then extend to $f\in C^{\infty}\left(\T^{2d}\right)$. 
		
		\textbf{Step 1}:  Replace $\mathcal B^N(f)$ by $\mathcal {\tilde B}^N(f)$ with a small error. Here $\mathcal {\tilde B}^N(f)$ reads
		\begin{equation*}
			\tilde {\mathcal B}^N(f)=N^{2}\sum_{i=1}^d\sum_x\\c_{x,x+e_i}
			\mathcal I_2^N(f,x,i)^2,
		\end{equation*}
		where $\mathcal I_2^N(f,x,i)$ is a replacement for $\pi_{x,x+e_i}\mathcal Q^N(f)$:
		\begin{equation*}
			\mathcal I_2^N(f,x,i):=2N^{-1-d}\sum_{y}\bar\eta^N(y)\nabla^N_{1}f\left(\frac{x}{N},\frac{y}{N}\right)\cdot\mathbf{v}_{x,i}=2N^{-1-\frac{d}{2}}\Y^N\left(\nabla^N_{1,i}f\left(\frac{x}{N},\frac{\cdot}{N}\right)\right)\cdot\mathbf{v}_{x,i}.
		\end{equation*}
		Denote the error between $\mathcal I_2^N(f,x,i)$ and $\pi_{x,x+e_i}\mathcal Q^N(f)$ by 
		\begin{equation*}
			\mathcal R_{2}^N(f,x,i):=\mathcal I_2^N(f,x,i)-\pi_{x,x+e_i}\mathcal Q^N(f).
		\end{equation*}  
		We calculate the second moment of $\mathcal R_2^N$ to be small in Lemma~\ref{lemma.R_2.expectation}, so that the replacement holds. The calculation is in the appendix. After these preparations, we replace $\mathcal B^N(f)$ by $\tilde{\mathcal B}^N(f)$ and calculate the error:
		\begin{align*}
			&\quad\E_\rho\left[\left|\mathcal B^N(f)-\tilde{\mathcal B}^N(f)\right|\right]\\
			&\leq N^2\sum_{i=1}^d\sum_x\E_{\rho}\left[\left|\pi_{x,x+e_i}\mathcal Q^N(f)^2-\mathcal I_2(f,x,i)^2\right|\right]\notag\\
			&= N^2\sum_{i=1}^d\sum_x\E_{\rho}\Big[\Big|-2\pi_{x,x+e_i}\mathcal Q^N(f)\mathcal R_{2}^N(f,x,i)-\mathcal R_{2}^N(f,x,i)^2\Big|\Big]\notag\\
			&\leq 2N^2\sum_{i=1}^d\sum_x\E_{\rho}\left[\left|\pi_{x,x+e_i}\mathcal Q^N(f)\mathcal R_{2}^N(f,x,i)\right|\right]+N^2\sum_{i=1}^d\sum_x\E_{\rho}\left[\mathcal R_{2}^N(f,x,i)^2\right]\notag,
		\end{align*}
		where 
		\begin{align*}
			&\quad N^2\sum_{i=1}^d\sum_x\E_{\rho}\left[\left|\pi_{x,x+e_i}\mathcal Q^N(f)\mathcal R_{2}^N(f,x,i)\right|\right]\\
			&\leq N^2\sum_{i=1}^d\sum_x\E_{\rho}\left[\left(\pi_{x,x+e_i}\mathcal Q^N(f)\right)^2\right]^\frac{1}{2}\E_{\rho}\left[\mathcal R_{2}^N(f,x,i)^2\right]^{\frac{1}{2}}\notag\\
			&\leq \E_{\rho}\left[N^2\sum_{i=1}^d\sum_x\left(\pi_{x,x+e_i}\mathcal Q^N(f)\right)^2\right]^\frac{1}{2}\E_{\rho}\left[N^2\sum_{i=1}^d\sum_x\mathcal R_{2}^N(f,x,i)^2\right]^{\frac{1}{2}}\notag\\
			&\leq \E_{\rho}\left[\mathcal B^N(f)\right]^\frac{1}{2}\E_{\rho}\left[N^2\sum_{i=1}^d\sum_x\mathcal R_{2}^N(f,x,i)^2\right]^{\frac{1}{2}}\notag.
		\end{align*}
		The $L^1$-boundedness of $\mathcal B^N(f)$ has been shown in Proposition~\ref{prop.B.bound.L1.L2}. Therefore, the replacement is done as long as we prove that 
		\begin{equation*}
			N^2\sum_{i=1}^d\sum_x\E_{\rho}\left[\mathcal R_{2}^N(f,x,i)^2\right]\leq C_{f}N^{-2\gamma},
		\end{equation*}
		for some $\gamma> 0$, which is given by Lemma~\ref{lemma.R_2.expectation}, with an error given by 
        \begin{equation}\label{eq.error for the replacement of B}
            \E_\rho\left[\left|\mathcal B^N(f)-\tilde{\mathcal B}^N(f)\right|\right]\leq C_f N^{-\gamma}.
        \end{equation}

		\textbf{Step 2}: Characterize the limit of $\tilde{\mathcal B}^N(f)$.
		Unlike the standard fluctuation field setting, the test function
        \[
        \nabla_{1,i}^N f\Big(\frac{x}{N},\frac{y}{N}\Big)
        \]
        still depends on the microscopic variable $x$. Therefore, 
        \[
        \mathcal Y^N\Big(\nabla_{1,i}^N f\Big(\frac{x}{N},\frac{y}{N}\Big)\Big)
        \]
        is not constant in $x$, and a direct law of large numbers argument is not applicable.  To overcome this barrier, we need to take advantage of the continuity of $x$ with respect to $\nabla^N_{1,i}f(x/N,\cdot/N)$. We take $K(N)=o(N)$, with $L(N)\ll K(N)\ll N$. Without loss of generality, we assume that $L$ is a factor of $K$ and $K$ is a factor of $N$. For $a \in \Lambda_K$ and $1\leq i\leq d$, we have 
		\[
		\sup_{x,y} \left|\nabla^N_{1,i}f\left(\frac{x}{N},\frac{y+a}{N}\right)-\nabla^N_{1,i}f\left(\frac{x}{N},\frac{y}{N}\right)\right|\leq N^{-1}K,
		\]
		where we used the uniform $L^2$-boundedness of the fluctuation field $\Y^N$. This gives for $a \in \Lambda_K$, and $1\leq j_1,j_2\leq d$,
		\begin{align*}
			&\E_{\rho}\bigg[\bigg|\Y^N\left(\nabla^N_{1,j_1}f\left(\frac{x}{N},\frac{\cdot+a}{N}\right)\right)\Y^N\left(\nabla^N_{1,j_2}f\left(\frac{x}{N},\frac{\cdot+a}{N}\right)\right)\\
			&\hspace{5cm}-\Y^N\left(\nabla^N_{1,j_1}f\left(\frac{x}{N},\frac{\cdot}{N}\right)\right)\Y^N\left(\nabla^N_{1,j_2}f\left(\frac{x}{N},\frac{\cdot}{N}\right)\right)\bigg|\bigg]\\
			&\leq \E_{\rho}\bigg[\bigg|\Y^N\left(\nabla^N_{1,j_1}f\left(\frac{x}{N},\frac{\cdot+a}{N}\right)\right)\left(\Y^N\left(\nabla^N_{1,j_2}f\left(\frac{x}{N},\frac{\cdot+a}{N}\right)\right)-\Y^N\left(\nabla^N_{1,j_2}f\left(\frac{x}{N},\frac{\cdot}{N}\right)\right)\right)\bigg\vert\bigg]\\
			&\ \ +\E_\rho\bigg[\bigg\vert\left(\Y^N\left(\nabla^N_{1,j_1}f\left(\frac{x}{N},\frac{\cdot+a}{N}\right)\right)-\Y^N\left(\nabla^N_{1,j_1}f\left(\frac{x}{N},\frac{\cdot+a}{N}\right)\right)\right)\Y^N\left(\nabla^N_{1,j_2}f\left(\frac{x}{N},\frac{\cdot}{N}\right)\right)\bigg|\bigg]\\
			&\leq C_fN^{-1}K.
		\end{align*}
		The above result shows that we can replace
		\[
		\Y^N\left(\nabla^N_{1,j_1}f\left(\frac{x}{N},\frac{\cdot}{N}\right)\right)\Y^N\left(\nabla^N_{1,j_2}f\left(\frac{x}{N},\frac{\cdot}{N}\right)\right)
		\]
		by the average 
		\[
		\frac{1}{|\Lambda_K|}\sum_{a\in \Lambda_K}\Y^N\left(\nabla^N_{1,j_1}f\left(\frac{x}{N},\frac{\cdot+a}{N}\right)\right)\Y^N\left(\nabla^N_{1,j_2}f\left(\frac{x}{N},\frac{\cdot+a}{N}\right)\right),
		\]
		with an error up to $C_fN^{-1}K$ in $L^1$. Passing the summation to $x$, we obtain
		\begin{multline}\label{ineq.average.B}
			\E_\rho\left[\mathcal {\tilde B}^N(f)-\frac{4}{\vert \Lambda_K\vert}N^{-d}\sum_{i=1}^d\sum_x\sum_{a\in \Lambda_K^x}c_{a,a+e_i}
			\Bigg\{\Y^N\left(\nabla^N_{1}f\left(\frac{x}{N},\frac{\cdot}{N}\right)\right)
			\cdot\mathbf{v}_{a,i}\Bigg\}^2\right]\\
			\leq C_fN^{-1}K\Vert \phi_L\Vert_\infty^2.
		\end{multline}
        Our next discussion is to replace $c_{a,a+e_i}\mathbf{v}_{a,i,j_1}\mathbf{v}_{a,i,j_2}$ by its expectation. Now we need to calculate the summation $\sum_{a\in \Lambda_K^x}$ by groups. Recall that $c_{y,y'}(\eta)$ depends only on $\{\eta_z: \vert z - y\vert \leq \r\}$ for some integer $\r > 0$. When $N$ is large enough such that $L(N)> 2\mathbf{r}$, we sum over $\mathcal K_n$, $1\leq n\leq (2L)^d$, where 
		\begin{align*}
			\mathcal K_n:=\left\{a\in \Lambda_K^x: a-a_n\in 2L \mathbb Z^d\right\}.
		\end{align*}
		and $a_n\in \Lambda_{2L}^x$ such that $a_n\neq a_m$ if $n\neq m$. By construction, if \(a,a'\in \mathcal K_n\) and \(a\neq a'\), then
        \[
        (a+\Lambda_L)\cap(a'+\Lambda_L)=\varnothing ,
        \]
        and the corresponding random variables are independent under \(\P_\rho\). Therefore, We can find that in every group $\mathcal K_n, 1\leq n\leq (2L)^d$, $c_{a,a+e_i}\mathbf{v}_{a,i,j_1}\mathbf{v}_{a,i,j_2}$ are independent of each other for different $a\in \mathcal K_n$. Using
		\[
		\vert x_1+x_2+\cdots+x_m\vert^{2p}\leq m^{2p-1}(|x_1|^{2p}+|x_2|^{2p}+\cdots|x_m|^{2p}),\quad p>\frac{1}{2},
		\]
		we obtain 
		\begin{align}\label{ineq.group sum}
			&\quad \E_{\rho}\left[\left|\frac{1}{|\Lambda_K\vert}\sum_{a\in \Lambda_K^x}\left(c_{a,a+e_i}\mathbf{v}_{a,i,j_1}\mathbf{v}_{a,i,j_2}-\E_\rho\left[c_{a,a+e_i}\mathbf{v}_{a,i,j_1}\mathbf{v}_{a,i,j_2}\right]\right)\right|^{2p}\right]\\
			&\leq \frac{(2L)^{d(2p-1)}}{|\Lambda_K\vert^{2p}}\sum_{1\leq n\leq (2L)^{d}}\E_{\rho}\left[\left|\sum_{a\in \mathcal J_n}\left(c_{a,a+e_i}\mathbf{v}_{a,i,j_1}\mathbf{v}_{a,i,j_2}-\E_\rho\left[c_{a,a+e_i}\mathbf{v}_{a,i,j_1}\mathbf{v}_{a,i,j_2}\right]\right)\right|^{2p}\right]\notag.
		\end{align}
		Then, by Burkholder's inequality for sums of independent centered random variables, we have the following estimate
		\begin{align}\label{ineq.Burkholder}
			&\ \E_{\rho}\left[\left|\sum_{a\in \mathcal K_n}\left(c_{a,a+e_i}\mathbf{v}_{a,i,j_1}\mathbf{v}_{a,i,j_2}-\E_\rho\left[c_{a,a+e_i}\mathbf{v}_{a,i,j_1}\mathbf{v}_{a,i,j_2}\right]\right)\right|^{2p}\right]\\\leq& \ C_p\left\{\sum_{a\in \mathcal K_n}\E_{\rho}\left[\left(c_{a,a+e_i}\mathbf{v}_{a,i,j_1}\mathbf{v}_{a,i,j_2}-\E_\rho\left[c_{a,a+e_i}\mathbf{v}_{a,i,j_1}\mathbf{v}_{a,i,j_2}\right]\right)^{2}\right]\right\}^p\notag\\
			\leq& \ C_p\vert \mathcal K_n\vert^p\Vert \phi_L\Vert_\infty^{4p}\notag,
		\end{align}
        where we used the uniform boundedness of $c$ and $\mathbf{v}$ in the last step.
		Combining \eqref{ineq.group sum} and \eqref{ineq.Burkholder}, we have a good estimate for the average term
		\begin{equation}\label{ineq.average.G}
			\E_{\rho}\left[\left|\frac{1}{|\Lambda_K\vert}\sum_{a\in \Lambda_K}\left(c_{a,a+e_i}\mathbf{v}_{a,i,j_1}\mathbf{v}_{a,i,j_2}-\E_\rho\left[c_{a,a+e_i}\mathbf{v}_{a,i,j_1}\mathbf{v}_{a,i,j_2}\right]\right)\right|^{2p}\right]\leq C_{p} L^pK^{-p}\Vert \phi_L\Vert_\infty^{4p}.
		\end{equation}
		Now, taking \(p=1\) in \eqref{ineq.average.G}, by Schwarz's inequality, we have
		\begin{align}\label{ineq.L1.martingale}
			&\ \E_\rho\Bigg[\Bigg|\frac{1}{|\Lambda_K\vert}\sum_{a\in \Lambda_K^x}\Y^N\left(\nabla^N_{1,j_1}f\left(\frac{x}{N},\frac{\cdot}{N}\right)\right)\Y^N\left(\nabla^N_{1,j_2}f\left(\frac{x}{N},\frac{\cdot}{N}\right)\right)\\&\hspace{5cm}\left(c_{a,a+e_i}\mathbf{v}_{a,i,j_1}\mathbf{v}_{a,i,j_2}-\E_\rho\left[c_{a,a+e_i}\mathbf{v}_{a,i,j_1}\mathbf{v}_{a,i,j_2}\right]\right)\Bigg|\Bigg]\notag\\\leq& \ \E_{\rho}\left[\left|\frac{1}{|\Lambda_K\vert}\sum_{a\in \Lambda_K^x}\left(c_{a,a+e_i}\mathbf{v}_{a,i,j_1}\mathbf{v}_{a,i,j_2}-\E_\rho\left[c_{a,a+e_i}\mathbf{v}_{a,i,j_1}\mathbf{v}_{a,i,j_2}\right]\right)\right|^{2}\right]^{\frac{1}{2}}\notag\\&\hspace{4cm}\E_\rho\left[\left|\Y^N\left(\nabla^N_{1,j_1}f\left(\frac{x}{N},\frac{\cdot}{N}\right)\right)\Y^N\left(\nabla^N_{1,j_2}f\left(\frac{x}{N},\frac{\cdot}{N}\right)\right)\right|^2\right]^{\frac{1}{2}}\notag\\ \leq&\  C LK^{-1}\Vert \phi_L\Vert_\infty^{2}\notag.
		\end{align}
		Based on \eqref{ineq.average.B} and \eqref{ineq.L1.martingale}, we obtain an improved replacement for $\mathcal{\tilde B}^N(f)$:
		\begin{multline}\label{eq.better replacement}
			\E_{\rho}\Bigg[\Bigg|\mathcal {\tilde B}^N(f)-\frac{4}{\vert \Lambda_K\vert}N^{-d}\sum_{i=1}^d\sum_{x}\sum_{a\in \Lambda_K^x}\sum_{j_1,j_2=1}^d\E_\rho\left[c_{a,a+e_i}\mathbf{v}_{a,i,j_1}\mathbf{v}_{a,i,j_2}\right]\\
			\Y^N\left(\nabla^N_{1,j_1}f\left(\frac{x}{N},\frac{\cdot}{N}\right)\right)\Y^N\left(\nabla^N_{1,j_2}f\left(\frac{x}{N},\frac{\cdot}{N}\right)\right)\Bigg|\bigg]\leq C\left(LK^{-1}+KN^{-1}\right)\Vert \phi_L\Vert_\infty^2.
		\end{multline}
		Finally, we calculate the averaged expectation for each $x$:
		\begin{align*}
			\frac{1}{\vert \Lambda_K\vert}\sum_{i=1}^d\sum_{a\in \Lambda_K^x}\E_\rho\left[c_{a,a+e_i}\mathbf{v}_{a,i,j_1}\mathbf{v}_{a,i,j_2}\right].
		\end{align*}
		Since $L(N)\ll K(N)$ and \(K\) is a multiple of \(L\), we partition \(\Lambda_K\) into disjoint boxes of side length \(L\), centered at $z\in L\mathbb Z^d\cap \Td_N$. If a small box is totally contained in $a\in \Lambda_K^{x}$, then we do the summation inside the box to get
		\begin{equation*}
			\quad\frac{1}{\vert \Lambda_L\vert}\sum_{i=1}^d\sum_{a\in \Lambda_L^z}\E_\rho\left[c_{a,a+e_i}\mathbf{v}_{a,i,j_1}\mathbf{v}_{a,i,j_2}\right]=\cc_{j_1j_2}(\rho)+O(L^{-\gamma_1}),\quad \gamma_1>0,
		\end{equation*}
		where the last step follows from the quantitative homogenization estimate in \cite[Theorem~1.5]{FGW24}. The number of boxes that are totally contained in $\Lambda_K^x$ is at least $\left(\frac{K}{L}-1\right)^d$. If a small box is not totally contained in $\Lambda_K^{x}$, then we use the trivial bound to estimate. Above all, we calculate 
		\begin{align}\label{eq.error in Lambda K}
			&\quad\frac{1}{\vert \Lambda_K\vert}\sum_{i=1}^d\sum_{a\in \Lambda_K^{x}}\E_\rho\left[c_{a,a+e_i}\mathbf{v}_{a,i,j_1}\mathbf{v}_{a,i,j_2}\right]-\cc_{j_1j_2}(\rho)\\
			&=\frac{1}{\vert \Lambda_K\vert}\left(\sum_{a\in \Lambda_K^{x}}\sum_{i=1}^d\E_\rho\left[c_{a,a+e_i}\mathbf{v}_{a,i,j_1}\mathbf{v}_{a,i,j_2}\right]-\cc_{j_1j_2}(\rho)\right)\notag\\
			&\leq \frac{1}{\vert \Lambda_K\vert} \left\{\left(\frac{K}{L}-1\right)^dL^dL^{-\gamma_1}+\left(\left(\frac{K}{L}\right)^d-\left(\frac{K}{L}-1\right)^d\right)L^d\Vert \phi_L\Vert_\infty^2\right\}\notag\\
			&\leq C\left(L^{-\gamma_1}+K^{-1}L\Vert \phi_L\Vert_\infty^2\right)\notag.
		\end{align}
        Combining \eqref{eq.better replacement} and \eqref{eq.error in Lambda K}, for every test function $f$, there exists $\gamma >0$ such that 
        \begin{equation*}
			\E_{\rho}\Bigg[\Bigg|\mathcal {\tilde B}^N(f)-4N^{-d}\sum_{x}
			\Y^N\left(\nabla^N_{1}f\left(\frac{x}{N},\frac{\cdot}{N}\right)\right)\cdot \mathbf c(\rho)\Y^N\left(\nabla^N_{1}f\left(\frac{x}{N},\frac{\cdot}{N}\right)\right)\Bigg|\bigg]\leq C_fN^{-\gamma}.
		\end{equation*}
		Including the first step of replacement in \eqref{eq.error for the replacement of B}, we have for test function $f$:
		\begin{equation*}
			\E_{\rho}\Bigg[\Bigg|\mathcal B^N(f)-4N^{-d}\sum_{x}
			\Y^N\left(\nabla^N_{1}f\left(\frac{x}{N},\frac{\cdot}{N}\right)\right)\cdot \mathbf c(\rho)\Y^N\left(\nabla^N_{1}f\left(\frac{x}{N},\frac{\cdot}{N}\right)\right)\Bigg|\bigg]\leq C_fN^{-\gamma}.
		\end{equation*}
		Therefore, considering the integral of time, we have a general result for every test function $f$ and $t\in [0,T]$:
		\begin{multline*}
			\E_{\rho}\Bigg[\Bigg|\left\la\mathcal M^N(f)\right\ra_t-\int_0^t4N^{-d}\sum_{x}
			\Y_s^N\left(\nabla^N_{1}f\left(\frac{x}{N},\frac{\cdot}{N}\right)\right)
            \cdot \mathbf c(\rho)\Y_s^N\left(\nabla^N_{1}f\left(\frac{x}{N},\frac{\cdot}{N}\right)\right)\, \d s\Bigg|\Bigg]\\\leq C_fN^{-\gamma}.
		\end{multline*}
        As $N\to\infty$, we have the weak convergence in $D([0, T], \mathcal S'(\Td))$
        \begin{align*}
			(\mathcal Y_t^N)_{t\in [0,T]}\xRightarrow{N\to\infty}(\mathcal Y_t)_{t\in [0,T]},  
		\end{align*}
        and let this limit live in the same space as $\mathcal M$. Using the convergence from discrete derivative to continuous derivative, and taking the integral with respect to time and space, the following weak convergence holds in $D([0, T], \mathbb{R})$:
        \begin{multline*}
            \int_0^tN^{-d}\sum_{x}
			\Y_s^N\left(\nabla^N_{1}f\left(\frac{x}{N},\frac{\cdot}{N}\right)\right)
            \cdot \mathbf c(\rho)\Y_s^N\left(\nabla^N_{1}f\left(\frac{x}{N},\frac{\cdot}{N}\right)\right)\, \d s \\\xRightarrow{N\to\infty} \int_0^t\int_{\T^d}\Y_s\left(\nabla_{1}f(x,\cdot)\right)\cdot \cc(\rho)\Y_s\left(\nabla_{1}f(x,\cdot)\right)\,\d x\,\d s.
        \end{multline*}
        Therefore, using \eqref{eq.M.Nk.joint}, every limit point of $\left\la\mathcal M^N(f)\right\ra_t$ along subsequence is characterized as follows: 
        \begin{equation}\label{eq.characterization.first version}
            \begin{split}
                \left\langle{\mathcal M}(f)\right\rangle_t  &= \lim_{N \to \infty} \la\mathcal M^{N}\ra_t\\
                &=4\int_0^t\int_{\T^d}\Y_s\left(\nabla_{1}f(x,\cdot)\right)\cdot \cc(\rho)\Y_s\left(\nabla_{1}f(x,\cdot)\right)\,\d x\,\d s.
            \end{split}
		\end{equation}

        For general $f\in C^{\infty}\left(\T^{2d}\right)$, consider $f=f_{\textit{sym}}+f_{\textit{asym}}$, where for $u,v\in \Td$,
		\[
		f_{\textit{sym}}(u,v)=\frac{1}{2}(f(u,v)+f(v,u)),
		\]
		and 
		\[
		f_{\textit{asym}}(u,v)=\frac{1}{2}(f(u,v)-f(v,u)).
		\]
		When we repeat the proof for general $f$, the same result holds for $f_{\textit{sym}}$, and all the terms for $f_{\textit{asym}}$ are $0$. We can write in a more symmetric form:
        \begin{equation}\label{eq.symmetric}
            \Y_s\left(\nabla_{1}f_{\textit{sym}}(x,\cdot)\right)= \frac{1}{2}\left(\Y_s\left(\nabla_{1}f(x,\cdot)\right)+\Y_s\left(\nabla_{1}f(\cdot,x)\right)\right).
        \end{equation}
        We combine \eqref{eq.characterization.first version} and \eqref{eq.symmetric} to obtain the desired characterization of the limit quadratic variation.
	\end{proof}

	\section{Drift}\label{section.drift}
	In this section, we give the tightness for the drift term $\mathcal A^N$ and give the characterization of its limit as $N\to\infty$.
	Recall that for all $f\in C^\infty(\T^{2d})$, the drift term ${\mathcal A}_t^N(f)$ has the form
	\begin{equation*}
		\mathcal A_t^N(f)=\int_0^t \mathcal{L}_{N}\mathcal{Q}_s^N(f)\,\d s,   
	\end{equation*}
	where $\mathcal{L}_{N}\mathcal{Q}_s^N(f)$ equals
	\begin{equation*}
		\mathcal{L}_{N}\mathcal{Q}_t^N(f)=N^{2}\sum_{i=1}^d \sum_{x}c_{x,x+e_i}\pi_{x,x+e_i}\mathcal{Q}_t^N(f).
	\end{equation*}
    The following result is the main object in this section.  
    \begin{proposition}[Convergence of the drift term]\label{prop.Characterization for the convergence of the drift term}
		For every function $f\in C^\infty\left(\mathbb{T}^{2d}\right)$, the sequence of $\{\mathcal A^N(f)\}_{N \in \N}$ admits a limit $\mathcal A(f)$ as $N\to\infty$ in $D\left([0,T],\R\right)$. Every limit $\mathcal A(f)$ is in $C\left([0,T],\R\right)$ and satisfies that 
		\begin{equation*}
			\mathcal A_t(f)=\int_0^t \mathrm{Tr}\left(\mathbf D(\rho)Q_s(\partial^2 _{1,2}f)\right)\d s,
		\end{equation*}
        where $Q$ is the associated limit of $\{Q^{N}\}_{N \in \N}$.  
	\end{proposition}
    We will justify the tightness in Section~\ref{sec.A.tight} and then give the characterization of its limit as $N\to\infty$ in Section~\ref{sec.A.limit}.
	\subsection{Tightness}\label{sec.A.tight}
	In this section, we prove the tightness for the drift term. The proof of tightness relies on two tools: one is Kipnis--Varadhan lemma; the other is the replacement lemma for the drift term. We postpone the replacement lemma in the next section and state first Kipnis--Varadhan lemma. Here we define the Sobolev norm $H^{k}$ to be 
	\begin{equation*}
		\Vert G\Vert_{H^k}^2:=\E_{\rho}\left[G (-\L_N)^{k}G\right],\qquad k\in\mathbb Z.
	\end{equation*}
	\begin{lemma}[\cite{komorowski2012fluctuations}, Kipnis--Varadhan]\label{Kipnis--Varadhan}
		For fixed $t>0$, there exists a uniform constant $C$ such that for centered $ G\in L^2\cap H^{-1}$, we have
		\begin{equation*}
			\E_\rho\left[\sup_{0\leq s\leq t}\left(\int_0^s G\left(\eta^N_\kappa\right)\, \d \kappa\right)^2\right]\leq Ct \Vert G\Vert_{H^{-1}}^2.
		\end{equation*}
	\end{lemma}

	Here we state our result for the tightness of the sequence of drift terms.
	\begin{lemma}\label{lemma.tightness.A}
		For every test function  $f$, the sequence of drift terms
        \[
        \left\{\mathcal A_t^N(f), t\in [0,T]\right\}_{N \in \N}
        \]
        is tight in $D([0, T], \mathbb{R})$ and admits a limit in $C([0, T], \mathbb{R})$.
	\end{lemma}
	\begin{proof}
		By Lemma~\ref{Kipnis--Varadhan}, we have
		\begin{align*}
			\Er\Ll[\mathcal A^N_t(f)^2\Rr] = \E_{\rho}\left[\left\vert\int_0^t \mathcal{L}_{N}\mathcal{Q}_s^N(f)\,\d s\right\vert^2\right]
			\leq Ct\left\Vert \mathcal{L}_{N}\mathcal{Q}_0^N(f)\right\Vert_{H^{-1}}^2,
		\end{align*}
		and we can use Proposition~\ref{prop.B.bound.L1.L2} to conclude its bound
		\begin{align*}
			\left\Vert \mathcal{L}_{N}\mathcal{Q}_0^N(f)\right\Vert_{H^{-1}}^2&=\E_{\rho}\left[\left(-\mathcal L_N \mathcal{Q}_0^N(f)\right) \left(-\mathcal L_N \right)^{-1}\left(-\mathcal L_N \mathcal{Q}_0^N(f)\right)\right] \\
			&=\E_{\rho}\left[\mathcal{Q}_0^N(f)\left(-\mathcal L_N\right) \mathcal{Q}_0^N(f)\right] \\
			&=\E_\rho\left[\mathcal B_0^N(f)\right]\leq C_{f}.
		\end{align*}
		Thus $\limsup_{N}\Er\Ll[(\mathcal A^N_t(f))^2\Rr]$ is bounded, which justifies the first condition in Proposition~\ref{Prop.Tightness.real-valued}.

		To verify the second condition in Proposition~\ref{Prop.Tightness.real-valued}, we apply Chebyshev’s inequality.
		\begin{equation}\label{ineq.drift.Chebyshev}
			\P_{\rho}\left[\omega\left(\mathcal A^N(f),r\right)\geq \varepsilon\right]\leq \frac{1}{\varepsilon^2}\E_{\rho}\left[\omega\left(\mathcal A^N(f),r\right)^2\right],
		\end{equation} 
		where
		\begin{align}\label{ineq.drift.Cauchy-Scharz}
			&\omega\left(\mathcal A^N(f),r\right)^2 \notag\\
			&=\sup_{s,t\in[0,T], \vert t-s\vert\leq r} \left\vert\int_s^t \mathcal{L}_{N}\mathcal{Q}_{\kappa}^N(f)\,\d\kappa\right\vert^2\\
			&\leq 2\sup_{s,t\in[0,T], \vert t-s\vert\leq r}\left(\left\vert\int_s^t \mathcal{L}_{N}\mathcal{Q}_{\kappa}^N(f)-Q_\kappa^N\left(\Delta_\mathbf{D}^Nf\right)\,\d\kappa\right\vert^2+\left\vert\int_s^t Q_\kappa^N\left(\Delta_\mathbf{D}^Nf\right)\,\d\kappa\right\vert^2\notag\right).
		\end{align}
		The part $Q^N\left(\Delta_\mathbf{D}^Nf\right)$ is defined as  
		\begin{equation*}
			Q^N\left(\Delta_\mathbf{D}^Nf\right):=\sum_{i,j=1}\mathbf D_{ij}(\rho) Q^N\left(\nabla_{1,j}^N\nabla_{1,i}^Nf+\nabla_{2,j}^N\nabla_{2,i}^Nf\right).
		\end{equation*}
		We split the modulus into two parts. The first term can be estimated by the inequality of Kipnis--Varadhan in Lemma~\ref{Kipnis--Varadhan},
		\begin{align}\label{eq.drift.tightness.1}
			&\E_\rho\left[\sup_{s,t\in[0,T], \vert t-s\vert\leq r}\left\vert\int_s^t \mathcal{L}_{N}\mathcal{Q}_{\kappa}^N(f)-Q_\kappa^N\left(\Delta_\mathbf{D}^Nf\right)\,\d\kappa\right\vert^2\right]\\
			&\leq C \sum_{s\in\{0,r,\cdots,\lfloor \frac{T}{r} \rfloor r\}}\E_\rho\left[\sup_{\substack{t\in[0,T], t-s\leq r \\}}\left\vert\int_s^t \mathcal{L}_{N}\mathcal{Q}_{\kappa}^N(f)-Q_\kappa^N\left(\Delta_\mathbf{D}^N f\right)\,\d\kappa\right\vert^2\right]\notag\\
            &\leq C \Ll(\frac{T}{r}\Rr) r \norm{ \mathcal{L}_{N}\mathcal{Q}_{0}^N(f)-Q_0^N\left(\Delta_\mathbf{D}^N f\right) }_{H^{-1}}\notag\\
			&\leq C_f \Ll(\frac{T}{r}\Rr)CrN^{-\gamma} \leq C_fTN^{-\gamma}\notag.
		\end{align}
        A replacement argument is needed in the passage from the third line to the forth line, which will be proved in Proposition~\ref{prop.replacement} of the next section.

		The second term of \eqref{ineq.drift.Cauchy-Scharz} is estimated directly using \(L^2\)-bounds:
		\begin{align}\label{eq.drift.tightness.2}
			&\E_\rho\left[\sup_{s,t\in[0,T], \vert t-s\vert\leq r}\left\vert\int_s^t Q_\kappa^N\left(\Delta_\mathbf{D}^Nf\right)\,\d\kappa\right\vert^2\right]\\
			&\leq \E_\rho\left[\sup_{s,t\in[0,T], \vert t-s\vert\leq r}r\int_s^t \left\vert Q_\kappa^N\left(\Delta_\mathbf{D}^Nf\right)\right\vert^2\,\d\kappa\right] \notag\\
			&\leq r\E_\rho\left[\int_0^T \left\vert Q_\kappa^N\left(\Delta_\mathbf{D}^Nf\right)\right\vert^2\,\d\kappa\right]\notag\\
			&=rT\E_\rho\left[\left\vert Q_\kappa^N\left(\Delta_\mathbf{D}^Nf\right)\right\vert^2\right]\notag\\
			&\leq C_{f}rT\notag.
		\end{align}
		Here from the first line to the second line, we use Cauchy--Schwarz inequality. Combining \eqref{ineq.drift.Chebyshev}-\eqref{eq.drift.tightness.2}, we have
		\begin{align*}
			\P_{\rho}\left[\omega\left(\mathcal A^N(f),r\right)\geq \varepsilon\right]
			&\leq \frac{1}{\varepsilon^2}\E_{\rho}\left[\omega\left(\mathcal A^N(f),r\right)^2\right]\notag\\
			&\leq \frac{C_{f}}{\varepsilon^2}\left(T N^{-\gamma}+rT\right),
		\end{align*}
		which completes the proof as 
		\begin{align*}
			\inf_{r>0}\limsup_{N\to\infty}\P\left[\omega\left(\mathcal A^N,r\right)\geq \varepsilon\right] \leq \inf_{r>0}\limsup_{N\to\infty}\frac{C_{f}}{\varepsilon^2}\left(T N^{-\gamma}+rT\right)=0.
		\end{align*}
	\end{proof}

	After showing the tightness for $\left\{\mathcal A_t^N(f), t\in [0,T]\right\}_{N \in \N}$, we want to prove that the corrected process $\left\{Z_t^N(f), t\in [0,T]\right\}_{N \in \N}$ is tight and tends to zero process as $N\to\infty$. We state it as the following lemma:
    \begin{lemma}\label{lemma.Z.convergence}
        For every test function  $f$, the sequence of corrected process
        \[
        \left\{Z_t^N(f), t\in [0,T]\right\}_{N \in \N}
        \]
        is tight in $D([0, T], \mathbb{R})$ and the limit is zero process.
    \end{lemma}
	\begin{proof}
	    Lemma~\ref{lemma.Z} gives that for every $t>0$,
        \[
        Z^N_t(f) \xrightarrow[N \to \infty]{L^2} 0.
        \]
        Therefore, by Proposition~\ref{Prop.Tightness.real-valued}, we only need to show that for all $\varepsilon>0$,  
			\begin{equation*}   
				\inf_{r>0}\limsup_{N\to\infty}\P\left[\omega\left(Z^N(f),r\right)\geq \varepsilon\right]=0.
			\end{equation*}
        By Dynkin's formula, the decomposition for $Z^N(f)$ is as follows:
        \begin{equation}\label{eq.Z.Dynkin.1}
            Z^N_t(f)=Z^N_0(f)+\int_0^t \L_N Z_s^N(f)\,\d s+M_t^N(f),
        \end{equation}
        where $M^N \equiv (M^N_t)_{t \geq 0}$ is a martingale term. Following the strategy in \cite{LyonsZheng1988}, we define a backward process for fixed $T>0$:
        \[
        \hat{\eta}^N_t := \eta^N_{T-t},\qquad t\in [0,T].
        \]
        Since the process is reversible, the generator for $\hat\eta_t^N$ is still $\L_N$. Then by Dynkin's formula, we have 
        \begin{equation}\label{eq.Z.Dynkin.2}
            {Z}^N_t(f)={Z}^N_T(f)+\int_0^{T-t} \L_N {Z}_{T-s}^N(f)\,\d s+\hat{M}_{T-t}^N(f),
    \end{equation}
    where $\hat{M}^N \equiv (\hat{M}^N_t)_{t \geq 0}$ is the martingale term for the backward process. Therefore, we have the following decomposition by adding \eqref{eq.Z.Dynkin.1} and \eqref{eq.Z.Dynkin.2}:
    \begin{equation*}
        2Z^N_t(f)=Z^N_0(f)+{Z}^N_T(f)+\int_0^T \L_N Z_s^N(f)\,\d s+M_t^N(f)-\hat{M}_{T-t}^N(f).
    \end{equation*}
    Then we have the following decomposition
    \begin{equation*}
        Z^N_t(f)=Z^N_0(f)+\frac{1}{2}{M}_t^N(f)+\frac{1}{2}(\hat{M}_T^N(f)-\hat{M}_{T-t}^N(f)).
    \end{equation*}
    Therefore, the modulus for $Z_t^N(f)$ reduces to the modulus for $M_t^N(f)$ and $\hat{M}_t^N(f)$. Here we have an observation that by Proposition~\ref{prop.tightness.quadratic.M}, \emph{the carr\'e du champ operator} for the corrected process $\mathcal Q_t^N(f)$ is uniformly bounded in $L^1$ and $L^2$ in $N\in \mathbb N$. Combining with the fact that \emph{the carr\'e du champ operator} for the original process $Q_t^N(f)$ is uniformly bounded in $L^1$ and $L^2$ in $N\in \mathbb N$, we have the conclusion that \emph{the carr\'e du champ operator} for the corrected process $Z_t^N(f)$ is also uniformly bounded in $L^1$ and $L^2$ in $N\in \mathbb N$. Then following the same procedure in Proposition~\ref{prop.tightness.quadratic.M}, the modulus for $M_t^N(f)$ and $\hat{M}_t^N(f)$ will decay to $0$ as the interval $r\downarrow 0$, which finishes the proof. 
	\end{proof}

	\subsection{Characterization of limit}\label{sec.A.limit}
	In this section, we give the characterization for the limit of the drift term. A key input here is the replacement argument.
	\begin{proposition}[Replacement]\label{prop.replacement}
		There exists $\gamma>0$ such that for every test function $f$ and every $t > 0$
		\begin{equation*}
			\left\Vert \mathcal{L}_{N}\mathcal{Q}_{t}^N(f) -Q_t^N\left(\Delta_\mathbf{D}^Nf\right)\right\Vert_{H^{-1}}\leq C_f N^{-\gamma}.
		\end{equation*}
	\end{proposition}
	\begin{proof}
		Since the norm is calculated under $\Pr$ which is stationary, we omit the index of time and write directly the term as $\mathcal{L}_{N}\mathcal{Q}^N(f), Q^N\left(\Delta_\mathbf{D}^Nf\right)$ in the following calculation. The proof can be divided into $3$ steps.

		\textbf{Step 0:} Decomposition.
		
		We test the term $\mathcal{L}_{N}\mathcal{Q}^N(f)-Q^N\left(\Delta_\mathbf{D}^Nf\right)$ with arbitrary function $G\left(\eta^N\right) \in H^1$:
		\begin{align*}
			&\quad \E_\rho\left[\left(\mathcal{L}_{N}\mathcal{Q}^N(f)-Q^N\left(\Delta_\mathbf{D}^Nf\right)\right)G\right]\\
			&=\E_\rho\left[N^2\sum_x\sum_{i=1}^dc_{x,x+e_i}\pi_{x,x+e_i}\mathcal Q^N(f)G\right]-\E_\rho\left[Q^N\left(\Delta_\mathbf{D}^Nf\right)G\right]\notag\\
			&=-\frac{1}{2}\E_\rho\left[N^2\sum_x\sum_{i=1}^dc_{x,x+e_i}\pi_{x,x+e_i}\mathcal Q^N(f)\pi_{x,x+e_i}G\right]-\E_\rho\left[Q^N\left(\Delta_\mathbf{D}^Nf\right)G\right]\notag\\
			&=-\frac{1}{2}\underbrace{\E_\rho\left[N^2\sum_x\sum_{i=1}^d\left(c_{x,x+e_i}\pi_{x,x+e_i}\mathcal Q^N(f)-\mathcal J_1^N(f,x,i)\right)\pi_{x,x+e_i}G\right]}_{\mathbf{I}}\notag\\
			&\quad\underbrace{-\frac{1}{2}\E_\rho\left[N^2\sum_x\sum_{i=1}^d\mathcal J_1^N(f,x,i)\pi_{x,x+e_i}G\right]-\E_\rho\left[Q^N\left(\Delta_\mathbf{D}^Nf\right)G\right]}_{\mathbf{II}}\notag,
		\end{align*}
		where 
		\begin{equation*}
			\mathcal J_1^N(f,x,i)=2N^{-1-d}\sum_{j=1}^d\sum_{\substack{y:y\neq x,x+e_i}}\mathbf D_{ij}(\rho)\left(\bar\eta^N(x)-\bar\eta^N(x+e_i)\right)\bar\eta^N(y)\nabla_{1,j}^Nf\left(\frac{x}{N},\frac{y}{N}\right).
		\end{equation*}
		From the first line to the second line, we use 
        \[
        \pi_b\pi_b= -2\pi_b, \qquad b\in \left(\Td_N\right)^*,
        \]
        and that $c_{x,y}$ is independent of $\eta^N(x)$ and $\eta^N(y)$. Term $\mathbf{I}$ is the core of the proof, connecting the drift term of the speed-change exclusion process to that of a constant-speed exclusion process. Term $\mathbf{II}$ corresponds to the error arising from the discrete approximation.

		\textbf{Step 1}: Estimate of $\mathbf{I}$.
		
		To estimate the term $\mathbf I$, we need to remove more diagonal terms to create more independence. Therefore, we introduce two terms $\mathcal  I_3^N(f,x,i)$ and $\mathcal J_2^N(f,x,i)$ to take the place of $\pi_{x,x+e_i}\mathcal Q^N(f)$ and $\mathcal J_1^N(f,x,i)$:
		\begin{equation*}
			\mathcal  I_3^N(f,x,i):=2N^{-1-d}\sum_{\substack{y:y\notin \Lambda_{2L}^{z(x)}}}\bar\eta^N(y)
			\mathbf{v}_{x,i}\cdot{\nabla_{1}^{N}f\left(\frac{z(x)}{N},\frac{y}{N}\right)},
		\end{equation*}
		\begin{equation*}
			\mathcal J_2^N(f,x,i):=2N^{-1-d}\sum_{j=1}^d\sum_{\substack{y:y\notin \Lambda_{2L}^{z(x)}}}\mathbf D_{ij}(\rho)\left(\bar\eta^N(x)-\bar\eta^N(x+e_i)\right)\bar\eta^N(y){\nabla_{1,j}^{N}f\left(\frac{z(x)}{N},\frac{y}{N}\right)},
		\end{equation*}
        where $\mathbf{v}_{x,i}$ is defined in \eqref{eq.v}. Here we split term $\mathbf I$ into three parts $\mathbf {I.1}$, $\mathbf {I.2}$, and $\mathbf {I.3}$, which will be estimated separately.
		\begin{align*}
			\mathbf I&=\E_\rho\left[N^2\sum_x\sum_{i=1}^d\left(c_{x,x+e_i}\pi_{x,x+e_i}\mathcal Q^N(f)-\mathcal J_1^N(f,x,i)\right)\pi_{x,x+e_i}G\right]=\mathbf{I.1}+\mathbf{I.2}+\mathbf{I.3},
		\end{align*}
		where the 3 terms respectively equal
		\begin{align*}
			\mathbf{I.1} &=\E_\rho\left[N^2\sum_x\sum_{i=1}^d\left(c_{x,x+e_i}\mathcal I_3^N(f,x,i)-\mathcal J_2^N(f,x,i)\right)\pi_{x,x+e_i}G\right],\\
			\mathbf{I.2} &=\E_\rho\left[N^2\sum_x\sum_{i=1}^dc_{x,x+e_i}\left(\pi_{x,x+e_i}\mathcal Q^N(f)-\mathcal I_3^N(f,x,i)\right)\pi_{x,x+e_i}G\right],\\
			\mathbf{I.3} &=\E_\rho\left[N^2\sum_x\sum_{i=1}^d\left(\mathcal J_2^N(f,x,i)-\mathcal J _1^N(f,x,i)\right)\pi_{x,x+e_i}G\right].
		\end{align*}
	
		\textbf{Step 1.1}: Estimate of $\mathbf{I.1}$.
		
		The estimate of the term $\mathbf{I.1}$ relies on the spatial cancellation estimate in (\ref{lem.FluxReplacement}) of Lemma~\ref{lem.Homogenization}. We calculate
		\begin{align*}
			\mathbf{I.1}
			&=\E_\rho\left[N^2\sum_x\sum_{i=1}^d\left(c_{x,x+e_i}\mathcal I_3^N(f,x,i)-\mathcal J_2^N(f,x,i)\right)\pi_{x,x+e_i}G\right]\\
			&=\E_\rho\left[2N^{1-d}\sum_{z\in\mathcal Z_L}\sum_{x\in \Lambda
				_L^z}\sum_{i=1}^d\pi_{x,x+e_i}G\sum_{\substack{y:y\notin \Lambda_{2L}^{z}}}\bar\eta^N(y)\right.\\
			&\left.\qquad\qquad{\nabla_{1}^{N}f\left(\frac{z}{N},\frac{y}{N}\right)}\cdot\left(c_{x,x+e_i}
			\mathbf{v}_{x,i}-\sum_{j=1}^d\mathbf D_{ij}(\rho)\left(\bar\eta^N(x)-\bar\eta^N(x+e_i)\right)e_j\right)\right]\\
			&=\E_\rho\left[2N^{1-d}\sum_{z\in\mathcal Z_L}\sum_{b\in \bar{\left(\Lambda_L^z\right)^*}}\pi_{b}G\sum_{\substack{y:y\notin \Lambda_{2L}^{z}}}\bar\eta^N(y)\sum_{i=1}^d{\nabla_{1,i}^{N}f\left(\frac{z}{N},\frac{y}{N}\right)}\mathbf g_{L,e_i,b}^z\right].
		\end{align*}
        Here we rewrite the bracketed term as \(\mathbf g^z_{L,e_i,b}\), using the definition of the centered flux in \eqref{eq.flux} 
    	\begin{equation*}
    		\mathbf{g}_{L,e_i,b}^z:= c_b\pi_b\left(\ell_{e_i}+\phi_{L,e_i}^z\right)-\pi_b\ell_{\mathbf D(\rho)e_i}.
    	\end{equation*}
        Since 
        \begin{equation*}
			\mathbf g_{L,e_i,b}^z\in \F_{\Lambda_{L+2\mathbf{r}+2}^z}\subset\F_{\Lambda_{2L}^z},
		\end{equation*}
        every variable \(\bar\eta(y)\) with \(y\notin\Lambda^{2z}_L\) is independent of \(\mathbf g^z_{L,e_i,b}\) under the product measure \(\P_\rho\). To make use of spatial cancellation in \cite[Proposition~4.3]{gu2025relaxation}, we use conditional expectation to split the terms whose supports are inside $\Lambda_{2L}^z$: 
		\begin{equation}\label{eq.spatial cancellation.1}
			\E_\rho\left[\sum_{b\in \bar{\left(\Lambda_L^z\right)^*}}\pi_{b}G\sum_{\substack{y:y\notin \Lambda_{2L}^{z}}}\bar\eta^N(y){\nabla_{1,i}^{N}f\left(\frac{z}{N},\frac{y}{N}\right)}\mathbf g_{L,e_i,b}^z\right]
			=\E_\rho\left[\sum_{b\in \bar{\left(\Lambda_L^z\right)^*}}\mathbf g_{L,e_i,b}^z\pi_bV_{e_i}^z\right],
		\end{equation}
		where 
		\begin{equation*}
			V_{e_i}^z= \E_\rho\left[G\sum_{\substack{y:y\notin \Lambda_{2L}^{z}}}\bar\eta^N(y){\nabla_{1,i}^{N}f\left(\frac{z}{N},\frac{y}{N}\right)}\Bigg|\F_{\Lambda_{2L}^z}\right].
		\end{equation*}
		Using spatial cancellation in (\ref{lem.FluxReplacement}) of Lemma~\ref{lem.Homogenization} with \(G=V^z_{e_i}\), we have
		\begin{equation}\label{eq.spatial cancellation.2}
			\E_\rho\left[\sum_{b\in \bar{\left(\Lambda_L^z\right)^*}}\mathbf g_{L,e_i,b}^z\pi_bV_{e_i}^z\right]\leq CL^{\frac{d}{2}-\alpha}\left(\sum_{b\in \bar{\left(\Lambda_L^z\right)^*}}\E_\rho\left[\left(\pi_bV_{e_i}^z\right)^2\right]\right)^\frac{1}{2}.
		\end{equation}
        Therefore, the gain \(L^{-\alpha}\) provided by the spatial cancellation estimate will compensate the growth of the number of bonds in subsequent summations. We use Cauchy--Schwarz inequality for conditional expectation to get the bound for the second moment of $\pi_b V^z_{e_i}$:
        \begin{align*}
			\E_\rho\left[\left(\pi_bV_{e_i}^z\right)^2\right]
			&=\E_\rho\left[\left(\E_\rho\left[\pi_bG\sum_{\substack{y:y\notin \Lambda_{2L}^{z}}}\bar\eta^N(y){\nabla_{1,i}^{N}f\left(\frac{z}{N},\frac{y}{N}\right)}\Bigg|\F_{\Lambda_{2L}^z}\right]\right)^2\right]\\
			&\leq \E_\rho\left[\E_\rho\left[\left(\pi_bG\right)^2\Big|\F_{\Lambda_{2L}^z}\right]\E_\rho\left[\left(\sum_{\substack{y:y\notin \Lambda_{2L}^{z}}}\bar\eta^N(y){\nabla_{1,i}^{N}f\left(\frac{z}{N},\frac{y}{N}\right)}\right)^2\Bigg|\F_{\Lambda_{2L}^z}\right]\right]\notag\\
            &=\E_\rho\left[\E_\rho\left[\left(\pi_bG\right)^2\Big|\F_{\Lambda_{2L}^z}\right]\right]\E_\rho\left[\left(\sum_{\substack{y:y\notin \Lambda_{2L}^{z}}}\bar\eta^N(y){\nabla_{1,i}^{N}f\left(\frac{z}{N},\frac{y}{N}\right)}\right)^2\right]\notag\\
		\end{align*}
        Here from the second line to the third line, we use that \(\{\eta(y):y\notin\Lambda^{2z}_L\}\) is independent of \(\mathcal F_{\Lambda^{2z}_L}\) under \(\P_\rho\). We thus obtain that
		\begin{equation}\label{eq.pi_bV}
            \begin{split}
            \E_\rho\left[\left(\pi_bV_{e_i}^z\right)^2\right]
			&\leq \E_\rho\left[\left(\pi_bG\right)^2\right]\E_\rho\left[\left(\sum_{\substack{y:y\notin \Lambda_{2L}^{z}}}\bar\eta^N(y){\nabla_{1,i}^{N}f\left(\frac{z}{N},\frac{y}{N}\right)}\right)^2\right]\\
			&\leq C_fN^d\E_\rho\left[\left(\pi_bG\right)^2\right].
            \end{split}
		\end{equation}
        Combining \eqref{eq.spatial cancellation.1}, \eqref{eq.spatial cancellation.2} and \eqref{eq.pi_bV}, we have
		\begin{align*}
			\mathbf{I.1}&=2N^{1-d}\sum_{z\in\mathcal Z_L}\E_\rho\left[\sum_{b\in \bar{\left(\Lambda_L^z\right)^*}}\pi_{b}G\sum_{\substack{y:y\notin \Lambda_{2L}^{z}}}\bar\eta^N(y)\sum_{i=1}^d{\nabla_{1,i}^{N}f\left(\frac{z}{N},\frac{y}{N}\right)}\mathbf g_{L,e_i,b}^z\right]\\
             &\leq CN^{1-d}L^{\frac{d}{2}-\alpha}\sum_{z\in\mathcal Z_L}\left(\sum_{b\in \bar{\left(\Lambda_L^z\right)^*}}\E_\rho\left[\left(\pi_bV_{e_i}^z\right)^2\right]\right)^\frac{1}{2}\\
            &\leq C_fN^{1-\frac{d}{2}}L^{\frac{d}{2}-\alpha}\sum_{z\in\mathcal Z_L}\left(\sum_{b\in \bar{\left(\Lambda_L^z\right)^*}}\E_\rho\left[\left(\pi_bG\right)^2\right]\right)^\frac{1}{2}\\
			&\leq C_fL^{-\alpha}\left(N^2\sum_{z\in\mathcal Z_L}\sum_{b\in \bar{\left(\Lambda_L^z\right)^*}}\E_\rho\left[\left(\pi_bG\right)^2\right]\right)^\frac{1}{2}\\
			&\leq C_f L^{-\alpha}\Vert G\Vert_{H^1},  
		\end{align*}
		where from the third line to the fourth line, we use Cauchy--Schwarz inequality. Therefore, we have for some $\gamma>0$,
		\begin{equation*}
			\mathbf{I.1}\leq C_fN^{-\gamma}\Vert G\Vert_{H^1}.
		\end{equation*}
		\textbf{Step 1.2}: Estimate of $\mathbf{I.2}$.
		
		This is a remainder term. By Cauchy--Schwarz inequality, Lemma~\ref{lemma.R_1.expectation} and  Lemma~\ref{lemma.minus.expectation}, there exists an exponent $\gamma > 0$ such that
		\begin{align*}
			\mathbf{I.2}
			&=\E_\rho\left[N^2\sum_x\sum_{i=1}^dc_{x,x+e_i}\left(\pi_{x,x+e_i}\mathcal Q^N(f)-\mathcal I_3^N(f,x,i)\right)\pi_{x,x+e_i}G\right]\\
			&=\E_\rho\left[N^2\sum_x\sum_{i=1}^dc_{x,x+e_i}\left(\mathcal R_1^N(f,x,i)+\mathcal I_1^N(f,x,i)-\mathcal I_3^N(f,x,i)\right)\pi_{x,x+e_i}G\right]\\
			&\leq \E_\rho\left[N^2\sum_x\sum_{i=1}^dc_{x,x+e_i}\mathcal R_1^N(f,x,i)^2\right]^{\frac{1}{2}}\Vert G\Vert_{H^1}\\
			&\quad+\E_\rho\left[N^2\sum_x\sum_{i=1}^dc_{x,x+e_i}\left(\mathcal I_1^N(f,x,i)-\mathcal I_3^N(f,x,i)\right)^2\right]^{\frac{1}{2}}\Vert G\Vert_{H^1}\\
			&\leq C_f N^{-\gamma}\Vert G\Vert_{H^1}.
		\end{align*}
		\textbf{Step 1.3}: Estimate of $\mathbf{I.3}$.
		
		This is a remainder term. By Cauchy--Schwarz inequality and Lemma~\ref{lemma.minus.expectation}, we have
		\begin{align*}
			\mathbf{I.3}&=\E_\rho\left[N^2\sum_x\sum_{i=1}^d\left(\mathcal J_2^N(f,x,i)-\mathcal J_1^N(f,x,i)\right)\pi_{x,x+e_i}G\right]\\
			&\leq \E_\rho\left[N^2\sum_x\sum_{i=1}^d\left(\mathcal J_2^N(f,x,i)-\mathcal J_1^N(f,x,i)\right)^2\right]^\frac{1}{2}\Vert G\Vert_{H^1}\\
			&\leq C_f N^{-\gamma}\Vert G\Vert_{H^1}.
		\end{align*}
		\textbf {Step 2}: Estimate of $\mathbf{II}$.
		
		We calculate
		\begin{align*}
			&\quad\E_\rho\left[N^2\sum_x\sum_{i=1}^d\mathcal J_1^N(f,x,i)\pi_{x,x+e_i}G\right]\\
			&=2N^{1-d}\sum_{i,j=1}^d\sum_{\substack{x,y\\y\neq x,x+e_i}}\mathbf D_{ij}(\rho)\E_\rho\left[\left(\bar\eta^N(x)-\bar\eta^N(x+e_i)\right)\bar\eta^N(y)\pi_{x,x+e_i}G\right]\nabla_{1,j}^Nf\left(\frac{x}{N},\frac{y}{N}\right)\\
			&=-4N^{1-d}\sum_{i,j=1}^d\sum_{\substack{x,y\\y\neq x,x+e_i}}\mathbf D_{ij}(\rho)\E_\rho\left[\left(\bar\eta^N(x)-\bar\eta^N(x+e_i)\right)\bar\eta^N(y)G\right]\nabla_{1,j}^Nf\left(\frac{x}{N},\frac{y}{N}\right)\\
			&=2\E_\rho\left[Q^N\left(\Delta_\mathbf{D}^Nf\right)G\right]\\
			&\qquad-4N^{-1-d}\sum_{i,j=1}^d\sum_{x}\mathbf D_{ij}(\rho)\E_\rho\left[\bar\eta^N(x)\bar\eta^N(x+e_i)G\right]\nabla_{2,i}^N\nabla_{1,i}^N\nabla_{1,j}^N f\left(\frac{x}{N},\frac{x}{N}\right).
		\end{align*}
		Therefore, we have
		\begin{align*}
			\mathbf{II}&=2N^{-1-d}\sum_{i,j=1}^d\sum_{x}\mathbf D_{ij}(\rho)\E_\rho\left[\bar\eta^N(x)\bar\eta^N(x+e_i)G\right]\nabla_{2,i}^N\nabla_{1,i}^N\nabla_{1,j}^N f\left(\frac{x}{N},\frac{x}{N}\right)\\
			&=2N^{-1-d}\sum_{i,j=1}^d\sum_{x}\mathbf D_{ij}(\rho)\E_\rho\left[\bar\eta^N(x)\bar\eta^N(x+e_i)\pi_{x,x+e_i}G\right]\nabla_{2,i}^N\nabla_{1,i}^N\nabla_{1,j}^N f\left(\frac{x}{N},\frac{x}{N}\right)\\
			&\leq 2N^{-2-d}\E_\rho\left[\sum_{i,j=1}^d\sum_x\mathbf D_{ij}(\rho)^2\bar\eta^N(x)^2\bar\eta^N(x+e_i)^2\nabla_{2,i}^N\nabla_{1,i}^N\nabla_{1,j}^N f\left(\frac{x}{N},\frac{x}{N}\right)^2\right]^\frac{1}{2}\Vert G\Vert_{H^1}\\
			&\leq C_fN^{-2-\frac{d}{2}}\Vert G\Vert_{H^1}.
		\end{align*}
		Combining all steps, we get for every $H^1$ function $G$,
		\begin{equation*}
			\E_\rho\left[\left(\mathcal{L}_{N}\mathcal{Q}^N(f)-Q^N\left(\Delta_\mathbf{D}^Nf\right)\right)G\right]\leq C_f N^{-\gamma}\Vert G\Vert_{H^1},
		\end{equation*}
		which concludes the proof.
	\end{proof}
    After this core preparation, we are now ready to give the characterization of the limit for the drift term $\mathcal A_t^N$, denoted by $\mathcal A_t$. 
    \begin{proposition}[Characterization for the limit of the drift term]\label{prop.limit.drift}
		For every function $f\in C^\infty\left(\T^{2d}\right)$, each limit point of $\left\{\mathcal A_t^N(f), t\in [0,T]\right\}_{N \in \N}$ satisfies
		\begin{equation}\label{eq.characterization of the drift term}
			\mathcal A_t(f)=\int_0^t \mathrm{Tr}\left(\mathbf D(\rho)Q_s(\partial^2 _{1,2}f)\right)\d s.
		\end{equation}
	\end{proposition} 
    \begin{proof}
        By Kipnis--Varadhan Lemma~\ref{Kipnis--Varadhan}, Proposition~\ref{prop.replacement} implies for each test function $f$,
	    \[
	    \E_\rho\left[\sup_{t\in[0,T]}\left\vert\mathcal A_t^N(f)-\int_0^t Q_s^N\left(\Delta_\mathbf{D}^Nf\right)\,\d s\right\vert^2\right]\leq C_fTN^{-\gamma}.
	    \]
        Therefore, $\Ll(\mathcal A_{t}(f)\Rr)_{t \in [0,T]}$ is also the limit of $\Ll(\int_0^{t} Q_s^N\left(\Delta_\mathbf{D}^Nf\right)\,\d s\Rr)_{t \in [0,T]}$.

        Recall the correction \eqref{eq.correction} and decomposition \eqref{eq.Dynkin}. As the tightness of $\mathcal M^N, \mathcal A^N, Z^N$ is restively justified in Proposition~\ref{prop.tightness.quadratic.M}, Lemma~\ref{lemma.tightness.A}, Lemma~\ref{lemma.Z.convergence}, the process $\{Q^N\}_{N \in \N}$ is also tight in $D\left([0,T], \S' \left(\Td\right)\right)$, i.e.
        \begin{align*}
            Q^{N_k}\xRightarrow{N_k\to\infty}Q .
        \end{align*}
        We also have $\Delta_\mathbf{D}^{N_k} f\xrightarrow{N_k} \mathrm{Tr}(\mathbf D(\rho)\partial^2_{1,2}f)$ uniformly. Taking the integral with respect to time and space, for every function $f\in C^\infty(\T^{2d})$, the following convergence holds in $D([0, T], \mathbb{R})$:
        \[
        \int_0^t Q_s^{N_k}\left(\Delta_\mathbf{D}^{N_k}f\right)\,\d s\xRightarrow{N_k\to\infty} \int_0^t \mathrm{Tr}\left(\mathbf D(\rho)Q_s(\partial _{1,2}f)\right)\d s.
        \]
        Therefore, along the same subsequence $\mathcal A^{N_k}(f)$, we have the following characterization:
        \begin{equation*}
			\mathcal A_t(f)=\int_0^t \mathrm{Tr}\left(\mathbf D(\rho)Q_s(\partial^2 _{1,2}f)\right)\d s.
		\end{equation*}
        This completes the proof.
    \end{proof}

	\section{Characterization of quadratic field}\label{section.characterization}
    In this section, we give the characterization for the quadratic field. In order to accomplish this, we first summarize the discussion for weak convergence in the previous sections into the following lemma:
    \begin{lemma}[Weak joint convergence]\label{lemma.weak convergence}
        The following joint convergence holds in the weak sense along subsequence
    \[
    (Q_t^N, \mathcal Q_t^N, \Y_t^N, \M_t^N, \mathcal A_t^N)_{t\in [0,T]} \xRightarrow{N \to \infty} (Q_t, Q_t,\Y_t,\M_t, \mathcal A_t)_{t \in [0,T]},
    \] 
    and $\M,\mathcal A$ are characterized as Propositions~\ref{prop.Characterization for the convergence of the martingale term} and ~\ref{prop.Characterization for the convergence of the drift term}.
    \end{lemma}
    \begin{proof}
        As mentioned, the tightness proved in Proposition~\ref{prop.tightness.quadratic.M}, Lemma~\ref{lemma.tightness.A}, Lemma~\ref{lemma.Z.convergence} entails that of $(Q_t^N, \mathcal Q_t^N, \Y_t^N, \M_t^N, \mathcal A_t^N, Z_t^N)_{t\in [0,T]}$. Lemma~\ref{lemma.Z.convergence} states that the limit of $(Z_t^N)_{t\in [0,T]}$ is zero process. Therefore, we have the target joint convergence, which finishes the proof.  
    \end{proof}

    Based on the weak convergence result in Lemma~\ref{lemma.weak convergence}, for every function ${f\in C^\infty\left(\T^{2d}\right)}$, we have 
    \begin{equation}\label{eq.mart problem}
        Q_t(f) = Q_0(f) + \int_0^t \mathrm{Tr}\left(\mathbf D(\rho)Q_s(\partial^2 _{1,2}f)\right)\d s + \mathcal{M}_t(f),
    \end{equation}
    where $\mathcal{M}(f)$ is a continuous martingale of quadratic variation
	\begin{multline*}
		\left\la \mathcal M (f)\right\ra_t
		= \int_0^t\int_{\T^d}\left\{\Y_s\left(\nabla_{1}f(x,\cdot)\right)+\Y_s\left(\nabla_{2}f(\cdot,x)\right)\right\}\\\cdot \cc(\rho)\,\left\{\Y_s\left(\nabla_{1}f(x,\cdot)\right)+\Y_s\left(\nabla_{2}f(\cdot,x)\right)\right\}\,\d x\,\d s.
	\end{multline*}
    It is nearly as Theorem~\ref{thm:quadratic fluctuation}, except \eqref{eq.W}.
	The quadratic variation and Levy's characterization of Brownian motion suggests that 
	\begin{equation*}
		\mathcal{M}_t(f)=\int_0^t\int_{\T^d}\sqrt{\cc(\rho)}\left\{\Y_s\left(\nabla_{1}f(x,\cdot)\right)+\Y_s\left(\nabla_{2}f(\cdot,x)\right)\right\}\cdot\d{\tilde \omega}(s,x),
	\end{equation*}
	where $\tilde w=\left\{\tilde w_i(t)\right\}_{1\leq i\leq d}$ is a $d$-dimensional space-time white noise. The final step is to verify that the limiting noise term $\tilde w$ coincides with the white noise $w$ defined in \eqref{first-order fluctuation}. The proof relies on \cite[Theorem~3.9]{GJ19}. The only difference is that the diffusion operator is not an isotropic one. We define the process 
    $\{\mathcal N_t;t\in [0,T]\}$ in the following way: for every function $f\in C^{\infty}\left(\Td\right)$,
    \begin{equation}\label{def.N}
        \mathcal N_t(f) :=\mathcal Y_t(f)-\mathcal Y_0(f)-\int_0^t \mathcal Y_s\left(\Delta_{\mathbf D}f\right),
    \end{equation}
    and summarize the above discussion in the following proposition.
    \begin{proposition}\label{prop.mart.unique}
        The distribution of $(Q_t, \mathcal Y_t, \mathcal N_t)_{t\in [0,T]}$ is uniquely determined.  
    \end{proposition}
	\begin{proof}
	    Let us verify the conditions in \cite[Theorem~3.9]{GJ19}, which are listed below:
        \begin{enumerate}
            \item [(i)] For every \( f \in C^\infty(\mathbb{T}^{2d}) \), the process \( \{\mathcal{N}_t(f); t \in [0, T]\} \) is a continuous martingale of quadratic variation 
            \[
            t\int_{\Td}\nabla f(u)\cdot\cc(\rho)\nabla f(u)\, \d u.
            \]
            \item [(ii)] For every \( f \in C^\infty(\mathbb{T}^d) \), the process \( \{\mathcal{Y}_t(f); t \in [0, T]\} \) satisfies the relation
            \[
            \mathcal Y_t(f)=\mathcal Y_0(f)+\int_0^t \mathcal Y_s\left(\Delta_{\mathbf D}f\right)\,\d s+\mathcal N_t(f).
            \]
            \item [(iii)] There exists a \( \mathcal{S}'(\mathbb{T}^{2d}) \)-valued process \( \{\mathcal{M}_t; t \in [0, T]\} \) such that for any \( f_1, f_2 \in C^\infty(\mathbb{T}^d) \),
            \[
            \mathcal{M}_t(f_1(x)f_2(y)) = \int_0^t \sqrt{\cc(\rho)}\left\{ \mathcal{Y}_s(f_1) \, \d\mathcal{N}_s(f_2) + \mathcal{Y}_s(f_2) \, \d\mathcal{N}_s(f_1) \right\}.
            \]
            \item [(iv)] For every \( f \in C^\infty(\mathbb{T}^{2d}) \), we have
            \[
            Q_t(f) = Q_0(f) + \int_0^t Q_s(\Delta
            _{\D}f) \, \d s + \mathcal{M}_t(f).
            \]
            \item [(v)] For every \( f \in C^\infty(\mathbb{T}^{2d}) \) and any \( t \in [0, T] \), the real-valued random variable \( Q_t(f) \) has a Gaussian distribution of mean zero and variance 
            \[
            \chi(\rho)^2\int_{\mathbb{T}^{2d}} f(u,v)^2 \, \d u \, \d v.
            \]
        \end{enumerate}
        Now we verify them one by one. 
        
        By Proposition~\ref{prop.linear fluctuation}, the quadratic variation of $\mathcal N_t(f)$ is given by 
	    \[
	       \la \mathcal N(f)\ra_t=t\int_{\Td}\nabla f(u)\cdot\cc(\rho)\nabla f(u)\, \d u.
	    \]
    This is exactly the condition (i). 
    
    Condition (ii) is automatically satisfied by definition of $\mathcal N_t$ in \eqref{def.N}. 

    As to condition (iii), we analyze with functions of the form $f(x,y)=f_1(x)f_2(y)$. By the definition of $Q_t^N$ in \eqref{eq.def.Q}, we see that
	\[
	Q_t^N(f)=
	\Y_t^N(f_1)\Y_t^N(f_2)-N^{-d}\sum_{x}\bar\eta_t^N(x)^2f_1\left(\frac{x}{N}\right)f_2\left(\frac{x}{N}\right),
	\]
	and taking $N\to\infty$, we conclude that
	\[
	Q_t(f)=
	\Y_t(f_1)\Y_t(f_2)-\chi(\rho)\int_{\Td}f_1\left(u\right)f_2\left(u\right)\,\d u.
	\]
	By \Ito's formula, we obtain the martingale decomposition of $\left\{\Y_t(f_1)\Y_t(f_2)\right\}_{t\in [0,T]}$:
	\begin{multline*}
		\Y_t(f_1)\Y_t(f_2)
		=\Y_0(f_1)\Y_0(f_2)+t\int_{\Td}\nabla f_1(u)\cdot \cc(\rho)\nabla f_2(u)\,\d u\\+\int_0^t\Y_s\left(\Delta_{\mathbf D}f_1\right)\Y_s(f_2)+\Y_s(f_1)\Y_s\left(\Delta_{\mathbf D}
		f_2\right)\,\d s\\
		+\int_0^t\int_{\Td}\Y_s(f_2)\sqrt{\cc(\rho)}\nabla f_1(u)\cdot\,\d\omega(s,u)+\Y_s(f_1)\sqrt{\cc(\rho)}\nabla f_2(u)\cdot\,\d\omega(s,u).
	\end{multline*}
	Note that with integration by parts and the relation between $\cc$ and $\mathbf D$ in \eqref{eq.def.D}, we have
	\begin{multline*}
		\int_0^t \mathrm{Tr}(\D(\rho)Q_s(\partial^2_{1,2}f))\,\d s
		=t\int_{\Td}\nabla f_1(u)\cdot \cc(\rho)\nabla f_2(u)\,\d u\\+\int_0^t\Y_s\left(\Delta_{\D}f_1\right)\Y_s(f_2)+\Y_s(f_1)\Y_s\left(\Delta_{\mathbf D}
		f_2\right)\,\d s.
	\end{multline*}
	From this we conclude that for function $f$ of the form $f(x,y)=f_1(x)f_2(y)$, we have 
	\begin{multline*}
		Q_t(f)=Q_0(f)+\int_0^t \mathrm {Tr}\left(\D(\rho)Q_s(\mathbf \partial_{1,2}^2f)\right)\,\d s
		\\+\int_0^t\sqrt{\cc(\rho)}\left\{\Y_s(f_2)\,\d \mathcal N_s(f_1)+\Y_s(f_1)\,\d \mathcal N_s(f_2)\right\}.
	\end{multline*}
    This finishes the verification of condition (iii). 
    
    The martingale problem in \eqref{eq.mart problem} gives condition (iv). 
    
    Finally, we show that the limit field is Gaussian for all time $t$ in condition (v). Here we use characteristic functions. For $f\in C^{\infty}\left(\T^{2d}\right)$, we have the characteristic function for the limit:
	\begin{equation*}
		\Phi^N(s)=\E_{\rho}\left[\exp\left(\mathit{i}s Q_t^N(f)\right)\right]=\prod_{x,y:x\neq y}\E_{\rho}\left[\exp\left(\mathit{i}s N^{-d}f\left(\frac{x}{N},\frac{y}{N}\right)\bar\eta_t^N(x)\bar\eta_t^N(y)\right)\right].
	\end{equation*}
	We use Taylor's expansion to get
	\begin{equation*}
		\E_{\rho}\left[\exp\left(\mathit{i}s N^{-d}f\left(\frac{x}{N},\frac{y}{N}\right)\bar\eta_t^N(x)\bar\eta_t^N(y)\right)\right]=1-\frac{1}{2}\chi(\rho)^2s^2N^{-2d}f\left(\frac{x}{N},\frac{y}{N}\right)^2+o\left(N^{-2d}\right).
	\end{equation*}
	Therefore, we have
	\begin{align*}
		\log \Phi^N(s)&=\sum_{x,y:x\neq y}\log\E_{\rho}\left[\exp\left(\mathit{i}s N^{-d}f\left(\frac{x}{N},\frac{y}{N}\right)\bar\eta_t^N(x)\bar\eta_t^N(y)\right)\right]\\
		&=-\sum_{x,y:x\neq y} \frac{1}{2}\chi(\rho)^2s^2N^{-2d}f\left(\frac{x}{N},\frac{y}{N}\right)^2+o(1)\\
		&\xrightarrow{N\to\infty}-\frac{1}{2}\chi(\rho)^2s^2\Vert f\Vert_{L^2}^2,
	\end{align*}
	which shows that for all $f\in C^{\infty}\left(\T^{2d}\right)$, the fluctuation $Q_t^N(f)$ tends to a Gaussian distribution as $N\to\infty$ at every time $t$, which concludes the proof.
	\end{proof}
	Now we are ready to summarize the proof for our main theorem.
    \begin{proof}[Proof of Theorem~\ref{thm:quadratic fluctuation}]
    
    Recall that by Dynkin's formula, for $f\in C^\infty(\T^{2d})$, we have the decomposition
	\begin{equation*}
		\mathcal{Q}_t^N(f)=\mathcal{Q}_0^N(f)+\mathcal{A}_t^N(f)+\M_t^N(f).
	\end{equation*}
    In Section~\ref{sec.M.tight}, we prove the tightness of $\M_t^N(f)$ and the limit of $\la\M^N(f)\ra_t$ is characterized in Section~\ref{sec.M.limit}. For the drift term $\mathcal A_t^N(f)$, the tightness is shown in Section~\ref{sec.A.tight} and the limit is verified in Section~\ref{sec.A.limit}. Therefore, we get the joint convergence as stated in Lemma~\ref{lemma.weak convergence}, and every limit point satisfies the martingale problem in \eqref{eq.mart problem}. With the help of Proposition~\ref{prop.mart.unique}, we conclude that the martingale problem has a unique solution in distribution. Above all, the limit point of $(Q_t^N)_{t\in [0,T]}$ is unique, and we have the conclusion that the whole sequence converges to that limit point, which is characterized by the Ornstein--Uhlenbeck process in Theorem~\ref{thm:quadratic fluctuation}. This finishes the proof.

    \end{proof}

	\appendix
		\section{Moment estimates of remainder terms} 
		In the appendix, we give the detailed calculation for the moments of the remainder terms.

        For the proof of Lemma~\ref{lemma.R_1.expectation}, which is used in the proof of Proposition~\ref{prop.B.bound.L1.L2} in Section~\ref{sec.M.tight}, we need the following three auxiliary lemmas: 
		\begin{lemma}\label{lemma.R_{1,1}.expectation}
			For every test function $f$, we have
			\begin{align*}
				&N^2\sum_{i=1}^d\sum_{x}\E_{\rho}\left[\mathcal R_{1,1}^N(f,x,i)^2\right]
				\leq C_{f}\left(N^{-2}L^2+N^{-d}L^{d+2}\right),
			\end{align*}
			and
			\begin{multline*}
				N^4\sum_{i=1}^d\sum_{x}\E_{\rho}\left[\mathcal R_{1,1}^N(f,x,i)^4\right]\\
				\leq C_{f}N^{-d}\left(N^{-4}L^4+N^{-2-d}L^{d+3}+N^{-2d}L^{2d+3}+\left(N^{-4-d}L^{-2d}+N^{-2d}L^{-1}\right)\Vert \phi_L\Vert_\infty^4\right)\notag.
			\end{multline*}
		\end{lemma}
		\begin{proof}
			We make a direct calculation for the second moment:
			\begin{align*}
				&\quad N^2\sum_{i=1}^d\sum_{x}\E_{\rho}\left[\mathcal R_{1,1}^N(f,x,i)^2\right]\\
				&=4N^{-2d}\sum_{i=1}^d\sum_{x}\E_\rho\left[\left\{\sum_{\substack{ z\in \mathcal Z_L }}
				\phi_L^{z}\cdot\pi_{x,x+e_i}\left(\sum_{y:y\notin \Lambda_L^z}\bar\eta^N(y){\nabla_{1}^{N}f\left(\frac{z}{N},\frac{y}{N}\right)}\right)\right\}^2\right]\notag\\
				&\leq  4 d  N^{-2d}\sum_{i,j=1}^d\sum_{x}\sum_{\substack{ z\in \mathcal Z_L}}\E_\rho\left[\left|
				\phi_{L,j}^ z\right|^2\right]\E_\rho\left[\left|\pi_{x,x+e_i}\left(\sum_{y:y\notin \Lambda_L^z}\bar\eta^N(y){\nabla_{1,j}^{N}f\left(\frac{z}{N},\frac{y}{N}\right)}\right)\right|^2\right]\notag,
			\end{align*}
			where the last step follows from independence and Cauchy--Schwarz inequality. For $1\leq j\leq d$ and $k\in\mathbb N_+$, we have
			\begin{equation*}
				\E_\rho\left[\left|\pi_{x,x+e_i}\left(\sum_{y:y\notin \Lambda_L^z}\bar\eta^N(y){\nabla_{1,j}^{N}f\left(\frac{z}{N},\frac{y}{N}\right)}\right)\right|^{2k}\right]\leq \begin{cases}
					0    & x,x+e_i\in \Lambda_L^z,\\
					N^{-2k}C_{f} &x,x+e_i\notin \Lambda_L^z,\\
					C_{f} &\textit{otherwise}.
				\end{cases}
			\end{equation*}
			Therefore, we have
			\begin{multline*}
				N^2\sum_{i=1}^d\sum_{x}\E_{\rho}\left[\mathcal R_{1,1}^N(f,x,i)^2\right]\\\leq C_{f}N^{-2d}L^{d+2}\left(\frac{N}{L}\right)^d\left(N^dN^{-2}+L^{d-1}\right)\leq C_{f}\left(N^{-2}L^2+N^{-d}L^{d+1}\right).
			\end{multline*}
			For the fourth moment of $\mathcal R_{1,1}^N$, we calculate
			\begin{align}\label{eq.expectation.R_{1,1}.fourth}
				&\quad\E_{\rho}\left[\mathcal R_{1,1}^N(f,x,i)^4\right]\\
				&=16N^{-4-4d}\E_\rho\left[\left\{\sum_{\substack{ z\in \mathcal Z_L }}
				\phi_L^{z}\cdot\pi_{x,x+e_i}\left(\sum_{y:y\notin \Lambda_L^z}\bar\eta^N(y){\nabla_{1}^{N}f\left(\frac{z}{N},\frac{y}{N}\right)}\right)\right\}^4\right]\notag\\
				&\leq 16d^{3}N^{-4-4d}\sum_{j=1}^d\E_\rho\left[\left\{\sum_{\substack{ z\in \mathcal Z_L }}
				\phi_{L,j}^{z}\pi_{x,x+e_i}\left(\sum_{y:y\notin \Lambda_L^z}\bar\eta^N(y){\nabla_{1,j}^{N}f\left(\frac{z}{N},\frac{y}{N}\right)}\right)\right\}^4\right]\notag,
			\end{align}
			where the last step follows from Cauchy--Schwarz inequality. We have a quick observation that the expectation can have a non-zero value if the four squares do pair each other. Therefore, we have
			\begin{align}\label{eq.expectation.R_{1,1}.fourth.2}
				&\quad\sum_x\E_\rho\left[\left\{\sum_{\substack{ z\in \mathcal Z_L }}
				\phi_{L,j}^{z}\pi_{x,x+e_i}\left(\sum_{y:y\notin \Lambda_L^z}\bar\eta^N(y){\nabla_{1,j}^{N}f\left(\frac{z}{N},\frac{y}{N}\right)}\right)\right\}^4\right]\\
				&=3\sum_x\sum_{\substack{ z_1,z_2\in \mathcal Z_L\\
						z_1\neq z_2}}\E_\rho\left[\left(\phi_{L,j}^ {z_1}\right)^2\right]\E_\rho\left[\left|\pi_{x,x+e_i}\left(\sum_{y_1:y_1\notin \Lambda_L^{z_1}}\bar\eta^N(y_1){\nabla_{1,j}^{N}f\left(\frac{z_1}{N},\frac{y_1}{N}\right)}\right)\right|^2\right]\notag\\
				&\hspace{2.4cm}\E_\rho\left[\left(\phi_{L,j}^ {z_2}\right)^2\right]\E_\rho\left[\left|\pi_{x,x+e_i}\left(\sum_{y_2:y_2\notin \Lambda_L^{z_2}}\bar\eta^N(y_2){\nabla_{1,j}^{N}f\left(\frac{z_2}{N},\frac{y_2}{N}\right)}\right)\right|^2\right]\notag\\
				&\quad+\sum_{x}\sum_{\substack{ z\in \mathcal Z_L}}\E_\rho\left[\left(
				\phi_{L,j}^ z\right)^4\right]\E_\rho\left[\left|\pi_{x,x+e_i}\left(\sum_{y:y\notin \Lambda_L^z}\bar\eta^N(y){\nabla_{1,j}^{N}f\left(\frac{z}{N},\frac{y}{N}\right)}\right)\right|^4\right]\notag\\
				&\leq C_{f}\left(L^{d+2}\right)^2\left(\left(\frac{N}{L}\right)^{2d}N^d\left(N^{-2}\right)^2+\left(\frac{N}{L}\right)^{2d}L^{d-1}N^{-2}+\left(\frac{N}{L}\right)^{d}L^{d-1}\right)\notag\\
				&\quad+C_{f}\left(\frac{N}{L}\right)^{d}\Vert \phi_L\Vert_\infty^4\left(N^{d}N^{-4}+L^{d-1}\right)\notag\\
				&\leq C_{f}\left(N^{-4+3d}L^4+N^{-2+2d}L^{d+3}+N^{d}L^{2d+3}+\left(N^{-4+2d}L^{-2d}+N^{d}L^{-1}\right)\Vert \phi_L\Vert_\infty^4\right)\notag.
			\end{align}
			Combining \eqref{eq.expectation.R_{1,1}.fourth} and \eqref{eq.expectation.R_{1,1}.fourth.2}, we have
			\begin{multline*}
				N^4\sum_{i=1}^d\sum_x\E_{\rho}\left[\mathcal R_{1,1}^N(f,x,i)^4\right]\\
				\leq C_{f}\left(N^{-4-d}L^4+N^{-2-2d}L^{d+3}+N^{-3d}L^{2d+3}+\left(N^{-4-2d}L^{-2d}+N^{-3d}L^{-1}\right)\Vert \phi_L^z\Vert_\infty^4\right),
			\end{multline*}
			which concludes the proof.
		\end{proof}
		\begin{lemma}\label{lemma.R_{1,2}.expectation}
			For test function $f$, we have
			\begin{equation*}
				N^2\sum_{i=1}^d\sum_x\E_\rho\left[\mathcal R_{1,2}^N(f,x,i)^2 \right]\leq C_{f}N^{-d} L^{d},
			\end{equation*}
			and
			\begin{equation*}
				N^4\sum_{i=1}^d\sum_x\E_\rho\left[\mathcal R_{1,2}^N(f,x,i)^4\right]\leq C_{f}N^{-3d}L^{2d}.
			\end{equation*}
		\end{lemma}
		\begin{proof}
			We make a direct calculation for the second moment:
			\begin{align*}
				&\quad N^2\sum_{i=1}^d\sum_x\E_\rho\left[\mathcal R_{1,2}^N(f,x,i)^2 \right]\\
				&=4N^{-2d}\sum_{i=1}^d\sum_{\substack{ z\in \mathcal Z_L }}\sum_{\substack{x,y\in \Lambda_L^z\\y\neq x,x+e_i}}\E_{\rho}\left[\bar\eta^N(y)^2\right]\E_{\rho}\left[
				\left(\bar\eta^N(x)-\bar\eta^N(x+e_i)\right)^2\right]\nabla_{1,i}^{N}f\left(\frac{x}{N},\frac{y}{N}\right)^2\notag\\
				&\leq C_{f}N^{-2d} \left(\frac{N}{L}\right)^d L^d L^d \notag\\
				&\leq C_{f}N^{-d} L^{d}\notag.
			\end{align*}
			For the fourth moment, we have
			\begin{align*}
				&\quad N^4\sum_{i=1}^d\sum_x\E_\rho\left[\mathcal R_{1,2}^N(f,x,i)^4\right]\\
				&=4\sum_{\substack{ z\in \mathcal Z_L }}\sum_{i=1}^d\sum_{x\in \Lambda_L^z}\E_\rho\left[\left(\bar\eta^N(x)-\bar\eta^N(x+e_i)\right)^4\right]\E_\rho\left[\left\{\sum_{\substack{y\in \Lambda_L^z\\y\neq x,x+e_i}}\bar\eta^N(y)\nabla_{1,i}^{N}f\left(\frac{x}{N},\frac{y}{N}\right)\right\}^4\right]\\
				&\leq C_{f}N^{-4d} \left(\frac{N}{L}\right)^d L^d \left(L^{2d}+L^d\right)\\
				&\leq C_{f}N^{-3d}L^{2d},
			\end{align*}
			where from the second line to the third line, the expectation can be non-zero as long as four $y$'s pair each other.
		\end{proof}
		\begin{lemma}\label{lemma.R_{1,3}.expectation}
			For test function $f$, we have
			\begin{equation*}
				N^2\sum_{i=1}^d\sum_x\E_\rho\left[\mathcal R_{1,3}^N(f,x,i)^2 \right]\leq C_{f}N^{-2} L^2,
			\end{equation*}
			and
			\begin{equation*}
				N^4\sum_{i=1}^d\sum_x\E_\rho\left[\mathcal R_{1,3}^N(f,x,i)^4\right]\leq C_{f}N^{-4-d} L^4.
			\end{equation*}
		\end{lemma}
		\begin{proof}
			We make a direct calculation for the second moment:
			\begin{align*}
				&\quad N^2\sum_{i=1}^d\sum_x\E_\rho\left[\mathcal R_{1,3}^N(f,x,i)^2 \right]\\
				&=4N^{-2d}\sum_{\substack{ z\in \mathcal Z_L }}\sum_{i=1}^d\sum_{x\in \Lambda_L^z}\sum_{\substack{y:y\notin \Lambda_L^z\\y\neq x+e_i}}\E_{\rho}\left[\bar\eta^N(y)^2\right]\E_{\rho}\left[
				\left(\bar\eta^N(x)-\bar\eta^N(x+e_i)\right)^2\right]\\
				&\hspace{6cm}\left\{\nabla_{1,i}^{N}f\left(\frac{x}{N},\frac{y}{N}\right)-{\nabla_{1,i}^{N}f\left(\frac{z}{N},\frac{y}{N}\right)}\right\}^2\notag\\
				&\leq C_{f}N^{-2d} \left(\frac{N}{L}\right)^d L^d N^d \left(\frac{L}{N}\right)^2\notag\\
				&\leq C_{f}N^{-2} L^2\notag.
			\end{align*}
			For the fourth moment, we have
			\begin{align*}
				&\quad N^4\sum_{i=1}^d\sum_x\E_\rho\left[\mathcal R_{1,3}^N(f,x,i)^4\right]\\
				&=16N^{-4d}\sum_{\substack{ z\in \mathcal Z_L }}\sum_{i=1}^d\sum_{x\in \Lambda_L^z}\E_\rho\left[\left(\bar\eta^N(x)-\bar\eta^N(x+e_i)\right)^4\right]\\
				&\qquad\E_\rho\left[\left\{\sum_{\substack{y:y\notin \Lambda_L^z\\y\neq x+e_i}}\bar\eta^N(y)\left(\nabla_{1,i}^{N}f\left(\frac{x}{N},\frac{y}{N}\right)-{\nabla_{1,i}^{N}f\left(\frac{z}{N},\frac{y}{N}\right)}\right)\right\}^4\right]\\
				&\leq C_{f}N^{-4d} \left(\frac{N}{L}\right)^d L^d \left(N^{2d}+N^d\right) \left(\frac{L}{N}\right)^4\notag\\
				&\leq C_{f}N^{-4-d}L^4,
			\end{align*}
			where the second step follows from the observation that the expectation can be non-zero when four $y$'s pair each other.
		\end{proof}
        \begin{proof}[Proof of Lemma~\ref{lemma.R_1.expectation}]
            Combining Lemma~\ref{lemma.R_{1,1}.expectation}, Lemma~\ref{lemma.R_{1,2}.expectation}, Lemma~\ref{lemma.R_{1,3}.expectation} and Cauchy--Schwarz inequality, we get the bound for the second and fourth moment for $\mathcal R_1^N$ in Lemma~\ref{lemma.R_1.expectation} immediately.
        \end{proof}
		
		Next, we give the detailed calculation for the second moment of $\mathcal R_2^N$, which is used in the proof of Proposition~\ref{prop.limit.quadratic.variation} in Section~\ref{sec.M.limit}.
		
		\begin{lemma}\label{lemma.R_2.expectation}
			For test function $f$, we have
			\begin{equation*}
				\quad N^2\sum_{i=1}^d\sum_{x}\E_\rho\left[\mathcal R_{2}^N(f,x,i)^2\right]\leq C_fN^{-\gamma},
			\end{equation*}
			for some $\gamma>0$.
		\end{lemma}
        The remainder consists of four terms:
		\begin{equation*}
			\mathcal R_{2}^N(f,x,i)=-\mathcal R_{1,1}^N(f,x,i)+\mathcal R_{2,1}^N(f,x,i)+\mathcal R_{2,2}^N(f,x,i)+\mathcal R_{2,3}^N(f,x,i).
		\end{equation*}
		The three new remainder terms are as follows.
		\begin{align*}
			\mathcal R_{2,1}^N(f,x,i)&=2N^{-1-d}\left(\eta^N(x)-\eta^N(x+e_i)\right)\\
			&\hspace{2.9cm}\left(\bar\eta^N(x)\nabla^N_{1,i}f\left(\frac{x}{N},\frac{x}{N}\right)+\bar\eta^N(x+e_i)\nabla^N_{1,i}f\left(\frac{x}{N},\frac{x+e_i}{N}\right)\right),\notag\\
			\mathcal R_{2,2}^N(f,x,i)&=2N^{-1-d}\sum_{y\in \Lambda_L^{z(x)}\cup\{x+e_i\}}\bar\eta^N(y)\nabla^N_{1}f\left(\frac{x}{N},\frac{y}{N}\right)\cdot\pi_{x,x+e_i}\phi_L^{z(x)},\\
			\mathcal R_{2,3}^N(f,x,i)&=2N^{-1-d}\sum_{\substack{y:y\neq x+e_i\\y\notin \Lambda_L^{z(x)}}}\bar\eta^N(y)\left\{\nabla^N_{1}f\left(\frac{x}{N},\frac{y}{N}\right)-{\nabla^N_{1}f\left(\frac{z(x)}{N},\frac{y}{N}\right)}\right\}\cdot\pi_{x,x+e_i}\phi_L^{z(x)}.
		\end{align*}
		\begin{proof}
			From the definition, it is easy to see that for each $1\leq i\leq d$ and $x\in \Td_N$,
			\begin{equation}\label{eq.R_{2.1}.bound}
				\left|\mathcal R_{2,1}^N(f,x,i)\right|\leq C_{f}N^{-1-d},
			\end{equation}
			and
			\begin{equation}\label{eq.R_{2.2}.bound}
				\left|\mathcal R_{2,2}^N(f,x,i)\right|\leq C_{f}N^{-1-d}L^d\Vert\phi_L\Vert_\infty.
			\end{equation}
			We calculate for test function $f$,
			\begin{align}\label{eq.R_{2.3}.bound}
				&\quad N^2\sum_{i=1}^d\sum_{x}\E_\rho\left[\mathcal R_{3,3}^N(f,x,i)^2\right]\\
				&\leq dN^{-2d}\sum_{z\in \mathcal Z_L}\sum_{i,j=1}^d\sum_{x\in \Lambda_L^z}\sum_{\substack{y:y\notin \Lambda_L^{z}\\y\neq x+e_i}}\E_{\rho}\left[\bar\eta^N(y)^2\right]\E_\rho\left[\left(\pi_{x,x+e_i}\phi_{L,j}^{z}\right)^2\right]\notag\\
				&\hspace{4cm}\left\{\nabla^N_{1,j}f\left(\frac{x}{N},\frac{y}{N}\right)-{\nabla^N_{1,j}f\left(\frac{z}{N},\frac{y}{N}\right)}\right\}^2\notag\\
				&\leq C_{f}N^{-2d}N^{d}N^d\Vert\phi_L\Vert_\infty^2\left(\frac{L}{N}\right)^2\notag\\
				&\leq N^{-2}L^2\Vert\phi_L\Vert_\infty^2\notag.
			\end{align}
			Combining \eqref{eq.R_{2.1}.bound}-\eqref{eq.R_{2.3}.bound}, Lemma~\ref{lemma.R_{1,1}.expectation} and Cauchy-Schwarz inequality, we conclude the proof.
			
		\end{proof}
		
		Finally, we give two detailed calculations, which are used in the proof of Proposition~\ref{prop.replacement} in Section~\ref{sec.A.limit}.
		
		\begin{lemma}\label{lemma.minus.expectation}
			For test function $f$, we have
			\begin{equation*}
				N^2\sum_{i=1}^d\sum_x\E_\rho\left[\left(\mathcal I_{1}^N(f,x,i)-\mathcal I_3^N(f,x,i)\right)^2 \right]\leq C_{f}N^{-2} L^2.
			\end{equation*}
            and 
            \begin{equation*}
				N^2\sum_{i=1}^d\sum_x\E_\rho\left[\left(\mathcal J_2^N(f,x,i)-\mathcal J_1^N(f,x,i)\right)^2 \right]\leq C_{f}\left(N^{-2} L^2+N^{-d}L^d\right).
			\end{equation*}
		\end{lemma}
		\begin{proof}
			We make a direct calculation for the second moment of $\mathcal I_{1}^N(f,x,i)-\mathcal I_3^N(f,x,i)$:
			\begin{align*}
				&\quad N^2\sum_{i=1}^d\sum_x\E_\rho\left[\left(\mathcal I_{1}^N(f,x,i)-\mathcal I_3^N(f,x,i)\right)^2 \right]\\
				&=N^2\sum_{i=1}^d\sum_x\E_\rho\left[\left(2N^{-1-d}\sum_{y\in \Lambda_{2L}^{z(x)}\backslash \left(\Lambda_{L}^{z(x)}\cup\{x+e_i\}\right)}\bar\eta^N(y)
				\mathbf{v}_{x,i}\cdot{\nabla_{1}^{N}f\left(\frac{z(x)}{N},\frac{y}{N}\right)}\right)^2 \right]\\
				&=4N^{-2d}\sum_{\substack{ z\in \mathcal Z_L }}\sum_{i=1}^d\sum_{x\in \Lambda_L^z}\sum_{y\in \Lambda_{L+2\mathbf{r}}^{z}\backslash \left(\Lambda_{L}^{z}\cup\{x+e_i\}\right)}\E_{\rho}\left[\bar\eta^N(y)^2\right]\E_{\rho}\left[
				\left(\mathbf{v}_{x,i}\cdot{\nabla_{1}^{N}f\left(\frac{z}{N},\frac{y}{N}\right)}\right)^2\right]\\
				&\leq 4N^{-2d}\sum_{\substack{ z\in \mathcal Z_L }}\sum_{y\in \Lambda_{L+2\mathbf{r}}^{z}\backslash \left(\Lambda_{L}^{z}\right)}\E_{\rho}\left[\bar\eta^N(y)^2\right]\sum_{i=1}^d\sum_{x\in \Lambda_L^z}\E_{\rho}\left[
				\left(\mathbf{v}_{x,i}\cdot{\nabla_{1}^{N}f\left(\frac{z}{N},\frac{y}{N}\right)}\right)^2\right]\\
				&\leq C_fN^{-d}L^{d-1}.
			\end{align*}
			For term $\mathcal J_1^N(f,x,i)-\mathcal J_2^N(f,x,i)$, we have the following decomposition
			\begin{align*}
				\mathcal J_1^N(f,x,i)-\mathcal J_2^N(f,x,i)=R_1^N(f,x,i)+R_2^N(f,x,i),
			\end{align*}
			where
			\begin{align*}
				&R_1^N(f,x,i):=2N^{-1-d}\sum_{j=1}\sum_{\substack{y\in \Lambda_{2L}^{z(x)}\\y\neq x,x+e_i}}\mathbf D_{ij}(\rho)\left(\bar\eta^N(x)-\bar\eta^N(x+e_i)\right)\bar\eta^N(y)\nabla_{1,j}^{N}f\left(\frac{x}{N},\frac{y}{N}\right),\\
				&R_2^N(f,x,i):=2N^{-1-d}\sum_{j=1}\sum_{\substack{y:y\notin \Lambda_{2L}^{z(x)}}}\mathbf D_{ij}(\rho)\left(\bar\eta^N(x)-\bar\eta^N(x+e_i)\right)\bar\eta^N(y)\\
				&\hspace{7cm}\left\{\nabla_{1,j}^{N}f\left(\frac{x}{N},\frac{y}{N}\right)-{\nabla_{1,j}^{N}f\left(\frac{z(x)}{N},\frac{y}{N}\right)}\right\}.
			\end{align*}
			Therefore, we have
			\begin{align*}
				&\quad \E_\rho\left[R_1^N(f,x,i)^2 \right]\\
				&=4N^{-2-2d}\sum_{j=1}^d\sum_{\substack{y\in \Lambda_{2L}^{z(x)}\\y\neq x,x+e_j}}
				\mathbf{D}_{ij}^2(\rho)\E_{\rho}\left[\bar\eta^N(y)^2\right]\E_{\rho}\left[\left(\bar\eta^N(x)-\bar\eta^N(x+e_i)\right)^2\right]\nabla_{1,j}^{N}f\left(\frac{x}{N},\frac{y}{N}\right)^2\\
				&\leq C_fN^{-2-2d}L^d,
			\end{align*}
			and
			\begin{align*}
				&\quad \E_\rho\left[R_2^N(f,x,i)^2 \right]\\
				&\leq 4N^{-2-2d}\sum_{j=1}^d\sum_{\substack{y:y\notin \Lambda_{2L}^{z(x)}}}
				\mathbf{D}_{ij}^2(\rho)\E_{\rho}\left[\bar\eta^N(y)^2\right]\E_{\rho}\left[\left(\bar\eta^N(x)-\bar\eta^N(x+e_i)\right)^2\right]\\
				&\hspace{7.7cm}\left\{\nabla_{1,j}^{N}f\left(\frac{x}{N},\frac{y}{N}\right)-{\nabla_{1,j}^{N}f\left(\frac{z(x)}{N},\frac{y}{N}\right)}\right\}^2\\
				&\leq C_fN^{-4-d}L^2.
			\end{align*}
			Finally, we sum over $1\leq i\leq d$ and $x\in \Td_N$ to conclude.
		\end{proof}
		
	\section*{Acknowledgements}
	This research is supported by the National Key R\&D Program of China (No. 2023YFA1010400) and NSFC (No. 12301166). We thank Claudio Landim for the comments on the preliminary version of the manuscript.
	
	\bibliographystyle{plain}
	\bibliography{KawasakiRef}

@preamble{"\def\cprime{$'$} "}

@article{FGW24,
  author  = {Funaki, Tadahisa and Gu, Chenlin and Wang, Han},
  title   = {{Quantitative Homogenization and Hydrodynamic Limit of Nongradient Exclusion Process}},
  journal = {Communications on Pure and Applied Mathematics},
  year    = {2024},
  pages   = {e70034}
}

@article {GJ19,
    AUTHOR = {Goncalves, Patr\'icia and Jara, Milton},
     TITLE = {Quadratic fluctuations of the symmetric simple exclusion},
   JOURNAL = {ALEA Lat. Am. J. Probab. Math. Stat.},
  FJOURNAL = {ALEA. Latin American Journal of Probability and Mathematical
              Statistics},
    VOLUME = {16},
      YEAR = {2019},
    NUMBER = {1},
     PAGES = {605--632},
      ISSN = {1980-0436},
   MRCLASS = {60K35 (35K10 35R60)},
  MRNUMBER = {3949272},
MRREVIEWER = {Adrian\ Muntean},
       DOI = {10.30757/alea.v16-22},
       URL = {https://doi.org/10.30757/alea.v16-22},
}

@incollection{FUY96,
  title={Hydrodynamic limit for lattice gas reversible under Bernoulli measures},
  author={Funaki, Tadahisa and Uchiyama, Kohei and Yau, Horng-Tzer},
  booktitle={Nonlinear Stochastic PDEs: Hydrodynamic Limit and Burgers’ Turbulence},
  pages={1--40},
  year={1996},
  publisher={Springer}
}

@preamble{"\newcommand{\noop}[1]{} "}

@incollection {fun96,
	AUTHOR = {Funaki, T.},
	TITLE = {Equilibrium fluctuations for lattice gas},
	BOOKTITLE = {It\^{o}'s stochastic calculus and probability theory},
	PAGES = {63--72},
	PUBLISHER = {Springer, Tokyo},
	YEAR = {1996},
	MRCLASS = {60K35 (60F05 82C20 82C22)},
	MRNUMBER = {1439518},
	MRREVIEWER = {Timo Sepp\"{a}l\"{a}inen},
}

@book{kipnis1998scaling,
	AUTHOR = {Kipnis, Claude and Landim, Claudio},
	TITLE = {Scaling limits of interacting particle systems},
	SERIES = {Grundlehren der Mathematischen Wissenschaften},
	VOLUME = {320},
	PUBLISHER = {Springer-Verlag, Berlin},
	YEAR = {1999},
	PAGES = {xvi+442},
	ISBN = {3-540-64913-1},
	MRCLASS = {60-02 (60F05 60F10 60K35 82B05 82C22)},
	MRNUMBER = {1707314},
	MRREVIEWER = {Timo Sepp\"{a}l\"{a}inen},
	DOI = {10.1007/978-3-662-03752-2},
	URL = {https://doi.org/10.1007/978-3-662-03752-2},
}

@book{komorowski2012fluctuations,
	AUTHOR = {Komorowski, Tomasz and Landim, Claudio and Olla, Stefano},
	TITLE = {Fluctuations in {M}arkov processes},
	SERIES = {Grundlehren der Mathematischen Wissenschaften},
	VOLUME = {345},
	PUBLISHER = {Springer, Heidelberg},
	YEAR = {2012},
	PAGES = {xviii+491},
	ISBN = {978-3-642-29879-0},
	MRCLASS = {60J25 (60F05 60G44 60J10 60J60 60K35 60K37)},
	MRNUMBER = {2952852},
	MRREVIEWER = {B\'{a}lint T\'{o}th},
	DOI = {10.1007/978-3-642-29880-6},
	URL = {https://doi.org/10.1007/978-3-642-29880-6},
}

@article{gu2025relaxation,
  title={Relaxation to equilibrium of conservative dynamics II: non-gradient exclusion processes},
  author={Gu, Chenlin and Yang, Linzhi},
  journal={arXiv preprint arXiv:2509.20797},
  year={2025}
}

@article {Mitoma1983,
    AUTHOR = {Mitoma, Itaru},
     TITLE = {Tightness of probabilities on {$C([0,1];{\mathcal S}\sp{\prime}
              )$}\ and {$D([0,1];{\mathcal S}\sp{\prime} )$}},
   JOURNAL = {Ann. Probab.},
  FJOURNAL = {The Annals of Probability},
    VOLUME = {11},
      YEAR = {1983},
    NUMBER = {4},
     PAGES = {989--999},
      ISSN = {0091-1798,2168-894X},
   MRCLASS = {60B11 (60B12)},
  MRNUMBER = {714961},
MRREVIEWER = {Peter\ Z.\ Daffer},
       URL =
              {http://links.jstor.org/sici?sici=0091-1798(198311)11:4<989:TOPOA>2.0.CO;2-P&origin=MSN},
}

@book {Limit_theorems_for_stochastic_processes,
    AUTHOR = {Jacod, Jean and Shiryaev, Albert N.},
     TITLE = {Limit theorems for stochastic processes},
    SERIES = {Grundlehren der mathematischen Wissenschaften [Fundamental
              Principles of Mathematical Sciences]},
    VOLUME = {288},
   EDITION = {Second},
 PUBLISHER = {Springer-Verlag, Berlin},
      YEAR = {2003},
     PAGES = {xx+661},
      ISBN = {3-540-43932-3},
   MRCLASS = {60-02 (60F17 60G48 60H05)},
  MRNUMBER = {1943877},
MRREVIEWER = {Dominique\ L\'epingle},
       DOI = {10.1007/978-3-662-05265-5},
       URL = {https://doi.org/10.1007/978-3-662-05265-5},
}

@book {Convergence_of_probability_measures,
    AUTHOR = {Billingsley, Patrick},
     TITLE = {Convergence of probability measures},
    SERIES = {Wiley Series in Probability and Statistics: Probability and
              Statistics},
   EDITION = {Second},
      NOTE = {A Wiley-Interscience Publication},
 PUBLISHER = {John Wiley \& Sons, Inc., New York},
      YEAR = {1999},
     PAGES = {x+277},
      ISBN = {0-471-19745-9},
   MRCLASS = {60B10 (28A33 60F17)},
  MRNUMBER = {1700749},
       DOI = {10.1002/9780470316962},
       URL = {https://doi.org/10.1002/9780470316962},
}

@incollection{LyonsZheng1988,
  author    = {Lyons, Terry J. and Zheng, Weian},
  title     = {A Crossing Estimate for the Canonical Process on a Dirichlet Space and a Tightness Result},
  booktitle = {Colloque Paul L{\'e}vy sur les Processus Stochastiques},
  series    = {Ast{\'e}risque},
  volume    = {157--158},
  pages     = {249--271},
  year      = {1988},
  publisher = {Soci{\'e}t{\'e} Math{\'e}matique de France}
}
\end{document}